\documentclass{article}
\usepackage{graphicx} 
\usepackage{amsfonts,amsmath,amssymb,amsthm}
\usepackage{stmaryrd}
\usepackage{fullpage}
\usepackage[linesnumbered,algoruled,boxed]{algorithm2e}
\usepackage{comment}
\usepackage{hyperref}
\usepackage{nicefrac}

\newtheorem{theorem}{Theorem}[section]
\newtheorem{lemma}[theorem]{Lemma}
\newtheorem{proposition}[theorem]{Proposition}
\newtheorem{assumption}[theorem]{Assumption}
\newtheorem{remark}[theorem]{Remark}

\newcommand{\dx}{\,\mathrm{d}x}

\providecommand{\keywords}[1]{\par\vspace{1ex}\noindent\textbf{Keywords:} #1\par\vspace{1ex}}

\usepackage{pgfplots}
\usepgfplotslibrary{groupplots}
\pgfplotsset{compat=newest}
\usepgfplotslibrary{external}
\usepgfplotslibrary{colorbrewer}
\newcommand{\logLogSlopeTriangle}[6]
{

    \pgfplotsextra
    {
        \pgfkeysgetvalue{/pgfplots/xmin}{\xmin}
        \pgfkeysgetvalue{/pgfplots/xmax}{\xmax}
        \pgfkeysgetvalue{/pgfplots/ymin}{\ymin}
        \pgfkeysgetvalue{/pgfplots/ymax}{\ymax}

        \pgfmathsetmacro{\xArel}{#1}
        \pgfmathsetmacro{\yArel}{#3}
        \pgfmathsetmacro{\xBrel}{#1-#2}
        \pgfmathsetmacro{\yBrel}{\yArel}
        \pgfmathsetmacro{\xCrel}{\xArel}

        \pgfmathsetmacro{\lnxB}{\xmin*(1-(#1-#2))+\xmax*(#1-#2)} 
        \pgfmathsetmacro{\lnxA}{\xmin*(1-#1)+\xmax*#1} 
        \pgfmathsetmacro{\lnyA}{\ymin*(1-#3)+\ymax*#3} 
        \pgfmathsetmacro{\lnyC}{\lnyA+#4*(\lnxA-\lnxB)}
        \pgfmathsetmacro{\yCrel}{\lnyC-\ymin)/(\ymax-\ymin)} 

        \coordinate (A) at (rel axis cs:\xArel,\yArel);
        \coordinate (B) at (rel axis cs:\xBrel,\yCrel);
        \coordinate (C) at (rel axis cs:\xCrel,\yCrel);

        \draw[#5]   (A)--
                    (B)-- 
                    (C)-- node[pos=0.5,anchor=west] {#6}
                    cycle;
    }
}

\newcommand{\logLogSlopeTriangleBelow}[6]
{

    \pgfplotsextra
    {
        \pgfkeysgetvalue{/pgfplots/xmin}{\xmin}
        \pgfkeysgetvalue{/pgfplots/xmax}{\xmax}
        \pgfkeysgetvalue{/pgfplots/ymin}{\ymin}
        \pgfkeysgetvalue{/pgfplots/ymax}{\ymax}

        \pgfmathsetmacro{\xArel}{#1}
        \pgfmathsetmacro{\yArel}{#3}
        \pgfmathsetmacro{\xBrel}{#1-#2}
        \pgfmathsetmacro{\yBrel}{\yArel}
        \pgfmathsetmacro{\xCrel}{\xArel}

        \pgfmathsetmacro{\lnxB}{\xmin*(1-(#1-#2))+\xmax*(#1-#2)} 
        \pgfmathsetmacro{\lnxA}{\xmin*(1-#1)+\xmax*#1} 
        \pgfmathsetmacro{\lnyA}{\ymin*(1-#3)+\ymax*#3} 
        \pgfmathsetmacro{\lnyC}{\lnyA+#4*(\lnxA-\lnxB)}
        \pgfmathsetmacro{\yCrel}{\lnyC-\ymin)/(\ymax-\ymin)} 

        \coordinate (A) at (rel axis cs:\xArel,\yArel);
        \coordinate (B) at (rel axis cs:\xBrel,\yCrel);
        \coordinate (C) at (rel axis cs:\xBrel,\yArel);

        \draw[#5]   (A)--
                    (B)-- node[pos=0.5,anchor=east] {#6}
                    (C)-- 
                    cycle;
    }
}
\title{Adaptive finite element method for an unregularized semilinear optimal control problem
\thanks{This project has received funding by the Federal Ministry of Education and Research (BMBF) and the Baden-Württemberg Ministry of Science as part of the Excellence Strategy of the German Federal and State Governments. In addition, the first author has been supported by ANID through FONDECYT postdoctoral project 3230126.}}
\author{Francisco Fuica
\thanks{Facultad de Matem\'aticas, Pontificia Universidad Cat\'olica de Chile, Avenida Vicu\~{n}a Mackenna 4860, Santiago, Chile.}
\thanks{Zukunftskolleg, Konstanz University, Universit\"atsstra{\ss}e 10, 78464 Konstanz, Germany. email \texttt{francisco.fuica@uc.cl}}
\and{Nicolai Jork\thanks{Department of Mathematics, Eberhard-Karls-Universit\"at T\"ubingen, D-72076 T\"ubingen, Germany. email \texttt{nicolai.jork@uni-tuebingen.de}}
}}

\date{\today}
\begin{document}

\maketitle

\begin{abstract}
We devise an a posteriori error estimator for an affine optimal control problem subject to a semilinear elliptic PDE and control constraints.
To approximate the problem, we consider a semidiscrete scheme based on the variational discretization approach.
For this solution technique, we design an a posteriori error estimator that accounts for the discretization of the state and adjoint equations, and prove, under suitable local growth conditions of optimal controls, reliability and efficiency properties of such error estimator. 
A simple adaptive strategy based on the devised estimator is designed and its performance is illustrated with numerical examples.
\end{abstract}

%


\section{Introduction}
This work aims to design and analyze a posteriori error estimates for an affine optimal control problem governed by a semilinear elliptic partial differential equation (PDE); bilateral control constraints are also considered.
The affine structure of the objective functional produces several difficulties, which are further increased by the nonlinearity of the constraining equation.
To be precise, in contrast to the setting where the objective functional includes the classical Tikhonov regularization term $\nicefrac{\lambda}{2}\|u\|_{L^{2}(\Omega)}^{2}$ with $\lambda > 0$, the affine structure does not incorporate such a term, making the analysis of second-order sufficient optimality conditions more involved.
We note that the available formulations of these conditions in the literature (see, e.g., \cite{CDJ2023,MR4791221}) are considerably weaker compared to the regularized case.
As a result, standard techniques, which rely on strong second-order sufficient optimality conditions, used to derive a posteriori error estimates for optimal control problems with nonlinear constraints cannot be directly applied.

The study of affine optimal control problems has its origin in optimal control problems subject to ordinary differential equations (ODEs). In this context, several works have addressed the stability of these problems, employing tools from Variational Analysis.
For some works in which the stability and the error estimation of the numerical approximation played a major role, we refer to \cite{MR3772055, MR4352916, MR1972537, MR3484461, MR4144113, MR3956959}. 
More recently, the study of affine optimal control problems with PDE constraints has gained considerable attention. 
In particular, the study of bang-bang optimal controls in the context of elliptic equations can be found in \cite{zbMATH06111099, zbMATH06797168, zbMATH06982426, CDJ2023, MR4525177, zbMATH07865499} and the references therein.
For results on numerical approximation and error analysis of these problems, particularly without the use of regularization strategies, we refer the reader to \cite{MR2891922, MR4791221, Fuica2024, CM2021}.

A widely used method for approximating solutions to optimal control problems--and PDEs in general--is the class of adaptive finite element methods (AFEMs). 
These are iterative methods that assess the discretization error after computation and guide adaptive mesh refinement by identifying regions where the solution is more difficult to approximate. 
This strategy allows maintaining an optimal distribution of computational resources, typically measured in terms of degrees of freedom. 
At the core of these methods we find a posteriori error estimators, which provide both global and local information on the error of discrete solutions and that can be easily computed from the numerical solution and given problem data.
A key component of a posteriori error analysis is establishing the \emph{reliability} and \emph{efficiency} of such estimators. 
Reliability ensures that the estimator does not underestimate the true error, while efficiency guarantees that the estimator is proportional to the actual error, avoiding significant overestimation \cite{MR3059294}. 
Together, these properties ensure that the estimator offers a meaningful measure of the accuracy of the numerical solution. 
The use of these methods in the context of affine optimal control problems is rather scarce; there exist only a few works on this matter \cite{Fuica2024,MR3095657,MR3385653}, all focusing on affine optimal control problems governed by the Poisson equation.
In \cite{MR3095657}, the author investigated a simultaneous Tikhonov regularization and discretization of bang-bang control problems, developing a parameter choice rule that adaptively selects the Tikhonov regularization parameter based on a posteriori computable quantities. 
However, the error estimates were not robust with respect to $\lambda > 0$.
This was improved in \cite{MR3385653}, in which robust global reliability estimates were provided.
We note that efficiency estimates were not provided in \cite{MR3095657,MR3385653} for the corresponding error estimators.
We conclude this paragraph by mentioning the very recent work \cite{Fuica2024}, in which reliable and efficient a posteriori error estimates were proved for a tracking-type optimal control problem without regularization terms. 

To the best of our knowledge, this is the first work to study a posteriori error estimates for an unregularized affine optimal control problem involving a semilinear elliptic PDE. 
To approximate optimal variables, we consider the variational discretization approach introduced in \cite{MR2122182}.
More precisely, we discretize optimal state and adjoint state variables using standard continuous piecewise linear functions, while the control variable is not discretized.
Since we are concerned with unregularized problems, we do not consider a Tikhonov regularization.
This allows us to avoid dealing with discretization and regularization errors simultaneously, eliminating the need to choose a suitable regularization parameter for each mesh; cf. \cite{MR3385653}.
However, this also means that we cannot use the standard approach revolving around the second-order sufficient optimality conditions as strong as in the regularized situation.
Within such an unregularized framework, we develop a residual-based a posteriori error estimator consisting of only two contributions, corresponding to the discretization errors in the state and adjoint equations. 
We shall assume the growth condition \eqref{eq:assumption_control_growth} and that the approximation error $\|\bar{u} - \bar{\mathfrak{u}}_{\ell}\|_{L^{1}(\Omega)}$ is smaller than an appropriate constant; here, $\bar{u}$ denotes a locally optimal control and $\bar{\mathfrak{u}}_{\ell}$ denotes a suitable approximation of $\bar{u}$.
Under these assumptions, we prove reliability and efficiency properties for the proposed error estimator in Theorems \ref{thm:global_rel_ocp} and \ref{thm:efficiency_est_ocp}, respectively, in two- and three-dimensional convex polygonal/polyhedral domains.
The analysis primarily relies on assumption \eqref{eq:assumption_control_growth}, a weaker form of optimality conditions designed for optimal control problems where the controls are expected to exhibit a bang-bang structure, which is typical of unregularized affine problems.

To precise the model under investigation, we let $\Omega\subset\mathbb{R}^d$ ($d\in \{2, 3\}$) be an open, bounded, and convex polygonal/polyhedral domain with boundary $\partial\Omega$ and let $y_\Omega\in L^2(\Omega)$ be a given desired state.
We thus define the tracking-type cost functional
\begin{align*}
J(u):=\frac{1}{2}\|y_{u} - y_{\Omega}\|_{\Omega}^{2},
\end{align*} 
and consider the following optimal control problem:
\begin{equation}\label{def:opt_cont_prob}
\min J(u) ~ \text{ subject to } ~ u\in U_{ad} ~ \text{ and } ~ -\Delta y_{u} + f(\cdot,y_u)  =  u ~  \text{ in }  \Omega, \quad y_u  =  0 ~ \text{ on }  \partial\Omega,
\end{equation}
where the set of \emph{admissible} controls $U_{ad}$ is given by
\[
U_{ad}:=\{v \in L^2(\Omega):  a \leq v(x) \leq b ~ \text{ for almost all (f.a.a.) } x \in \Omega \},
\]
with $a,b\in \mathbb{R}$ being such that $a<b$.
Assumptions on the nonlinear term $f$ will be deferred until Section \ref{sec:not_and_prel}.
We note that the PDE-constrained optimization problem \eqref{def:opt_cont_prob} entails minimizing a cost functional in which the cost of the control is negligible.

We organize the remainder of the manuscript as follows. 
Section \ref{sec:not_and_prel} establishes the notation and preliminary results related to the PDEs involved in our problem. 
In Section \ref{sec:OCP}, we introduce a weak formulation for the problem \eqref{def:opt_cont_prob} and present first- and second-order optimality conditions. 
The core of our work is Section \ref{sec:a_post_analysis}, in which we propose a semidiscrete approximation scheme, design for it a residual-type posteriori error estimator, and prove reliability and efficiency properties of such estimator.
We conclude with Section \ref{sec:num_ex}, where we display a number of numerical examples that illustrate the theory and exhibit a competitive performance of the devised AFEM.


\section{Notation and preliminaries}\label{sec:not_and_prel}


\subsection{Notation}
Throughout this work, we use standard notation for Lebesgue and Sobolev spaces and their corresponding norms. 
In particular, the space $H_0^{1}(\Omega)$ denotes the closure of $C_0^{\infty}(\Omega)$ under $\|\nabla \cdot\|_{L^{2}(\Omega)}$.
The dual space of $H_0^{1}(\Omega)$ is denoted by $H^{-1}(\Omega)$.
Given an open and bounded domain $\omega$, we denote by $(\cdot,\cdot)_{\omega}$ and $\| \cdot \|_{\omega}$ the inner product and norm of $L^{2}(\omega)$, respectively.
If $\mathcal{X}$ and $\mathcal{Y}$ are Banach function spaces, we write $\mathcal{X}\hookrightarrow \mathcal{Y}$ to denote that $\mathcal{X}$ is continuously embedded in $\mathcal{Y}$. 
The relation $\mathfrak{a} \lesssim \mathfrak{b}$ indicates that $\mathfrak{a} \leq C \mathfrak{b}$, with a positive constant that depends neither on $\mathfrak{a}$, $\mathfrak{b}$ nor on the discretization parameters. 
The value of the constant $C$ may vary from one occurrence to another.


\subsection{Results on the involved PDEs}

For completeness, in this section, we collect properties of solutions to linear and semilinear elliptic PDEs. 
We notice that in \cite{CDJ2023}, the results are obtained for a non-monotone and non-coercive semilinear elliptic PDE. 
Since the PDE considered in this work can be seen as a special case, the results in \cite{CDJ2023} apply.


\subsubsection{Linear elliptic PDE}
Let $a_0 \in L^\infty(\Omega)$ be a nonnegative function and let $z\in L^{2}(\Omega)$. We consider the equation
\begin{equation*}
    -\Delta \varphi_{z} +a_0 \varphi_{z} = z \text{ in }\Omega, \quad \varphi_z = 0 \text{ on } \partial\Omega.
\end{equation*}
Its variational formulation is given by: Find $\varphi_{z}\in H_0^{1}(\Omega)$ such that
\begin{align}\label{eq:weak_st_lin}
(\nabla \varphi_{z},\nabla v)_{\Omega} + (a_0\varphi_z,v)_{\Omega} = (z,v)_{\Omega} \quad \forall v \in H_0^1(\Omega).
\end{align}

We now present some properties for $\varphi_{z}$. 
We start with the following result; see, e.g., \cite[Lemma 2.2]{CDJ2023} and \cite[Theorems 3.2.1.2]{MR775683}.

\begin{lemma}[well-posedness and stability]\label{thm:wp_and_stab_linear}
    Let $z\in L^r(\Omega)$ with $r>d/2$. 
    Then, the linear equation \eqref{eq:weak_st_lin} has a unique solution $\varphi_{z}\in H^1_0(\Omega)\cap C(\bar \Omega)$.
    In addition, it holds that
\begin{equation}\label{solubound}
\|\varphi_z\|_{H^1_0(\Omega)}+\|\varphi_z\|_{L^{\infty}(\Omega)}\leq C_r\| z\|_{L^r(\Omega)},
\end{equation}
where $C_r>0$ is independent of $a_0$ and $z$. 
Moreover, if $\Omega$ is convex, then $\varphi_z\in H^1_0(\Omega)\cap H^2(\Omega)$ and
\begin{equation*}
\|\varphi_z\|_{H^2(\Omega)}
\lesssim
\|z\|_{\Omega}.
\end{equation*}

\label{estlincont}
\end{lemma}

For the proof of the next result, we refer the reader to \cite[Lemma 2.3]{CDJ2023}.

\begin{lemma}[stability in $L^{s}(\Omega)$]\label{lemma:stab_Ls}
Let $a_0 \in L^\infty(\Omega)$ be a nonnegative function.
Assume that $s \in [1,\frac{d}{d - 2})$ and that $s'$ is its conjugate. 
Then, there exists a positive constant $C_{s'}$ independent of $a_0$ such that
\begin{equation*}
\|\varphi_{z}\|_{L^s(\Omega)} \le C_{s'}\|z\|_{L^1(\Omega)} \qquad \forall z \in H^{-1}(\Omega) \cap L^1(\Omega),
\end{equation*}
where $C_{s'}$ is given by \eqref{solubound} with $r = s'$.
\label{L2.2}
\end{lemma}


\subsubsection{Semilinear elliptic PDE}
We begin by introducing a variational formulation for the semilinear elliptic state equation in \eqref{def:opt_cont_prob}: Given $u\in L^{2}(\Omega)$ find $y_{u}\in H_0^{1}(\Omega)$ such that
\begin{align}\label{eq:weak_st_eq2}
(\nabla y_{u},\nabla v)_{\Omega} + (f(\cdot,y),v)_{\Omega} = (u,v)_{\Omega} \quad \forall v \in H_0^1(\Omega).
\end{align}
We assume that the nonlinear function $f$ satisfies the following assumptions:
\begin{assumption}\label{A1}
We assume that $f:\Omega \times \mathbb{R}\longrightarrow \mathbb{R}$ is a Carath\'eodory function of class $C^2$ with respect to the second variable satisfying:
\begin{align*}
\left\{\begin{array}{l}
f(\cdot,0)\in L^\infty(\Omega) \text{ and } \frac{\partial f}{\partial y}(x,y) \geq 0 \  \forall y \in \mathbb{R},
\vspace{1mm}\\
\displaystyle \forall M>0\ \exists C_{f,M}>0 \text{ s.t. }\left\vert\frac{\partial f}{\partial y}(x,y)\right\vert+
\left\vert\frac{\partial^2 f}{\partial y^2}(x,y)\right\vert \leq C_{f,M} \ \forall \vert y\vert \leq M,
\vspace{1mm}\\
\displaystyle \forall \rho > 0 \text{ and } \forall M>0\ \exists\ \epsilon > 0 \text{ such that } \\\displaystyle
\left\vert\frac{\partial^2f}{\partial y^2}(x,y_2) - \frac{\partial^2f}{\partial y^2}(x,y_1)\right\vert < \rho \ \forall \vert y_1\vert, \vert y_2\vert \leq M \mbox{ with }  \vert y_2 - y_1\vert \le \epsilon, \end{array}\right.
\end{align*}
for almost every $x \in \Omega$.
\end{assumption}

The following theorem guarantees existence and uniqueness of a weak solution for \eqref{eq:weak_st_eq2}, and regularity properties for such a solution; see, e.g., \cite[Sections 2.1 and 3.1]{MR2009948}.

\begin{theorem}[well-posedness and regularity]\label{contandregsem}
Let Assumption \ref{A1} hold. For every $u \in L^r(\Omega)$ with $r > d/2$ there exists a unique $y_u \in
H^1_0 (\Omega)\cap C(\bar  \Omega)$ solution of \eqref{eq:weak_st_eq2}. Moreover, there exists a constant $T_r > 0$ independent of $u$ such that
\begin{equation*}
    \| y_u\|_{H^1_0(\Omega)}+\|y_u\|_{L^{\infty}(\Omega)}\leq T_r(\|u\|_{L^r(\Omega)}+\|f(\cdot,0)\|_{L^\infty(\Omega)}).
\end{equation*}
If $u_k\rightharpoonup u$ weakly in $L^r(\Omega)$, then we have the strong convergence 
\begin{equation*}
    \| y_{u_k}-y_u\|_{H^1_0(\Omega)}+\|y_{u_k}-y_u\|_{L^{\infty}(\Omega)}\to 0.
\end{equation*}
\end{theorem}
For each $r>d/2$, we define the map $G_r:L^r(\Omega)\to H_0^1(\Omega)\cap C(\bar\Omega)$ by $G_r(u)=y_u$.
\begin{theorem}[properties of $G_r$]\label{T2.2}
		Let Assumption \ref{A1} hold. 
        For every $r > d/2$ the map $G_r$ is of class $C^2$.
        In addition, the first derivative at $u\in L^r(\Omega)$ in the direction $z\in L^r(\Omega)$, denoted by $\varphi_{u,z} = G'_r(u)z$, corresponds to the unique solution to
    \begin{align*}
     -\Delta \varphi_{u,z} +\frac{\partial f}{\partial y}(x,y_u)\varphi_{u,z} = z \textnormal{ in }\Omega,\quad \varphi_{u,z}=0 \textnormal{ on } \partial\Omega.
    \end{align*}    
    The second derivative at $u\in L^r(\Omega)$ in the directions $z_1, z_2\in L^r(\Omega)$, denoted by $\varphi_{u,z_1,z_2} = G''_r(u)(z_1,z_2)$, corresponds to the unique solution to
\begin{align*}
-\Delta \varphi_{u,z_1,z_2} +\frac{\partial f}{\partial y}(x,y_u)\varphi_{u,z_1,z_2} = -\frac{\partial^2f}{\partial y^2}(x,y_u)\varphi_{u,z_1}\varphi_{u,z_2} \textnormal{ in }\Omega, \quad \varphi_{u,z_1,z_2}=0 \textnormal{ on } \partial\Omega.
\end{align*}
\end{theorem}
\begin{proof}
    For a proof, we refer the reader to \cite[Theorem 2.6]{CDJ2023} (see also \cite[Theorem 2.11]{MR4162938}).
\end{proof}


\section{The optimal control problem}\label{sec:OCP}
\label{sec:ocp}

This section introduces a weak formulation for problem \eqref{def:opt_cont_prob} and recalls first- and second-order optimality conditions. 
We also introduce a finite element discretization scheme.


\subsection{Weak formulation}

We consider the following weak version of problem \eqref{def:opt_cont_prob}: Find
\begin{align}\label{def:weak_ocp}
\min \{ J(u): u \in U_{ad}\}
\end{align}
subject to 
\begin{align}\label{eq:weak_st_eq}
y_{u}\in H_0^{1}(\Omega) ~:~ (\nabla y_{u},\nabla v)_{\Omega} + (f(\cdot,y),v)_{\Omega} = (u,v)_{\Omega} \quad \forall v \in H_0^1(\Omega).
\end{align}
As a consequence of the direct method of the calculus of variations, problem \eqref{def:weak_ocp}--\eqref{eq:weak_st_eq} admits at least one global optimal solution $\bar{u}\in U_{ad}$ \cite[Theorem 4.15]{MR2583281}. 
Since the problem is not convex, we may also have local minimizers: Given $r>d/2$, we say that $\bar u \in U_{ad}$ is an $L^r (\Omega)$-weak local minimum of problem \eqref{def:weak_ocp}--\eqref{eq:weak_st_eq}, if there exists $\varepsilon > 0$ such that
\[
J( \bar u) \le J(u) \quad \forall u \in U_{ad} \text{ with } \|u-\bar u\|_{L^r(\Omega)}\le \varepsilon.
\]
We say that $\bar u \in U_{ad}$ is a strong local minimum of \eqref{def:weak_ocp}--\eqref{eq:weak_st_eq}, if there exists $\varepsilon >0$ such that
\[
J( \bar u) \le J(u) \quad \forall u \in U_{ad} \text{ with } \|y_u-y_{\bar u}\|_{C(\bar \Omega)}\le \varepsilon.
\]
We say that $\bar u \in U_{ad}$ is a strict weak (resp. strong) local minimum if the above inequalities are
strict for $u\ne \bar u $.
Strong local minimizers were first considered in \cite[Definition 1.6]{BBS2014}. For more details on these notions of optimality, we refer to \cite[Lemma 2.8]{CM2020}.


\subsection{The optimality conditions}
We present the first and second variations of the cost functional to discuss necessary and sufficient optimality conditions. 
The calculations in the following are standard and we refer to \cite[Sections 4.6 and 4.10]{MR2583281}.

\begin{proposition}[properties of $J$]\label{T3.1}
Let Assumption \ref{A1} hold.
For every $r > \frac{d}{2}$, the functional $J:L^r(\Omega) \longrightarrow \mathbb{R}$ is of class $C^2$. 
Moreover, given $u, z, z_1, z_2 \in L^r(\Omega)$ we have
\begin{align*}
     J'(u)z = (y_u-y_\Omega,\varphi_{u,z})_{\Omega}= (p_u, z)_{\Omega}, 
\end{align*}
\begin{align*}
   J''(u)(z_1,z_2) = \left(1 - p_u\tfrac{\partial^2f}{\partial y^2}(x,y_u),
       \varphi_{u,z_1}\varphi_{u,z_2}\right)_{\Omega}.
\end{align*}
Here, $p_u \in H^1_0(\Omega) \cap C(\bar \Omega)$ is the unique solution to the adjoint equation
\begin{equation*}
      -\Delta p_u + \frac{\partial f}{\partial y}(x,y_u)p_u  =  y_u-y_\Omega \textnormal{ in } \Omega, \quad p_u = 0 \textnormal{ on } \partial\Omega.
\end{equation*}
\end{proposition}

We define the Hamiltonian 
$\Omega\times\mathbb R\times \mathbb R\times \mathbb R \ni (x,y,\varphi,u) \mapsto H(x,y,\varphi,u) \in \mathbb R$
by
\begin{align*}
       H(x,y,\varphi,u):=\nicefrac{1}{2}(y_u-y_d)^2+\varphi(u-f(x,y)).
\end{align*}
We now establish the following local form of the Pontryagin-type necessary optimality conditions for problem \eqref{def:weak_ocp}-\eqref{eq:weak_st_eq} (see e.g. \cite[Theorem 4]{CM2021}, \cite[Theorem 1.1]{MR2891922} or \cite[Section 4.8]{MR2583281}).
 
\begin{theorem}[first-order optimality conditions]
\label{pontryagin}
Let Assumption \ref{A1} hold.
If $\bar u$ is a weak or strong local minimizer for problem \eqref{def:weak_ocp}-\eqref{eq:weak_st_eq}, 
then there exist unique elements $\bar y, \bar p\in H^1_0(\Omega)\cap C(\bar \Omega)$ such that 
\begin{align}
&-\Delta\bar y + f(x,\bar y) = \bar u \textnormal{ in } \Omega,\quad \bar y = 0\textnormal{ on } \partial\Omega,
\nonumber\\
&-\Delta\bar p = \frac{\partial H}{\partial y}(\cdot,\bar y,\bar p,\bar u) \textnormal{ in } \Omega,\quad \bar p = 0\textnormal{ on } \partial\Omega,
\nonumber\\
& J'(\bar{u})(u - \bar{u}) = (\bar p, u - \bar u)_{\Omega}\dx \ge 0 \quad \forall u \in U_{ad}. \label{eq:var_ineq} 
\end{align}
\end{theorem}

Inequality \eqref{eq:var_ineq} implies that, f.a.a. $x\in\Omega$, we have 
\begin{align*}
\bar{u}(x) = a ~ \text{ if } ~  \bar{p}(x) > 0, \quad \bar{u}(x) \in [a,b] ~ \text{ if } ~  \bar{p}(x) = 0, \quad \bar{u}(x) = b  ~ \text{ if } ~  \bar{p}(x) < 0;
\end{align*}
see \cite[Remark 1.2]{MR2891922}.

For the upcoming analysis, we assume, given $\gamma\in ((2/(2+d),1]$, that there exists positive constants $\kappa$ and $\alpha$ such that 
\begin{align}\label{eq:assumption_control_growth}
J'(\bar{u})(u - \bar{u}) + J''(\bar{u})(u - \bar{u})^2 \geq \kappa \|u - \bar{u}\|_{L^{1}(\Omega)}^{1+\frac{1}{\gamma}} \qquad \forall u\in U_{ad} \text{ with }\|u - \bar{u}\|_{L^1(\Omega)} < \alpha.
\end{align}

The next estimate follows from \cite[Lemma 3.9]{MR4791221} (see also \cite[Section 3]{MR4525177}).
Let $\bar{u},u \in U_{ad}$ and define $u_{\theta}:= \bar{u} + \theta(u - \bar{u})$ for some measurable function $\theta$ with $0 \leq \theta(x) \leq 1$. Then, for every $\epsilon > 0$ there exists $\delta_{\varepsilon}$ such that
\begin{align}\label{eq:sec_var_est}
    |[J''(u_{\theta})- J''(\bar{u})](u - \bar{u})^2|
    \leq 
    \varepsilon\|u - \bar{u}\|_{L^1(\Omega)}^{1+\frac{1}{\gamma}} \qquad \forall u\in U_{ad} \text{ with } \|u - \bar{u}\|_{L^{1}(\Omega)} < \delta_{\varepsilon}.
\end{align}
For error analysis it will be important the case $\varepsilon = \frac{\kappa}{2}$ with $\kappa$ given as in \eqref{eq:assumption_control_growth} (cf. Theorem \ref{thm:global_rel_ocp}). 
Due to \eqref{eq:sec_var_est} assumption \eqref{eq:assumption_control_growth} leads to a strict local minimum; see \cite[Theorem 3.11]{MR4791221}.
\begin{theorem}[local optimality]\label{thm:eqhalf}
Let $\bar u\in U_{ad}$ be given and let \eqref{eq:assumption_control_growth} be satisfied together with the first-order necessary optimality condition \eqref{eq:var_ineq}.
Then there exist positive constants $\kappa$ and $\alpha$ such that
    \begin{equation*}
    J(u)-J(\bar u)\geq \kappa\|u-\bar u\|_{L^1(\Omega)}^{1+\frac{1}{\gamma}},
    \end{equation*}
    for all $u\in U_{ad}$ with $\|u-\bar u\|_{L^1(\Omega)}<\alpha$.
\end{theorem}


\subsection{Some general remarks}

\begin{remark}[more general problem]

Extending the analysis to a more general affine problem is possible.
For example, we could consider the case in which $J(u):=\int_{\Omega} L(x,y(x),u(x))\textnormal{d}x$, with a function $L$ satisfying the following properties:
The function $L:\Omega \times \mathbb{R}^2 \longrightarrow \mathbb{R}$ is Carath\'eodory and of class $C^2$ with respect to the second variable. 
Moreover, $L(x,y,u) = L_a(x,y) + L_b(x,y)u$, with $L_a(\cdot,0), L_b(\cdot,0) \in L^1(\Omega)$ and
\begin{align*}
\left\{\begin{array}{l} 
\displaystyle \forall M > 0 \ \exists  C_{L,M} > 0 \text{ such that }\vspace{1mm}\\
\displaystyle \Big\vert\frac{\partial L}{\partial y}(x,y,u)\Big\vert + \Big\vert\frac{\partial^2L}{\partial y^2}(x,y,u)\Big\vert \le C_{L,M}\ \forall \vert u\vert \le M, \, \forall y\in \mathbb{R}, \vspace{1mm} \\
\displaystyle \forall \rho > 0 \text{ and } M > 0 \ \vspace{1mm} \exists \epsilon>0 \text{ such that }\\
\displaystyle\left\vert\frac{\partial^2L}{\partial y^2}(x,y_2,u) - \frac{\partial^2L}{\partial y^2}(x,y_1,u)\right\vert < \rho\ \vert y_1\vert, \vert y_2\vert \leq M \mbox{ with }  \vert y_2 - y_1\vert \le \epsilon, \end{array}\right.
\end{align*}
for almost every $x \in \Omega$.
 We omitted this case in the paper to facilitate the presentation of the main results. 
\end{remark}

\begin{remark}[larger range for $\gamma$]
The growth condition established in Theorem \eqref{thm:eqhalf} can be extended to $\gamma\in (0,1]$ by assuming \eqref{eq:assumption_control_growth} for $ \gamma \in (0,1]$ together with the second-order sufficient optimality condition commonly used in control constraint problems. 
For the detailed arguments, we refer to \cite[Remark 3.12]{MR4791221}.
\end{remark}


\subsection{Finite element approximation}\label{sec:FEM}

We start by presenting standard ingredients for finite element approximations \cite{MR2373954,MR2050138}. 
Let $\mathcal{T} = \{T\}$ be a conforming partition of $\bar{\Omega}$ into simplices $T$ with size $h_T := \textrm{diam}(T)$. 
We denote by $\mathbb{T}$ a collection of conforming and shape regular meshes that are refinements of an initial mesh $\mathcal{T}_0$.
Given a mesh $\mathcal{T}_{\ell} \in \mathbb{T}$ with $\ell\in\mathbb{N}_{0}$, we denote by $\mathcal{E}_{\ell}$ the set of \emph{internal} $(d-1)$-dimensional interelement boundaries $e$ of $\mathcal{T}_{\ell}$. 
For $T \in \mathcal{T}_{\ell}$, we let $\mathcal{E}_T$ denote the subset of $\mathcal{E}_{\ell}$ which contains the sides of the element $T$. 
We denote by $\mathcal{N}_e \subset \mathcal{T}_{\ell}$ the subset that contains the two elements that have $e\in \mathcal{E}_{\ell}$ as a side, i.e., $\mathcal{N}_e=\{T^+,T^-\}$, where $T^+, T^- \in \mathcal{T}_{\ell}$ are such that  $T^{+} \neq T^{-}$ and $e = T^+ \cap T^-$. 
Given $T \in \mathcal{T}_{\ell}$, we define the \emph{star} associated with the $T$ as
\begin{equation}\label{def:patch}
\mathcal{N}_T:= \left \{ T^{\prime}\in\mathcal{T}_{\ell}: \mathcal{E}_{T}\cap \mathcal{E}_{T^\prime}\neq\emptyset \right \}.
\end{equation}
In an abuse of notation, below we denote by $\mathcal{N}_T$ either the set itself or the union of its elements. 

Given a mesh $\mathcal{T}_{\ell} \in \mathbb{T}$ with $\ell\in\mathbb{N}_{0}$, we define the finite element space of continuous piecewise polynomials of degree one that vanish on the boundary as
\begin{align*}
\mathbb{V}_{\ell} := \{v_{\ell}\in C(\bar{\Omega}): v_{\ell}|_{T}\in \mathbb{P}_1(T) ~ \forall T\in \mathcal{T}_{\ell} \} \cap H_0^{1}(\Omega).
\end{align*}

Given $v_{\ell} \in \mathbb{V}_{\ell}$ and $e \in \mathcal{E}_{\ell}$, we define the \emph{jump} or interelement residual on $e$ by
\[
\llbracket \nabla v_{\ell}\cdot \mathbf{n} \rrbracket|_{e} = \llbracket \nabla v_{\ell}\cdot \mathbf{n} \rrbracket := \mathbf{n}^{+} \cdot \nabla v_{\ell}|^{}_{T^{+}} + \mathbf{n}^{-} \cdot \nabla v_{\ell}|^{}_{T^{-}},
\]
where $\mathbf{n}^{\pm}$ denote the unit exterior normal to the element $T^{\pm}$. Here, $T^{+}$, $T^{-} \in \mathcal{N}_{e}$.


\subsubsection{Semidiscrete scheme}
We consider the following semidiscrete version of the optimal control problem \eqref{def:weak_ocp}--\eqref{eq:weak_st_eq}: Find 
\begin{align}\label{eq:discrete_opc_semi}
\min \{ J_{\ell}(\mathfrak{u}): \mathfrak{u} \in U_{ad}\}
\end{align}
subject to the discrete state equation
\begin{align}\label{eq:discrete_pde_semi}
(\nabla y_{\ell}(\mathfrak{u}),\nabla v_\ell)_{\Omega} + (f(\cdot,y_{\ell}(\mathfrak{u})),v_{\ell})_{\Omega}
=
(\mathfrak{u}, v_\ell)_{\Omega} 
\qquad \forall v_\ell\in \mathbb{V}_{\ell}.
\end{align}
Here, $J_{\ell}(\mathfrak{u}) = \nicefrac{1}{2}\|y_{\ell}(\mathfrak{u}) - y_{\Omega}\|_{\Omega}^{2}$.
Problem \eqref{eq:discrete_opc_semi}--\eqref{eq:discrete_pde_semi} admits at least one optimal control $\bar{\mathfrak{u}}$ which can be characterized as follows \cite[Section 4.1]{MR4791221}:
\begin{align}\label{eq:semi_var_ineq}
J_{\ell}^{\prime}(\bar{\mathfrak{u}})(u - \bar{\mathfrak{u}}) = (\bar{p}_{\ell}, u - \bar{\mathfrak{u}})_{\Omega} \geq 0 \quad \forall u \in U_{ad},
\end{align}
where $\bar{p}_{\ell}\in \mathbb{V}_{\ell}$ solves the discrete adjoint state equation
\begin{align}\label{eq:discrete_adj_eq}
(\nabla v_{\ell},\nabla \bar{p}_{\ell})_{\Omega} + \left(\frac{\partial f}{\partial y}(\cdot,\bar{y}_{\ell})\bar{p}_{\ell},v_{\ell}\right)_{\Omega}=(\bar{y}_{\ell} - y_{\Omega},v_{\ell})_{\Omega} \quad \forall v_{\ell} \in \mathbb{V}_{\ell}.
\end{align}
Here, $\bar{y}_{\ell}=\bar{y}_{\ell}(\bar{\mathfrak{u}})$ solves problem \eqref{eq:discrete_pde_semi} with $\mathfrak{u} = \bar{\mathfrak{u}}$.

As in the continuous case, the variational inequality \eqref{eq:semi_var_ineq} implies the following characterization for optimal controls $\bar{\mathfrak{u}}$, f.a.a. $x\in \Omega$: 
\begin{equation}\label{eq:pointwise_charac_discrete}
\bar{\mathfrak{u}}(x) = a ~ \text{ if } ~  \bar{p}_{\ell}(x) > 0, \quad \bar{\mathfrak{u}}(x) \in [a,b] ~ \text{ if } ~  \bar{p}_{\ell}(x) = 0, \quad \bar{\mathfrak{u}}(x) = b  ~ \text{ if } ~  \bar{p}_{\ell}(x) < 0. \hspace{-0.2cm}
\end{equation}

Since $\bar{\mathfrak{u}}$ implicitly depends on $\mathcal{T}_\ell$, in what follows we shall adopt the notation $\bar{\mathfrak{u}}_{\ell}$.


\section{A posteriori error analysis}\label{sec:a_post_analysis}

In this section, we devise an a posteriori error estimator for the optimal control problem \eqref{def:weak_ocp}--\eqref{eq:weak_st_eq}.
Such an error estimator will be formed by the sum of two contributions related to the discretization of the state and adjoint state equations.
We prove reliability and efficiency properties for the aforementioned error estimator.


\subsection{Reliability}\label{sec:reliability}

In what follows, we present error estimators for discrete state and adjoint state equations.


\subsubsection{A posteriori error estimates: state equation}

Let $\bar{\mathfrak{u}}_{\ell}$ be a local minimum of the semidiscrete optimal control problem. 
We introduce the auxiliary variable $y_{\bar{\mathfrak{u}}_{\ell}} \in H_0^1(\Omega)$ as the unique solution to
\begin{align}\label{eq:aux_state_apost}
(\nabla y_{\bar{\mathfrak{u}}_{\ell}},\nabla v)_{\Omega} + (f(\cdot,y_{\bar{\mathfrak{u}}_{\ell}}),v)_{\Omega} = (\bar{\mathfrak{u}}_{\ell},v)_{\Omega} \quad \forall v \in H_0^1(\Omega).
\end{align}
We note that the discrete optimal state $\bar{y}_{\ell}$, associated to $\bar{\mathfrak{u}}_{\ell}$, corresponds to the finite element approximation of $y_{\bar{\mathfrak{u}}_{\ell}}$ in $\mathbb{V}_{\ell}$.

We introduce the following global error estimator and local error indicator, respectively,
\begin{align}\label{def:state_indicator}
\eta_{st}^2 := \sum_{T\in \mathcal{T}_{\ell}} \eta_{st,T}^2,
\qquad
\eta_{st,T}^2 := h_T^4\| \bar{\mathfrak{u}}_{\ell} - f(\cdot,\bar{y}_{\ell})\|_{T}^{2} + \sum_{e\in \mathcal{E}_{T}}h_T^3\|\llbracket \nabla \bar{y}_{\ell}\cdot \mathbf{n} \rrbracket\|_e^2. 
\end{align}

We present the following reliability result and, for the sake of completeness, a proof.

\begin{lemma}[reliability: state equation]\label{lemma:rel_state}
Let Assumption \ref{A1} hold.
Let $\bar{\mathfrak{u}}_{\ell}$ be a local minimum of the semidiscrete optimal control problem and $\bar{y}_{\ell}$ its associated discrete optimal state.
If $\frac{\partial f}{\partial y}(\cdot, y)\in L^{\infty}(\Omega)$ for all $y\in H_0^1(\Omega)$, then
\begin{align}\label{eq:reliability_state_l2}
    \|y_{\bar{\mathfrak{u}}_{\ell}} - \bar{y}_{\ell}\|_{\Omega} 
    \lesssim
    \eta_{st},
\end{align}
where the hidden constant is independent of $y_{\bar{\mathfrak{u}}_{\ell}}$ and $\bar{y}_{\ell}$, the size of the elements in the mesh $\mathcal{T}_{\ell}$, and $\#\mathcal{T}_{\ell}$.
\end{lemma}
\begin{proof}
Let $\varphi \in H_0^1(\Omega)$ be the unique solution to
\begin{align*}
    (\nabla \varphi, \nabla v)_\Omega + \left(\frac{\partial f}{\partial y}(\cdot, y_{\theta})\varphi,v\right)_{\Omega} = (y_{\bar{\mathfrak{u}}_{\ell}} - \bar{y}_{\ell}, v)_{\Omega} \qquad \forall v \in H_0^1(\Omega),
\end{align*}
where $y_{\theta} = \bar{y}_{\ell} + \theta(y_{\bar{\mathfrak{u}}_{\ell}} - \bar{y}_{\ell})$, with $\theta\in (0,1)$, satisfies $\frac{\partial f}{\partial y}(\cdot, y_{\theta})(y_{\bar{\mathfrak{u}}_{\ell}} - \bar{y}_{\ell})= f(\cdot,y_{\bar{\mathfrak{u}}_{\ell}}) - f(\cdot,\bar{y}_{\ell})$.
Hence, choosing $v = y_{\bar{\mathfrak{u}}_{\ell}} - \bar{y}_{\ell}$ in the previous equation, and invoking the Galerkin orthogonality between $y_{\bar{\mathfrak{u}}_{\ell}}$ and $\bar{y}_{\ell}$, we obtain 
\begin{align*}
\|y_{\bar{\mathfrak{u}}_{\ell}} - \bar{y}_{\ell}\|_{\Omega}^2 
= & \,
(\nabla \varphi, \nabla (y_{\bar{\mathfrak{u}}_{\ell}} - \bar{y}_{\ell}))_\Omega + \left(\frac{\partial f}{\partial y}(\cdot, y_{\theta})\varphi,y_{\bar{\mathfrak{u}}_{\ell}} - \bar{y}_{\ell}\right)_{\Omega} \\
= & \,
(\nabla \varphi, \nabla (y_{\bar{\mathfrak{u}}_{\ell}} - \bar{y}_{\ell}))_\Omega + (\varphi, f(\cdot,y_{\bar{\mathfrak{u}}_{\ell}}) - f(\cdot,\bar{y}_{\ell}))_{\Omega} \\
= & \,
(\nabla (\varphi - \mathcal{L}_{\ell}\varphi), \nabla (y_{\bar{\mathfrak{u}}_{\ell}} - \bar{y}_{\ell}))_\Omega + (\varphi - \mathcal{L}_{\ell}\varphi, f(\cdot,y_{\bar{\mathfrak{u}}_{\ell}}) - f(\cdot,\bar{y}_{\ell}))_{\Omega}.
\end{align*}
Here, $\mathcal{L}_{\ell}: H^2(\Omega) \to \mathbb{V}_{\ell}$ denotes the classical Lagrange interpolator \cite{MR2373954,MR2050138}.
Then, the use of an elementwise integration by parts formula and standard approximation
properties for $\mathcal{L}_{\ell}$ results in 
\begin{align*}
   \|y_{\bar{\mathfrak{u}}_{\ell}} - \bar{y}_{\ell}\|_{\Omega}^2 
   \lesssim 
     \sum_{T\in \mathcal{T}_{\ell}}\left(h_T^2\| \bar{\mathfrak{u}}_{\ell} - f(\cdot,\bar{y}_{\ell})\|_{T}\|\varphi\|_{H^2(T)} + \sum_{e\in \mathcal{E}_{T}}h_T^{\frac{3}{2}}\|\llbracket \nabla \bar{y}_{\ell}\cdot \mathbf{n} \rrbracket\|_e\|\varphi\|_{H^2(T)}\right).
\end{align*}
We conclude the proof in view of Cauchy-Schwarz inequality and the bound $\|\varphi\|_{H^2(\Omega)} \lesssim \|y_{\bar{\mathfrak{u}}_{\ell}} - \bar{y}_{\ell}\|_{\Omega}$ (cf. Theorem \ref{thm:wp_and_stab_linear}).
\end{proof}

\begin{remark}[on assumption $\frac{\partial f}{\partial y}(\cdot, y)\in L^{\infty}(\Omega)$ for all $y\in H_0^1(\Omega)$]
Let $y_{\theta} = \bar{y}_{\ell} + \theta(y_{\bar{\mathfrak{u}}_{\ell}} - \bar{y}_{\ell})$, with $\theta\in (0,1)$ be given as in the proof of Lemma \ref{lemma:rel_state}. 
Since $\bar{y}_{\ell}$ is not necessarily uniformly bounded in $L^{\infty}(\Omega)$ when $\ell$ increases, we cannot ensure that $\|y_{\theta}\|_{L^{\infty}(\Omega)} \leq C$ uniformly for some $C>0$.
Therefore, such an assumption is required to have, in \eqref{eq:reliability_state_l2}, a hidden constant independent of $\|\bar{y}_{\ell}\|_{L^{\infty}(\Omega)}$.
\end{remark}


\subsubsection{A posteriori error estimates: adjoint state equation}

Let $\bar{\mathfrak{u}}_{\ell}$ be a local minimum of the semidiscrete optimal control problem and $\bar{y}_{\ell}$ its associated discrete optimal state.
We introduce the auxiliary variable $p_{\bar{\mathfrak{u}}_{\ell}} \in H_0^1(\Omega)$ as the unique solution to
\begin{align}\label{eq:aux_adjoint_apost}
(\nabla v,\nabla p_{\bar{y}_{\ell}})_{\Omega} + \left(\frac{\partial f}{\partial y}(\cdot, \bar{y}_{\ell})p_{\bar{y}_{\ell}},v\right)_{\Omega} = (\bar{y}_{\ell} - y_{\Omega},v)_{\Omega} \quad \forall v \in H_0^1(\Omega).
\end{align}
We immediately notice that $\bar{p}_{\ell}$ corresponds to the finite element approximation of $p_{\bar{y}_{\ell}}$ in $\mathbb{V}_{\ell}$.

We define the global error estimator and local error indicators for the adjoint state equation as
\begin{align}\label{def:adjoint_indicator}
 {\eta}_{adj}:=\max_{T\in \mathcal{T}_{\ell}}\eta_{adj,T}, \qquad
\eta_{adj,T} := h_T^{2-\frac{d}{2}}\| \bar{y}_{\ell} - y_{\Omega} - \tfrac{\partial f}{\partial y}(\cdot, \bar{y}_{\ell})\bar{p}_{\ell}\|_{T} + h_T\max_{e\in\mathcal{E}_{T}}\|\llbracket \nabla \bar{p}_{\ell}\cdot \mathbf{n} \rrbracket\|_{L^\infty(e)}.
\end{align} 

To present the global reliability result below, we introduce the term
\begin{equation*}
\iota_{\ell}:= \left|\log\bigg(\max_{T\in\mathcal{T}_{\ell}}\frac{1}{h_T}\bigg)\right|.
\end{equation*}

\begin{lemma}[reliability: adjoint state equation]\label{lemma:rel_adj}
Let Assumption \ref{A1} hold.
Let $\bar{\mathfrak{u}}_{\ell}$ be a local minimum of the semidiscrete optimal control problem with $\bar{y}_{\ell}$ and $\bar{p}_{\ell}$ being the corresponding state and adjoint state, respectively. 
If $\frac{\partial f}{\partial y}(\cdot, y)\in L^{\infty}(\Omega)$ for all $y\in H_0^1(\Omega)$, then
\begin{align*}
    \|p_{\bar{y}_{\ell}} - \bar{p}_{\ell}\|_{L^{\infty}(\Omega)} 
    \lesssim
    \iota_{\ell}\eta_{adj},
\end{align*}
where the hidden constant is independent of $p_{\bar{\mathfrak{u}}_{\ell}}$ and $\bar{p}_{\ell}$, the size of the elements in the mesh $\mathcal{T}_{\ell}$, and $\#\mathcal{T}_{\ell}$.
\end{lemma}
\begin{proof}
We introduce the auxiliary variable $\zeta \in H_0^{1}(\Omega)\cap H^{2}(\Omega)$ as the unique solution to
\begin{align}\label{eq:aux_zeta}
    (\nabla \zeta, \nabla v)_{\Omega} = \left(\frac{\partial f}{\partial y}(\cdot, \bar{y}_\ell)(p_{\bar{y}_{\ell}} - \bar{p}_{\ell}),v\right)_{\Omega} \quad \forall v\in H_0^{1}(\Omega).
\end{align}
Then, we define the term $\mathfrak{e} := (p_{\bar{y}_{\ell}} - \bar{p}_{\ell}) + \zeta$, and immediately note that $\mathfrak{e}\in W_0^{1,q}(\Omega)$ for some $q>d$, since $p_{\bar{y}_{\ell}},\zeta\in H^{2}(\Omega)$.
From the definition of  $\mathfrak{e}$ and problem \eqref{eq:aux_zeta}, we infer that $\zeta$ also solves 
\begin{align*}
    (\nabla \zeta, \nabla v)_{\Omega} + \left(\frac{\partial f}{\partial y}(\cdot, \bar{y}_\ell)\zeta,v\right)_{\Omega}= \left(\frac{\partial f}{\partial y}(\cdot, \bar{y}_\ell)\mathfrak{e},v\right)_{\Omega} \quad \forall v\in H_0^{1}(\Omega).
\end{align*}
Consequently, in view of Lemma \ref{thm:wp_and_stab_linear}, we have that $\|\zeta\|_{L^{\infty}(\Omega)} \lesssim \|\mathfrak{e}\|_{L^{\infty}(\Omega)}$, which implies 
\begin{equation*}
     \|p_{\bar{y}_{\ell}} - \bar{p}_{\ell}\|_{L^{\infty}(\Omega)} 
     \lesssim
     \|\mathfrak{e}\|_{L^{\infty}(\Omega)},
\end{equation*}
with a hidden constant that is independent of $\bar{y}_{\ell}$.
In what follows, we estimate $\|\mathfrak{e}\|_{L^{\infty}(\Omega)}$ based on the arguments developed in the proof of \cite[Lemma 4.2]{MR3878607}. 

Let $x_{\mathrm{M}}\in \Omega$ be such that $\|\mathfrak{e}\|_{L^{\infty}(\Omega)} = |\mathfrak{e}(x_{\mathrm{M}})|$.
Let $G=G(x_{\mathrm{M}},\cdot)$ be the Green function defined in Appendix \ref{sec:appendix} and let $G_{\ell}(x_{\mathrm{M}},\cdot) = G_{\ell} \in \mathbb{V}_{\ell} $ be some suitable quasi-interpolation of this function by using, e.g., the Scott-Zhang interpolant (see \cite{MR1011446} and \cite[Section 4.8]{MR2373954}).
The use of representation \eqref{eq:Green_rep}, a density argument that allows us to choose $v=G$ in \eqref{eq:aux_zeta}, even when $G\notin H_0^{1}(\Omega)$, and Galerkin orthogonality yield
\begin{align}\label{eq:identity_green_adjoint}
    \mathfrak{e}(x_{\mathrm{M}})
    = (\nabla G, \nabla \mathfrak{e})_{\Omega} 
    = & \,  (\nabla G, \nabla (p_{\bar{y}_{\ell}} - \bar{p}_{\ell}))_{\Omega} +  \left(\tfrac{\partial f}{\partial y}(\cdot, \bar{y}_\ell)(p_{\bar{y}_{\ell}} - \bar{p}_{\ell}),G\right)_{\Omega} \\
    = & \,  (\nabla (G - G_{\ell}), \nabla (p_{\bar{y}_{\ell}} - \bar{p}_{\ell}))_{\Omega} + \left(\tfrac{\partial f}{\partial y}(\cdot, \bar{y}_{\ell})(G - G_{\ell}),p_{\bar{y}_{\ell}} - \bar{p}_{\ell}\right)_{\Omega}. \nonumber
\end{align}
The same aforementioned density argument can be used to conclude that $v=G$ in \eqref{eq:aux_adjoint_apost} is also well-defined.
This, in combination with an elementwise integration by parts in \eqref{eq:identity_green_adjoint}, results in
\begin{equation}\label{eq:error_e_I_II}
      \mathfrak{e}(x_{\mathrm{M}})
      = 
      \sum_{T\in \mathcal{T}_{\ell}} \left( \bar{y}_{\ell} - y_{\Omega} - \tfrac{\partial f}{\partial y}(\cdot, \bar{y}_{\ell})\bar{p}_{\ell}, G - G_{\ell}\right)_{T} - \sum_{e\in \mathcal{E}_{\ell}}(\llbracket \nabla \bar{p}_{\ell}\cdot \mathbf{n}\rrbracket, G - G_{\ell})_{e} = : \mathsf{I} + \mathsf{II}.
\end{equation}

We bound the term $\mathsf{I}$. The use Cauchy-Schwarz inequality yields
\begin{equation*}
    |\mathsf{I}| \leq \sum_{T\in \mathcal{T}_{\ell}}\|\bar{y}_{\ell} - y_{\Omega} - \tfrac{\partial f}{\partial y}(\cdot, \bar{y}_{\ell})\bar{p}_{\ell}\|_{T} \| G - G_{\ell}\|_{T}.
\end{equation*}
We now estimate the sum in the right-hand side of the previous expression on two different subsets of $\mathcal{T}_{\ell}$, namely, $\mathcal{M}_{x_{M}}:=\{T\in \mathcal{T}_{\ell} : x_{M} \in \mathcal{M}_{T}\}$ and $\mathcal{M}^{c}_{x_{M}}:= \mathcal{T}_{\ell}\setminus \mathcal{M}_{x_{M}}$, where $\mathcal{M}_T:= \left \{ T^{\prime}\in\mathcal{T}_{\ell}: T\cap T^\prime\neq\emptyset \right \}$ for every $T\in \mathcal{T}_{\ell}$.
We begin by estimating over the elements in $\mathcal{M}^{c}_{x_{M}}$. 
The use of standard approximation results (\cite[Section 4]{MR1011446}, see also \cite[Proposition 5.1]{MR3076038}) yields
\begin{align*}
    \|G - G_{\ell}\|_{T} \lesssim h_T^{2-\tfrac{d}{2}}\|D^{2}G\|_{L^1(\mathcal{M}_{T})} \quad \forall T\in \mathcal{M}^{c}_{x_{M}}.
\end{align*}
Hence, the finite intersection property of patches and estimate \eqref{eq:D2G_est_rho} with $\omega_{\rho}(x_M)$ being the biggest ball satisfying $\omega_\rho(x_{M})\subset \cup_{T\in \mathcal{T}_{\ell}:x_{M}\in T}T$, in which case $\min_{T'\in \mathcal{T}_{\ell}}h_{T'}\lesssim \rho$, give
\begin{align*}
    \sum_{T\in \mathcal{M}_{x_{M}}^{c}}\|\bar{y}_{\ell} - y_{\Omega} - \tfrac{\partial f}{\partial y}(\cdot, \bar{y}_{\ell})\bar{p}_{\ell}\|_{T} \| G - G_{\ell}\|_{T}
    \lesssim & \,
    \max_{T\in\mathcal{T}_{\ell}}\left(h_{T}^{2-\tfrac{d}{2}}\|\bar{y}_{\ell} - y_{\Omega} - \tfrac{\partial f}{\partial y}(\cdot, \bar{y}_{\ell})\bar{p}_{\ell}\|_{T}\right)\sum_{T\in \mathcal{M}_{x_{M}}^{c}}\|D^{2}G\|_{L^1(\mathcal{M}_{T})} \\
    \lesssim & \,
    \iota_{\ell}\eta_{adj}.
\end{align*}
To estimate the sum in $\mathcal{M}_{x_{M}}$, we invoke approximation properties of quasi-interpolators and estimate \eqref{eq:DG_est_rho}.
Hence, we obtain $\|G - G_{\ell}\|_{T}  \lesssim \| \nabla G\|_{L^{2d/(d+2)}(\mathcal{M}_{T})} \lesssim h_{T}^{2 - \frac{d}{2}}$.  
Finally, since the number of elements in $\mathcal{M}_{x_{M}}$ is uniformly bounded, we obtain that
\begin{align*}
    \sum_{T\in \mathcal{M}_{x_{M}}}\|\bar{y}_{\ell} - y_{\Omega} - \tfrac{\partial f}{\partial y}(\cdot, \bar{y}_{\ell})\|_{T}\|G - G_{\ell}\|_{T}
    \lesssim
    \sum_{T\in \mathcal{M}_{x_{M}}}h_T^{2-\tfrac{d}{2}}\|\bar{y}_{\ell} - y_{\Omega} - \tfrac{\partial f}{\partial y}(\cdot, \bar{y}_{\ell})\|_{T} 
    \lesssim
    \eta_{adj}.
\end{align*}
The previous estimates thus imply that $|\mathsf{I}| \lesssim \iota_{\ell}\eta_{adj}$.

To estimate $\mathsf{II}$ in \eqref{eq:error_e_I_II}, we write $|\mathsf{II}| \leq \sum_{e\in \mathcal{E}_{\ell}} \|\llbracket \nabla \bar{p}_{\ell}\cdot \mathbf{n}\rrbracket\|_{L^{\infty}(e)}\|G - G_{\ell}\|_{L^{1}(e)}$; we shall estimate the latter by using again the sets $\mathcal{M}_{x_{M}}$ and $\mathcal{M}^{c}_{x_{M}}$.
On the one hand, consider the case when $e\in \mathcal{E}_{T}$ with $T\in \mathcal{M}^{c}_{x_{M}}$.
The use of the scaled trace inequality $\|G - G_{\ell}\|_{L^{1}(e)} \lesssim h_T^{-1}\|G - G_{\ell}\|_{L^{1}(T)} + \|\nabla (G - G_{\ell})\|_{L^{1}(T)}$ \cite[Section 3.3]{MR3059294}, in combination with standard approximation properties (\cite[Section 4]{MR1011446}, see also \cite[Proposition 5.1]{MR3076038}), results in $\|G - G_{\ell}\|_{L^1(e)} \lesssim h_T\| D^2 G\|_{L^{1}(\mathcal{M}_{T})}$.
Consequently, in view of estimate \eqref{eq:D2G_est_rho}, we conclude that
\begin{align*}
    \sum_{e\in \mathcal{E}_{T}: T \in \mathcal{M}^{c}_{x_{M}}}\|\llbracket \nabla \bar{p}_{\ell}\cdot \mathbf{n}\rrbracket\|_{L^{\infty}(e)}\|G - G_{\ell}\|_{L^{1}(e)}
    \lesssim
    \eta_{adj}\sum_{T \in \mathcal{M}^{c}_{x_{M}}}\|D^2G\|_{L^1(\mathcal{M}_{T})}
    \lesssim 
    \iota_{\ell}\eta_{adj}.
\end{align*}
On the other hand, to estimate the product $\|\llbracket \nabla \bar{p}_{\ell}\cdot \mathbf{n}\rrbracket\|_{L^{\infty}(e)}\|G - G_{\ell}\|_{L^{1}(e)}$ when $e\in \mathcal{E}_{T}$ with $T\in \mathcal{M}_{x_{M}}$, we use that $G\in W^{1,(d+1)/d}(\Omega)$ (note that $(d+1)/d < d/(d-1)$ and $G\in W^{1,\mathsf{p}}(\Omega)$ with $\mathsf{p} < d/(d-1)$), an application of the inequality $\|G - G_{\ell}\|_{L^{1}(e)} \lesssim h_T^{-1}\|G - G_{\ell}\|_{L^{1}(T)} + \|\nabla (G - G_{\ell})\|_{L^{1}(T)}$, and standard approximation estimates. 
These arguments yield 
\begin{equation*}
    \|\llbracket \nabla \bar{p}_{\ell}\cdot \mathbf{n}\rrbracket\|_{L^{\infty}(e)}\|G - G_{\ell}\|_{L^{1}(e)}
    \lesssim 
    h_T^{d-d^2/(d+1)}\|\llbracket \nabla \bar{p}_{\ell}\cdot \mathbf{n}\rrbracket\|_{L^{\infty}(e)}\|\nabla G\|_{L^{(d+1)/d}(\mathcal{M}_{T})}.
\end{equation*}
This, in combination with the estimate $\|\nabla G\|_{L^{(d+1)/d}(\mathcal{M}_{T})} \lesssim h_{T}^{1-d+d^2/(d+1)}$ (cf. estimate \eqref{eq:DG_est_rho}), gives us
\begin{align*}
    \sum_{e\in \mathcal{E}_{T}: T\in \mathcal{M}_{x_{M}}}\|\llbracket \nabla \bar{p}_{\ell}\cdot\mathbf{n}\rrbracket\|_{L^{\infty}(e)}\|G - G_{\ell}\|_{e}
    \lesssim
    \sum_{T\in \mathcal{M}_{x_M}}h_T\|\llbracket \nabla \bar{p}_{\ell}\cdot \mathbf{n}\rrbracket\|_{L^{\infty}(\partial T\setminus \partial \Omega)}
    \lesssim
    \eta_{adj},
\end{align*}
where, in the last inequality, we have used that $\#\mathcal{M}_{x_M}$ is uniformly bounded.
Therefore, $|\mathsf{II}| \lesssim \iota_{\ell}\eta_{adj}$.

We conclude the desired result using the estimates obtained for $\mathsf{I}$ and $\mathsf{II}$ into \eqref{eq:error_e_I_II}.
\end{proof}


\subsubsection{A posteriori error estimates: optimal control problem}

We propose the a posteriori error estimator
\begin{align}\label{def:total_estimator}
    \eta_{ocp}:= \eta_{st} + \eta_{adj},
\end{align}
where $\eta_{st}$ and $\eta_{adj}$ are defined in \eqref{def:state_indicator} and \eqref{def:adjoint_indicator}, respectively.

\begin{theorem}[global reliability]\label{thm:global_rel_ocp}
Let Assumption \ref{A1} hold.
Let $\bar{u}\in U_{ad}$ be a local solution to \eqref{def:weak_ocp}--\eqref{eq:weak_st_eq} such that it satisfies assumption \eqref{eq:assumption_control_growth}.
Let $\bar{\mathfrak{u}}_{\ell}$ be a local minimum to the semidiscrete optimal control problem with $\bar{y}_{\ell}$ and  $\bar{p}_{\ell}$ being the corresponding state and adjoint state, respectively. 
If $\|\bar{u} - \bar{\mathfrak{u}}_{\ell}\|_{L^1(\Omega)} < \min\{\delta_{\frac{\kappa}{2}},\alpha\}$, $\frac{\partial f}{\partial y}(\cdot, y)\in L^{\infty}(\Omega)$ for all $y\in H_0^1(\Omega)$, and $\frac{\partial f}{\partial y}(\cdot, y)$ is globally Lipschitz with respect to $y\in \mathbb{R}$, then
\begin{equation}\label{eq:global_reliab}
    \|\bar{u} - \bar{\mathfrak{u}}_{\ell}\|_{L^1(\Omega)} 
    \lesssim
    (1 + \iota_{\ell})^{\gamma}\eta_{ocp}^{\gamma},
\end{equation}
and 
\begin{equation}\label{eq:global_reliab_st_adj}
    \|\bar{y}-\bar{y}_{\ell}\|_{\Omega} \lesssim  (1 + \iota_{\ell})^{\gamma}\eta_{ocp}^{\gamma} + \eta_{st}, \qquad
    \|\bar{p}-\bar{p}_{\ell}\|_{L^{\infty}(\Omega)} \lesssim (1 + \iota_{\ell})^{\gamma}\eta_{ocp}^{\gamma} +  (1 + \iota_{\ell})\eta_{ocp}.
\end{equation}
The hidden constants in \eqref{eq:global_reliab} and \eqref{eq:global_reliab_st_adj} are independent of the continuous and discrete optimal variables, the size of the elements in the mesh $\mathcal{T}_{\ell}$, and $\#\mathcal{T}_{\ell}$.
\end{theorem}
\begin{proof}
We proceed in three steps.

\underline{Step 1.} (estimation of $\|\bar{u} - \bar{\mathfrak{u}}_{\ell}\|_{L^1(\Omega)}$)
Choosing $u = \bar{u}$ in \eqref{eq:semi_var_ineq} and using the mean value theorem, we obtain
\begin{align*}
  0  \geq & \,  J_{\ell}^{\prime}(\bar{\mathfrak{u}}_{\ell})(\bar{\mathfrak{u}}_{\ell} - \bar{u}) \\
   \geq & \, [J'(\bar{u})(\bar{\mathfrak{u}}_{\ell} - \bar{u}) + J''(\bar{u})(\bar{\mathfrak{u}}_{\ell} - \bar{u})^2] - |[J'(\bar{\mathfrak{u}}_{\ell}) - J'(\bar{u})](\bar{\mathfrak{u}}_{\ell} - \bar{u}) - J''(\bar{u})(\bar{\mathfrak{u}}_{\ell} - \bar{u})^2| + [J'_{\ell}(\bar{\mathfrak{u}}_{\ell}) - J'(\bar{\mathfrak{u}}_{\ell})](\bar{\mathfrak{u}}_{\ell} - \bar{u}) \\
 = & \, [J'(\bar{u})(\bar{\mathfrak{u}}_{\ell} - \bar{u}) + J''(\bar{u})(\bar{\mathfrak{u}}_{\ell} - \bar{u})^2]  - |[J''(u_{\theta})- J''(\bar{u})](\bar{\mathfrak{u}}_{\ell} - \bar{u})^2| + [J'_{\ell}(\bar{\mathfrak{u}}_{\ell}) - J'(\bar{\mathfrak{u}}_{\ell})](\bar{\mathfrak{u}}_{\ell} - \bar{u}),
\end{align*}
with $u_{\theta} = \bar{u} + \theta (\bar{\mathfrak{u}}_{\ell} - \bar{u})$ with $\theta\in (0,1)$.
The use of assumptions $\|\bar{u} - \bar{\mathfrak{u}}_{\ell}\|_{L^1(\Omega)} < \min\{\delta_{\frac{\kappa}{2}},\alpha\}$ and \eqref{eq:assumption_control_growth}, in combination with the local continuity property \eqref{eq:sec_var_est} (with $\varepsilon=\tfrac{\kappa}{2}$), yield
\begin{align}\label{eq:estimate_error_I}
    \frac{\kappa}{2}\|\bar{\mathfrak{u}}_{\ell} - \bar{u}\|_{L^1(\Omega)}^{1+\frac{1}{\gamma}} 
    \leq 
    [J'(\bar{\mathfrak{u}}_{\ell}) - J'_{\ell}(\bar{\mathfrak{u}}_{\ell}) ](\bar{\mathfrak{u}}_{\ell} - \bar{u}).
\end{align}
To estimate the term in the right-hand side of \eqref{eq:estimate_error_I}, we introduce the auxiliary variable $p_{y_{\bar{\mathfrak{u}}_{\ell}}} \in H_0^1(\Omega)$ as the unique solution to
\begin{align*}
(\nabla v,\nabla p_{y_{\bar{\mathfrak{u}}_{\ell}}})_{\Omega} + \left(\frac{\partial f}{\partial y}(\cdot, y_{\bar{\mathfrak{u}}_{\ell}})p_{y_{\bar{\mathfrak{u}}_{\ell}}},v\right)_{\Omega} = (y_{\bar{\mathfrak{u}}_{\ell}} - y_{\Omega},v)_{\Omega} \quad \forall v \in H_0^1(\Omega),
\end{align*}
where $y_{\bar{\mathfrak{u}}_{\ell}}\in H_0^1(\Omega)$ corresponds to the unique solution to \eqref{eq:aux_state_apost}.
We immediately note that $J'(\bar{\mathfrak{u}}_{\ell})(\bar{\mathfrak{u}}_{\ell} - \bar{u}) = (p_{y_{\bar{\mathfrak{u}}_{\ell}}},\bar{\mathfrak{u}}_{\ell} - \bar{u})_{\Omega}$, which implies, in light of \eqref{eq:estimate_error_I}, that 
\begin{align*}
 \frac{\kappa}{2}\|\bar{\mathfrak{u}}_{\ell} - \bar{u}\|_{L^1(\Omega)}^{1+\frac{1}{\gamma}} \leq (p_{y_{\bar{\mathfrak{u}}_{\ell}}} - \bar{p}_{\ell},\bar{\mathfrak{u}}_{\ell} - \bar{u})_{\Omega}.   
\end{align*}
We thus invoke the auxiliary variable $p_{\bar{y}_{\ell}}$, solution to \eqref{eq:aux_adjoint_apost}, and the a posteriori error estimate stated in Lemma \ref{lemma:rel_adj} to obtain
\begin{align}\label{eq:estimate_error_II}
   \|\bar{\mathfrak{u}}_{\ell} - \bar{u}\|_{L^1(\Omega)}^{\frac{1}{\gamma}} 
    \lesssim
    \|p_{y_{\bar{\mathfrak{u}}_{\ell}}} - p_{\bar{y}_{\ell}}\|_{L^{\infty}(\Omega)} + \|p_{\bar{y}_{\ell}} - \bar{p}_{\ell}\|_{L^{\infty}(\Omega)}
    \lesssim 
    \|p_{y_{\bar{\mathfrak{u}}_{\ell}}} - p_{\bar{y}_{\ell}}\|_{L^{\infty}(\Omega)} + \iota_{\ell}\eta_{adj}.
\end{align}
To estimate $\|p_{y_{\bar{\mathfrak{u}}_{\ell}}} - p_{\bar{y}_{\ell}}\|_{L^{\infty}(\Omega)}$, we first note that $p_{y_{\bar{\mathfrak{u}}_{\ell}}} - p_{\bar{y}_{\ell}} \in H_0^{1}(\Omega)\cap H^2(\Omega)$ is the solution to
\begin{align*}
    (\nabla v,\nabla (p_{y_{\bar{\mathfrak{u}}_{\ell}}} - \nabla p_{\bar{y}_{\ell}}))_{\Omega} + \left(\frac{\partial f}{\partial y}(\cdot, y_{\bar{\mathfrak{u}}_{\ell}})(p_{y_{\bar{\mathfrak{u}}_{\ell}}} - p_{\bar{y}_{\ell}}),v\right)_{\Omega} = (y_{\bar{\mathfrak{u}}_{\ell}} - \bar{y}_{\ell},v)_{\Omega} + \left(\left[\frac{\partial f}{\partial y}(\cdot, \bar{y}_{\ell}) - \frac{\partial f}{\partial y}(\cdot, y_{\bar{\mathfrak{u}}_{\ell}})\right]p_{\bar{y}_{\ell}},v\right)_{\Omega}
 \end{align*}
 for all $v \in H_0^1(\Omega)$.
 Then, an application of Theorem \ref{thm:wp_and_stab_linear}, the uniform boundedness of $\|p_{\bar{y}_{\ell}}\|_{L^{\infty}(\Omega)}$, and the Lipschitz property of $\frac{\partial f}{\partial y}(\cdot, y)$ with respect to $y\in\mathbb{R}$, imply that
 \begin{align}\label{eq:py-pyl_infty}
     \|p_{y_{\bar{\mathfrak{u}}_{\ell}}} - p_{\bar{y}_{\ell}}\|_{L^{\infty}(\Omega)}
     \lesssim
     \|y_{\bar{\mathfrak{u}}_{\ell}} - \bar{y}_{\ell}\|_{\Omega} + \|\tfrac{\partial f}{\partial y}(\cdot, \bar{y}_{\ell}) - \tfrac{\partial f}{\partial y}(\cdot, y_{\bar{\mathfrak{u}}_{\ell}})\|_{\Omega}
     \lesssim
     \|y_{\bar{\mathfrak{u}}_{\ell}} - \bar{y}_{\ell}\|_{\Omega}.
 \end{align}
Therefore, using the bound $\|y_{\bar{\mathfrak{u}}_{\ell}} - \bar{y}_{\ell}\|_{\Omega} \lesssim \eta_{st}$ (see Lemma \ref{lemma:rel_state}) in \eqref{eq:py-pyl_infty}, and the resulting estimate in \eqref{eq:estimate_error_II}, we conclude that 
\begin{align}\label{eq:estimate_control_u-ul}
    \|\bar{\mathfrak{u}}_{\ell} - \bar{u}\|_{L^1(\Omega)} 
    \lesssim
    (\eta_{st} + \iota_{\ell}\eta_{adj})^{\gamma}
    \leq
    (1 + \iota_{\ell})^{\gamma}\eta_{ocp}^{\gamma}.
\end{align}

\underline{Step 2.} (estimation of $\|\bar{y} - \bar{y}_{\ell}\|_{\Omega}$) The use of the triangle inequality and Lemma \ref{lemma:rel_state} results in
\begin{align}\label{eq:triangle_y-yl}
    \|\bar{y} - \bar{y}_{\ell}\|_{\Omega}
    \leq
    \|\bar{y} - y_{\bar{\mathfrak{u}}_{\ell}}\|_{\Omega} + \|y_{\bar{\mathfrak{u}}_{\ell}} - \bar{y}_{\ell}\|_{\Omega}
    \lesssim 
    \|\bar{y} - y_{\bar{\mathfrak{u}}_{\ell}}\|_{\Omega} + \eta_{st}.
\end{align}
We note that $\bar{y} - \bar{y}_{\bar{\mathfrak{u}}_{\ell}}\in H_0^{1}(\Omega)\cap H^2(\Omega)$ corresponds to the unique solution to
\begin{align*}
(\nabla v,\nabla (\bar{y} - \bar{y}_{\bar{\mathfrak{u}}_{\ell}}))_{\Omega} + \left(\frac{\partial f}{\partial y}(\cdot, y_{\theta})(\bar{y} - \bar{y}_{\bar{\mathfrak{u}}_{\ell}}),v\right)_{\Omega} = (\bar{u} - \bar{\mathfrak{u}}_{\ell},v)_{\Omega} \quad \forall v \in H_0^1(\Omega),
\end{align*}
with $y_{\theta} = \bar{y}_{\bar{\mathfrak{u}}_{\ell}} + \theta (\bar{y} - \bar{y}_{\bar{\mathfrak{u}}_{\ell}})$ for some $\theta\in (0,1)$.
Hence, from Lemma \ref{lemma:stab_Ls} it follows that $\|\bar{y} - \bar{y}_{\bar{\mathfrak{u}}_{\ell}}\|_{\Omega} \lesssim \| \bar{u} - \bar{\mathfrak{u}}_{\ell}\|_{L^1(\Omega)}$. 
Using the latter in \eqref{eq:triangle_y-yl} and invoking estimate \eqref{eq:estimate_control_u-ul}, we arrive at 
\begin{align*}
    \|\bar{y} - \bar{y}_{\ell}\|_{\Omega} 
    \lesssim 
    (1 + \iota_{\ell})^{\gamma}\eta_{ocp}^{\gamma} + \eta_{st}.
\end{align*}

\underline{Step 3.} (estimation of $\|\bar{p} - \bar{p}_{\ell}\|_{L^{\infty}(\Omega)}$)
The triangle inequality and Lemma \ref{lemma:rel_adj} yield
\begin{align*}
    \|\bar{p} - \bar{p}_{\ell}\|_{L^{\infty}(\Omega)}
    \leq
    \|\bar{p} - \bar{p}_{\bar{y}_{\ell}}\|_{L^{\infty}(\Omega)} + \|\bar{p}_{\bar{y}_{\ell}} - \bar{p}_{\ell}\|_{L^{\infty}(\Omega)}
    \lesssim 
    \|\bar{p} - \bar{p}_{\bar{y}_{\ell}}\|_{L^{\infty}(\Omega)} + \iota_{\ell}\eta_{adj}.
\end{align*}
Using similar arguments to those that lead to \eqref{eq:py-pyl_infty} allows us to obtain $\|\bar{p} - \bar{p}_{\bar{y}_{\ell}}\|_{L^{\infty}(\Omega)} \lesssim \|\bar{y} - \bar{y}_{\ell}\|_{\Omega}$.
This bound, in conjunction with $\|\bar{y} - \bar{y}_{\ell}\|_{\Omega} \lesssim (1 + \iota_{\ell})^{\gamma}\eta_{ocp}^{\gamma} + \eta_{st}$, gives as a result the desired estimate.
\end{proof}

\begin{remark}[Case $\gamma=1$]
    If $\gamma = 1$, then we have the global reliability estimate
    \begin{align*}
    \|\bar{u} - \bar{\mathfrak{u}}_{\ell}\|_{L^1(\Omega)}   +  \|\bar{y}-\bar{y}_{\ell}\|_{\Omega}  + \|\bar{p}-\bar{p}_{\ell}\|_{L^{\infty}(\Omega)}
    \lesssim \iota_{\ell}\eta_{ocp}.
    \end{align*}
    In particular, the error estimator $\eta_{ocp}$ is an upper bound for the total error associated with the optimal control problem \eqref{def:weak_ocp}--\eqref{eq:weak_st_eq}.
\end{remark}


\subsubsection{Comparison with the error of the regularized problem}

Let $\lambda>0$. We consider the Tikhonov regularized problem
\begin{align}\label{ocp:Tik}
\min_{ u \in U_{ad}} \left\{ J_\lambda(u):=\nicefrac{1}{2}\|y_{u} - y_{\Omega}\|_{\Omega}^{2}+\nicefrac{\lambda}{2} \|u\|_{\Omega}^2\right\} \quad \text{subject to} \quad \eqref{eq:weak_st_eq}.
\end{align}
We estimate the degree to which the numerical solutions to the unregularized and regularized problems differ. 
This is of interest since the regularized problem is often chosen as a substitute for the unregularized one due to its easier mathematical structure; see, e.g., \cite{MR3095657,MR3385653}.
The result in this section allows for a comparison of the numerical cost and quality of convergence of the regularized problem as the regularization parameter $\lambda >0$ progressively decreases.

Three main ingredients are needed to achieve this estimation. 
The first ingredient is the stability under perturbations of optimal controls of the unregularized problem, which is guaranteed by condition \eqref{eq:assumption_control_growth}; see \cite{MR4525177}. 
That is, given $\bar{u}_\lambda$ a locally optimal solution to \eqref{ocp:Tik} that is sufficiently close to a locally optimal solution $\bar{u}$, that satisfies condition \eqref{eq:assumption_control_growth}, it holds for a positive constant $C$ independent of $\lambda$, that
\begin{equation*}
    \| \bar u-\bar u_\lambda\|_{L^1(\Omega)}\leq C\| \bar u_\lambda \|_{L^\infty(\Omega)} \lambda.
\end{equation*}
The second ingredient is the global reliability estimate $\|\bar{u} - \bar{\mathfrak{u}}_{\ell}\|_{L^1(\Omega)}  \lesssim (1 + \iota_{\ell})^{\gamma}\eta_{ocp}^{\gamma}$ (cf. \eqref{eq:global_reliab}) proved in Theorem \ref{thm:global_rel_ocp}.
The third and final ingredient is a globally reliable a posteriori error estimator for solutions to the regularized problem \eqref{ocp:Tik}. This can be found, e.g., in \cite[Theorems 5.2 and 7.1]{zbMATH07477246}.
These estimates together, allow us to conclude that 
\begin{align*}
    \|\bar u_{\ell} -\bar u_{\lambda,\ell}\|_{L^1(\Omega)}&\leq \|\bar u_{\ell} -\bar u\|_{L^1(\Omega)}+\|\bar u-\bar u_{\lambda}\|_{L^1(\Omega)}+\|\bar u_{\lambda} -\bar u_{\lambda,\ell}\|_{L^1(\Omega)}\\
    &\lesssim (1 + \iota_{\ell})^{\gamma}\eta_{ocp}^{\gamma} + \|\bar u_\lambda\|_{L^\infty(\Omega)} \lambda +\|\bar u_{\lambda} -\bar u_{\lambda,\ell}\|_{\Omega} \\
    &\lesssim (1 + \iota_{\ell})^{\gamma}\eta_{ocp}^{\gamma} + \lambda + E_{ocp},
\end{align*}
where $\bar{u}_{\lambda,\ell}$ denotes a suitable finite element approximation of $\bar{u}_\lambda$ and $E_{ocp}$ is an error estimator that can be chosen as in paper \cite{zbMATH07477246}.


\subsection{Efficiency}

In what follows, we analyze efficiency properties for the a posteriori error estimators ${\eta}_{st}$ and ${\eta}_{adj}$, introduced in section \ref{sec:reliability}. 

Given $T \in \mathcal{T}_{\ell}$ and $e \in \mathcal{E}_T$, we denote by $\varphi_{T}$ and $\varphi_{e}$ the classical \emph{interior} and \emph{edge} bubble functions, respectively; see, e.g., \cite[Section 2.3.1]{MR1885308}. 
We also introduce the following notation: for an edge, triangle or tetrahedron $\omega$, let $\mathcal{V}(\omega)$ be the set of vertices of $\omega$. 
We recall that $\mathcal{N}_{e}$ denotes the patch composed by the two elements $T^+$ and $T^-$ sharing $e$. 
Hence, we introduce the following edge/face bubble function 
\begin{equation}\label{def:aux_bubble}
\psi_{e}|_{\mathcal{N}_{e}}=d^{4d}
\left(\prod_{\texttt{v}\in\mathcal{V}(e)} \phi_{\texttt{v}}^{T^+} \phi_{\texttt{v}}^{T^-}\right)^{2},
\end{equation}
where, for $\texttt{v} \in \mathcal{V}(e)$, $\phi_{\texttt{v}}^{T^{\pm}}$ denotes the barycentric coordinates of $T^\pm$, which are understood as functions over $\mathcal{N}_{e}$. 
Important properties of this bubble function are: $\psi_{e} \in \mathbb{P}_{4d}(\mathcal{N}_{e})$, $\psi_{e} \in C^2(\mathcal{N}_{e})$, and $\psi_{e} = 0$ on $\partial \mathcal{N}_{e}$. 
In addition, $\nabla \psi_{e} = 0 \textrm{ on } \partial \mathcal{N}_{e}$ and $\llbracket \nabla \psi_e\cdot\mathbf{n}\rrbracket= 0$ on $e$.

Given $T\in\mathcal{T}_{\ell}$, we let $\Pi_{T} : L^2(T) \to \mathbb{P}_{0}(T)$ be the orthogonal projection operator into constant functions over $T$, i.e., $\Pi_{T}v:=\tfrac{1}{|T|}\int_{T}v(x)\dx$ for all $v\in L^{2}(T)$.

\begin{lemma}[local efficiency of ${\eta}_{st}$]\label{lemma:efficiency_est_2}
Let $\bar{u}\in U_{ad}$ be a solution to problem \eqref{def:weak_ocp}--\eqref{eq:weak_st_eq} with $\bar{y}$ being its associated optimal state. 
Let $\bar{\mathfrak{u}}_{\ell}\in U_{ad}$ be a solution to the semidiscrete problem with $\bar{y}_{\ell}$ being the corresponding discrete state variable. 
Assume that $f(\cdot,y)$ is globally Lipschitz with respect to $y\in \mathbb{R}$.
Then, for $T\in\mathcal{T}_{\ell}$, the local error indicator $\eta_{st,T}$, defined as in \eqref{def:state_indicator}, satisfies
\begin{align*}
\eta_{st,T}^2
\lesssim 
(1 + h_T^4)\|\bar{y} - \bar{y}_{\ell}\|_{\mathcal{N}_T}^2
+
h_{T}^{4 - d}\|\bar{u}-\bar{\mathfrak{u}}_{\ell}\|_{L^1(\mathcal{N}_T)}^2
+
\sum_{T'\in\mathcal{N}_{T}}h_{T}^{4}\|(1 - \Pi_{T})(\bar{\mathfrak{u}}_{\ell} - f(\cdot,\bar{y}_{\ell}))\|_{T'}^2,
\end{align*}
where $\mathcal{N}_T$ is defined as in \eqref{def:patch} and the hidden constant is independent of continuous and discrete optimal variables, the size of the elements in the mesh $\mathcal{T}_{\ell}$, and $\#\mathcal{T}_{\ell}$.
\end{lemma}
\begin{proof} Let $v \in H_0^{1}(\Omega)$ be such that $v|_T\in C^2(T)$ for all $T\in \mathcal{T}_{\ell}$. 
Choosing $v$ as a test function in \eqref{eq:weak_st_eq} and applying elementwise integration by parts we obtain
\begin{equation*}
( \nabla (\bar{y} - \bar{y}_{\ell}), \nabla v)_{\Omega} + (f(\cdot,\bar{y}) - f(\cdot,\bar{y}_{\ell}),v)_{\Omega} -  (\bar{u} - \bar{\mathfrak{u}}_{\ell},v)_{\Omega} 
=
\sum_{T\in \mathcal{T}_{\ell}} ( \bar{\mathfrak{u}}_{\ell} - f(\cdot,\bar{y}_{\ell}), v )_{T}  - \sum_{e\in\mathcal{E}_{\ell}}(\llbracket \nabla \bar{y}_{\ell}\cdot \mathbf{n} \rrbracket,v)_{e}.
\end{equation*}
Elementwise integration by parts also yields
\begin{equation*}
 (\nabla (\bar{y} - \bar{y}_{\ell}), \nabla v)_{\Omega} = \sum_{e\in\mathcal{E}_{\ell}} (\llbracket \nabla v\cdot \mathbf{n} \rrbracket, \bar{y}-\bar{y}_{\ell})_{e} - \sum_{T\in \mathcal{T}_{\ell}} (\bar{y}-\bar{y}_\ell, \Delta v)_{T}.
\end{equation*}
Combining both identities we obtain, for any $v \in H_0^{1}(\Omega)$ such that $v|_T\in C^2(T)$ for all $T\in \mathcal{T}_{\ell}$, the identity
\begin{align}\label{eq:error_eq_st}
&\sum_{e\in\mathcal{E}_{\ell}} (\llbracket \nabla v\cdot \mathbf{n} \rrbracket, \bar{y}-\bar{y}_\ell)_{e}  - \sum_{T\in \mathcal{T}_{\ell}} (\bar{y}-\bar{y}_\ell, \Delta v)_{T}  + (f(\cdot,\bar{y}) - f(\cdot,\bar{y}_{\ell}),v)_{\Omega} -  (\bar{u} - \bar{\mathfrak{u}}_{\ell},v)_{\Omega} \\
& =
\sum_{T\in \mathcal{T}_{\ell}}\left[ (\Pi_{T}(\bar{\mathfrak{u}}_{\ell} - f(\cdot,\bar{y}_{\ell})), v )_{T} + ((1 - \Pi_{T})(\bar{\mathfrak{u}}_{\ell} - f(\cdot,\bar{y}_{\ell})), v )_{T} \right]  - \sum_{e\in\mathcal{E}_{\ell}} ( \llbracket \nabla \bar{y}_\ell\cdot \mathbf{n} \rrbracket, v)_{e}. \nonumber
\end{align}

We now proceed in two steps.

\underline{Step 1.} (estimation of $h_{T}^{4}\|\bar{\mathfrak{u}}_{\ell} - f(\cdot,\bar{y}_{\ell})\|_{T}^{2}$) Let $T\in \mathcal{T}_{\ell}$. 
An application of the triangle inequality gives
\begin{align}\label{eq:triangle_res_st}
h_{T}^{4}\|\bar{\mathfrak{u}}_{\ell} - f(\cdot,\bar{y}_{\ell})\|_{T}^{2}
\lesssim
h_{T}^{4}\|\Pi_{T}(\bar{\mathfrak{u}}_{\ell} - f(\cdot,\bar{y}_{\ell}))\|_{T}^{2} + h_{T}^{4}\|(1 - \Pi_{T})(\bar{\mathfrak{u}}_{\ell} - f(\cdot,\bar{y}_{\ell}))\|_{T}^2.
\end{align}
We concentrate on the first term on the right-hand side of \eqref{eq:triangle_res_st}.
Choose $v = \varphi_{T}^{2}\Pi_{T}(\bar{\mathfrak{u}}_{\ell} - f(\cdot,\bar{y}_{\ell}))$ in \eqref{eq:error_eq_st}. 
Utilizing that $\nabla (\varphi^{2}\Pi_{T}(\bar{\mathfrak{u}}_{\ell} - f(\cdot,\bar{y}_{\ell})))=0$ on $\partial T$, the inverse estimate $\|\varphi_{T}^{2}\Pi_{T}(\bar{\mathfrak{u}}_{\ell} - f(\cdot,\bar{y}_{\ell}))\|_{L^{\infty}(T)} \lesssim h_{T}^{-\frac{d}{2}}\|\varphi_{T}^{2}\Pi_{T}(\bar{\mathfrak{u}}_{\ell} - f(\cdot,\bar{y}_{\ell}))\|_T$, and standard properties of the interior bubble function $\varphi_{T}$ we arrive at
\begin{align*}
\|\Pi_T(\bar{\mathfrak{u}}_{\ell} - f(\cdot,\bar{y}_{\ell}))\|_{T}^{2}
\lesssim ~ & 
\|\bar{y}-\bar{y}_\ell\|_{T}\|\Delta (\varphi_{T}^{2}\Pi_{T}(\bar{\mathfrak{u}}_{\ell} - f(\cdot,\bar{y}_{\ell})))\|_T  + \big(\|f(\cdot,\bar{y}) - f(\cdot,\bar{y}_{\ell})\|_{T} \\
& + h_{T}^{-\frac{d}{2}}\|\bar{u} - \bar{\mathfrak{u}}_{\ell}\|_{L^{1}(T)} + \|(1 - \Pi_{T})(\bar{\mathfrak{u}}_{\ell} - f(\cdot,\bar{y}_{\ell}))\|_{T}\big)\|\Pi_T(\bar{\mathfrak{u}}_{\ell} - f(\cdot,\bar{y}_{\ell}))\|_T.
\end{align*}
The fact that $\Delta (\varphi_{T}^{2} \Pi_{T}(\bar{\mathfrak{u}}_{\ell} - f(\cdot,\bar{y}_{\ell}))) = \Pi_{T}(\bar{\mathfrak{u}}_{\ell} - f(\cdot,\bar{y}_{\ell}))\Delta \varphi_{T}^{2}$ combined with properties of $\varphi_{T}$ implies that $\|\Delta (\varphi_{T}^{2}\Pi_{T}(\bar{\mathfrak{u}}_{\ell} - f(\cdot,\bar{y}_{\ell})))\|_{T} \lesssim h_{T}^{-2}\|\Pi_{T}(\bar{\mathfrak{u}}_{\ell} - f(\cdot,\bar{y}_{\ell}))\|_{T}$. 
Moreover, since $f(\cdot,y)$ is globally Lipschitz, we have $\|f(\cdot,\bar{y}) - f(\cdot,\bar{y}_{\ell})\|_T \lesssim \|\bar{y} - \bar{y}_{\ell}\|_T$.
Consequently, it follows that
\begin{equation}\label{eq:res_T_st}
h_{T}^{4}\|\Pi_{T}(\bar{\mathfrak{u}}_{\ell} - f(\cdot,\bar{y}_{\ell}))\|_{T}^{2}
\lesssim (1 + h_T^4)\|\bar{y}-\bar{y}_\ell\|_{T}^2 + h_{T}^{4-d}\|\bar{u} - \bar{\mathfrak{u}}_{\ell}\|_{L^{1}(T)}^2 
 + h_T^4\|(1- \Pi_{T})(\bar{\mathfrak{u}}_{\ell} - f(\cdot,\bar{y}_{\ell}))\|_{T}^2.
\end{equation}
The use of estimate \eqref{eq:res_T_st} in \eqref{eq:triangle_res_st} yields the desired bound.

\underline{Step 2.} (estimation of $h_{T}^{3}\|\llbracket \nabla \bar{y}_\ell\cdot \mathbf{n} \rrbracket\|_{e}^{2}$) 
Let $T\in \mathcal{T}_{\ell}$ and $e\in \mathcal{E}_{T}$.  
We note that $\llbracket \nabla \bar{y}_\ell\cdot \mathbf{n} \rrbracket$ is only defined on $e$. 
Since this jump term is constant, we can easily extend it to the patch $\mathcal{N}_{e}$ by using its value.
From now on, we shall make no distinction between the jump term and its extension.

We consider the bubble function $\psi_{e}$, defined in \eqref{def:aux_bubble}, and choose $v = \llbracket\nabla \bar{y}_\ell\cdot\mathbf{n}\rrbracket\psi_{e}$ in \eqref{eq:error_eq_st}. 
Then, using that $\llbracket\nabla \bar{y}_\ell\cdot\mathbf{n}\rrbracket\in \mathbb{R}$, properties of $\psi_{e}$, and the basic estimate $\|\llbracket\nabla \bar{y}_\ell\cdot\mathbf{n}\rrbracket\psi_{e}\|_{T} \lesssim h_{T}^{\frac{1}{2}}\| \llbracket\nabla \bar{y}_\ell\cdot\mathbf{n}\rrbracket \|_{e}$ we arrive at
\begin{align*}
 \|\llbracket \nabla \bar{y}_\ell\cdot \mathbf{n}\rrbracket \psi_e^{\frac{1}{2}}\|_{e}^{2}
\lesssim &
\sum_{T' \in \mathcal{N}_{e}} \big( h_{T'}^{-2}\|\bar{y}-\bar{y}_{\ell}\|_{T'} + \|\bar{\mathfrak{u}}_{\ell}  - f(\cdot,\bar{y}_{\ell})\|_{T'} \\
& \,  +  \|f(\cdot,\bar{y}) - f(\cdot,\bar{y}_{\ell})\|_{T} + h_{T'}^{-\frac{d}{2}}\|\bar{u} - \bar{\mathfrak{u}}_{\ell}\|_{L^1(T')} \big) h_{T}^{\frac{1}{2}}\| \llbracket\nabla \bar{y}_\ell\cdot\mathbf{n}\rrbracket \|_{e}.
\end{align*}
We then apply standard bubble functions arguments, the shape regularity property of the family $\{ \mathcal{T}_\ell \}$, and $\|f(\cdot,\bar{y}) - f(\cdot,\bar{y}_{\ell})\|_T \lesssim \|\bar{y} - \bar{y}_{\ell}\|_T$ to obtain
\begin{align*}
h_T^{\frac{3}{2}}\| \llbracket \nabla \bar{y}_\ell \cdot\mathbf{n}\rrbracket\|_{e}
\lesssim
\sum_{T'\in\mathcal{N}_e}  
\big(
(1+h_T^2)\|\bar{y}-\bar{y}_{\ell}\|_{T'} + h_T^2\|\bar{\mathfrak{u}}_{\ell} - f(\cdot,\bar{y}_{\ell})\|_{T'} + h_{T}^{2-\frac{d}{2}}\|\bar{u} - \bar{\mathfrak{u}}_{\ell}\|_{L^{1}(T')}
\big).
\end{align*}
A combination of this bound with the estimate proved for $h_T^4\|\bar{\mathfrak{u}}_{\ell} - f(\cdot,\bar{y}_{\ell})\|_{T}^2$, in Step 1, ends the proof.
\end{proof}

We now continue with the study of local efficiency properties of the estimator ${\eta}_{adj}$, defined in \eqref{def:adjoint_indicator}.

\begin{lemma}[local efficiency of $\eta_{adj}$]\label{lemma:efficiency_adj}
Let $\bar{u}\in U_{ad}$ be a solution to problem \eqref{def:weak_ocp}--\eqref{eq:weak_st_eq} with $\bar{y}$ and $\bar{p}$ being the corresponding optimal state and adjoint state, respectively. 
Let $\bar{\mathfrak{u}}_{\ell}\in \mathbb{U}_{ad}$ be a solution to the semidiscrete problem with $\bar{y}_{\ell}$ and $\bar{p}_{\ell}$ being the corresponding discrete state and adjoint state variables, respectively.  
If $\frac{\partial f}{\partial y}(\cdot, y)\in L^{2}(\Omega)$ for all $y\in H_0^1(\Omega)$ and $\frac{\partial f}{\partial y}(\cdot, y)$ is globally Lipschitz with respect to $y\in \mathbb{R}$,
then the local error indicator $\eta_{adj,T}$ \textnormal{(}$T\in\mathcal{T}_{\ell}$\textnormal{)}, defined as in \eqref{def:adjoint_indicator}, satisfies 
\begin{equation*}
\eta_{adj,T}^2
\lesssim
    (1 + h_{T}^{4 - d})\|\bar{p}-\bar{p}_{\ell}\|_{L^\infty(\mathcal{N}_T)}^2 + h_{T}^{4 - d}\|\bar{y}-\bar{y}_{\ell}\|_{\mathcal{N}_T}^2 + \sum_{T'\in\mathcal{N}_{T}}h_{T}^{4-d}\|(1 - \Pi_{T})(y_{\Omega} + \tfrac{\partial f}{\partial y}(\cdot, \bar{y}_{\ell})\bar{p}_{\ell})\|_{T'}^{2},
\end{equation*}
where $\mathcal{N}_T$ is defined as in \eqref{def:patch} and the hidden constant is independent of continuous and discrete optimal variables, the size of the elements in the mesh $\mathcal{T}_{\ell}$, and $\#\mathcal{T}_{\ell}$.
\end{lemma}
\begin{proof}
Let $v \in H_0^{1}(\Omega)$ such that $v|_T\in C^2(T)$ ($T\in \mathcal{T}_{\ell}$).
Similar arguments to those that lead to \eqref{eq:error_eq_st} yield
\begin{align}\label{eq:error_eq_adj}
\sum_{e\in\mathcal{E}_{\ell}} (\llbracket \nabla v\cdot \mathbf{n} \rrbracket, \bar{p}-\bar{p}_{\ell})_{e} - \sum_{T\in \mathcal{T}_{\ell}} (\bar{p}-\bar{p}_{\ell}, & \, \Delta v)_{T} +\left(\left[\tfrac{\partial f}{\partial y}(\cdot, \bar{y}) - \tfrac{\partial f}{\partial y}(\cdot, \bar{y}_{\ell})\right]\bar{p},v\right)_{\Omega} \\
 + \left( \tfrac{\partial f}{\partial y}(\cdot, \bar{y}_{\ell})(\bar{p} - \bar{p}_{\ell}), v\right)_{\Omega} -  (\bar{y} - \bar{y}_{\ell},v)_{\Omega} =  &
\sum_{T\in \mathcal{T}_{\ell}} \left( \bar{y}_{\ell} - \Pi_{T}\left(y_{\Omega} + \tfrac{\partial f}{\partial y}(\cdot, \bar{y}_{\ell})\bar{p}_{\ell}\right), v \right)_{T}  - \sum_{e\in\mathcal{E}_{\ell}} ( \llbracket \nabla \bar{p}_{\ell}\cdot \mathbf{n} \rrbracket, v)_{e} \nonumber\\
& + \sum_{T\in \mathcal{T}_{\ell}}( (\Pi_{T} - 1)(y_{\Omega} + \tfrac{\partial f}{\partial y}(\cdot, \bar{y}_{\ell})\bar{p}_{\ell}),v)_{T}. \nonumber
\end{align}
We proceed on the basis of two steps.

\underline{Step 1.} (estimation of $h_T^{4-d}\| \bar{y}_{\ell} - y_{\Omega} - \tfrac{\partial f}{\partial y}(\cdot, \bar{y}_{\ell})\bar{p}_{\ell}\|_{T}^{2}$) Let $T\in\mathcal{T}_{\ell}$. 
To simplify the presentation of the material, we define $R_{T}:= \left.\left(\bar{y}_{\ell} - \Pi_{T}\left(y_{\Omega} + \tfrac{\partial f}{\partial y}(\cdot, \bar{y}_{\ell})\bar{p}_{\ell}\right)\right)\right|_{T}^{}$.
We recall that $\Pi_{T}$ denotes the orthogonal projection operator into constant functions over $T$. 
The triangle inequality implies 
\begin{equation}\label{eq:triangle_st_2}
h_{T}^{4-d}\|\bar{y}_{\ell} - y_{\Omega} - \tfrac{\partial f}{\partial y}(\cdot, \bar{y}_{\ell})\bar{p}_{\ell}\|_{T}^{2} 
\lesssim
h_{T}^{4-d}\|R_{T}\|_{T}^{2}  
+
h_{T}^{4-d}\| (1 - \Pi_T)(y_{\Omega} + \tfrac{\partial f}{\partial y}(\cdot, \bar{y}_{\ell})\bar{p}_{\ell})\|_{T}^{2}.
\end{equation}
Let us concentrate on $h_{T}^{4-d}\|R_{T}\|_{T}^{2}$.
Invoke the interior bubble function $\varphi_{T}$ and take $v=\varphi_{T}^{2}R_{T}$ in \eqref{eq:error_eq_adj}.
Then, the fact that $\tfrac{\partial f}{\partial y}(\cdot,\bar{y}_{\ell})\in L^2(\Omega)$ and $\bar{p}\in L^{\infty}(\Omega)$, and the use of the inequality $\|\Delta(\varphi_{T}^{2}R_{T})\|_{L^{1}(T)} 
\lesssim 
h_{T}^{\frac{d}{2}-2}\|R_{T}\|_{T}
$, which stems from \cite[Lemma 4.5.3]{MR2373954}, give
\begin{equation*}
\|R_{T}\|_{T} 
\lesssim
\|(1 - \Pi_{T})(y_{\Omega} + \tfrac{\partial f}{\partial y}(\cdot, \bar{y}_{\ell})\bar{p}_{\ell})\|_{T} + \|\bar{y} - \bar{y}_{\ell}\|_{T} + \|\tfrac{\partial f}{\partial y}(\cdot, \bar{y}) - \tfrac{\partial f}{\partial y}(\cdot, \bar{y}_{\ell})\|_{L^{2}(T)} + \big(1 + h_{T}^{\frac{d}{2}-2}\big)\|\bar{p} - \bar{p}_{\ell}\|_{L^{\infty}(T)}.
\end{equation*}
Now, using that $\frac{\partial f}{\partial y}(\cdot, y)$ is globally Lipschitz we obtain that 
\begin{align*}
h_{T}^{4-d}\|R_T\|_{T}^{2}
\lesssim 
h_{T}^{4-d}\|(1 - \Pi_{T})(y_{\Omega} + \tfrac{\partial f}{\partial y}(\cdot, \bar{y}_{\ell})\bar{p}_{\ell})\|_{T}^2 + h_{T}^{4-d}\|\bar{y} - \bar{y}_{\ell}\|_{T}^2 + (h_{T}^{4 - d} + 1)\|\bar{p} - \bar{p}_{\ell}\|_{L^{\infty}(T)}^{2}. \nonumber
\end{align*}
This estimate, together with \eqref{eq:triangle_st_2}, leads to the desired bound for the element residual term.

\underline{Step 2.} (estimation of $h_{T}^{2}\|\llbracket \nabla \bar{p}_{\ell}\cdot \mathbf{n} \rrbracket\|_{L^{\infty}(e)}^{2}$)  
Let $T\in \mathcal{T}_{\ell}$ and $e\in \mathcal{E}_{T}$. 
The fact that $\llbracket \nabla \bar{p}_{\ell}\cdot \mathbf{n} \rrbracket\in \mathbb{R}$ implies 
\begin{equation*}
\|\llbracket \nabla \bar{p}_{\ell}\cdot \mathbf{n} \rrbracket\|_{L^{\infty}(e)}^{2} 
=
|\llbracket \nabla \bar{p}_{\ell}\cdot \mathbf{n} \rrbracket|^{2}
=
|e|^{-1}
\|\llbracket \nabla \bar{p}_{\ell}\cdot \mathbf{n} \rrbracket\|_{e}^{2},
\end{equation*}
where $|e|$ denotes the measure of $e$. 
In view of the shape regularity of the mesh $\mathcal{T}_{\ell}$ we have that $|e|\approx h_T^{d-1}$ and consequently $h_T^{2}\|\llbracket \nabla \bar{p}_{\ell}\cdot \mathbf{n} \rrbracket\|_{L^{\infty}(e)}^{2}  \approx
h_T^{3-d} \|\llbracket \nabla \bar{p}_{\ell}\cdot \mathbf{n} \rrbracket\|_{e}^{2}$. 
In what follows, we estimate $h_T^{3-d} \|\llbracket \nabla \bar{p}_{\ell}\cdot \mathbf{n} \rrbracket\|_{e}^{2}$.

We proceed as in Lemma \ref{lemma:efficiency_est_2} and extend the jump term $\llbracket \nabla \bar{p}_{\ell}\cdot \mathbf{n} \rrbracket$ to the patch $\mathcal{N}_{e}$, making no distinction between the jump term and its extension.
Then, we choose $v = \llbracket \nabla \bar{p}_{\ell}\cdot \mathbf{n} \rrbracket\psi_{e}$ in \eqref{eq:error_eq_adj} and use that $\llbracket \nabla \bar{p}_{\ell}\cdot \mathbf{n} \rrbracket\in \mathbb{R}$, $\psi_e \in H^2_0(\mathcal{N}_e)$, and $\llbracket\nabla \psi_{e} \cdot \mathbf{n}\rrbracket = 0$. 
From these arguments we derive 
\begin{align*}
\|\llbracket \nabla \bar{p}_{\ell}\cdot \mathbf{n} \rrbracket\|_{e}^{2}
\lesssim &
\sum_{T'\in \mathcal{N}_{e}} \big(\|\bar{y}_{\ell} - y_{\Omega} - \tfrac{\partial f}{\partial y}(\cdot, \bar{y}_{\ell})\bar{p}_{\ell}\|_{T} + \|\bar{y} - \bar{y}_{\ell}\|_{T} \\
& \, + \|\tfrac{\partial f}{\partial y}(\cdot, \bar{y}) - \tfrac{\partial f}{\partial y}(\cdot, \bar{y}_{\ell})\|_{L^{2}(T)} + \big(1 + h_{T}^{\frac{d}{2}-2}\big)\|\bar{p} - \bar{p}_{\ell}\|_{L^{\infty}(T)}\big)\|\psi_e\llbracket \nabla \bar{p}_{\ell}\cdot \mathbf{n} \rrbracket\|_{T'}.
\end{align*}
Using the latter, the estimate $\|\psi_e\llbracket \nabla \bar{p}_{\ell}\cdot \mathbf{n} \rrbracket\|_{T'}\lesssim h_{T}^{\frac{1}{2}} \|\llbracket \nabla \bar{p}_{\ell}\cdot \mathbf{n} \rrbracket\|_{e}$, and the Lipschitz property of $\frac{\partial f}{\partial y}(\cdot, y)$, we arrive at
\begin{align*}
h_T^{3-d}\|\llbracket \nabla \bar{p}_{\ell}\cdot \mathbf{n} \rrbracket\|_{e}^{2}
\lesssim 
\sum_{T'\in \mathcal{N}_{e}} \big(h_{T}^{4-d}\|\bar{y}_{\ell} - y_{\Omega} - \tfrac{\partial f}{\partial y}(\cdot, \bar{y}_{\ell})\bar{p}_{\ell}\|_{T}^{2} + h_{T}^{4-d}\|\bar{y} - \bar{y}_{\ell}\|_{T}^{2} + \big(h_{T}^{4-d} + 1\big)\|\bar{p} - \bar{p}_{\ell}\|_{L^{\infty}(T)}^{2}\big).
\end{align*}
We conclude the proof by using the bound obtained for $h_{T}^{4-d}\|\bar{y}_{\ell} - y_{\Omega} - \tfrac{\partial f}{\partial y}(\cdot, \bar{y}_{\ell})\bar{p}_{\ell}\|_{T}^{2}$ in the previous step.
\end{proof}

The next result is an immediate consequence of Lemmas \ref{lemma:efficiency_est_2} and \ref{lemma:efficiency_adj}.

\begin{theorem}[local efficiency]\label{thm:efficiency_est_ocp}
In the framework of Lemma \ref{lemma:efficiency_adj} we have, for $T\in\mathcal{T}_{\ell}$, that
\begin{align*}
\eta_{st,T}^2 + \eta_{adj,T}^2
\lesssim  & ~
\|\bar{y} - \bar{y}_{\ell}\|_{\mathcal{N}_T}^2
+ \|\bar{p}-\bar{p}_{\ell}\|_{L^\infty(\mathcal{N}_T)}^2 + h_{T}^{4 - d}\|\bar{u}-\bar{\mathfrak{u}}_{\ell}\|_{L^1(\mathcal{N}_T)}^2 \\
& \, + 
\sum_{T'\in\mathcal{N}_{T}}\left(h_{T}^{4}\|(1 - \Pi_{T})(\bar{\mathfrak{u}}_{\ell} - f(\cdot,\bar{y}_{\ell}))\|_{T'}^2 + h_{T}^{4-d}\|(1 - \Pi_{T})(y_{\Omega} + \tfrac{\partial f}{\partial y}(\cdot, \bar{y}_{\ell})\bar{p}_{\ell})\|_{T'}^{2}\right),
\end{align*}
where $\mathcal{N}_T$ is defined as in \eqref{def:patch} and the hidden constant is independent of continuous and discrete optimal variables, the size of the elements in the mesh $\mathcal{T}_{\ell}$, and $\#\mathcal{T}_{\ell}$.
\end{theorem}

\section{Numerical examples}\label{sec:num_ex}
In the present section we perform three numerical experiments in two-dimensional domains. 
The results support our theoretical results and show the performance of the error estimator $\eta_{ocp}$, defined in \eqref{def:total_estimator}.

The numerical examples were performed using a code that we developed in \texttt{MATLAB$^\copyright$ (R2024a)}. 
All system matrices as well as the term $(\bar{\mathfrak{u}}_{\ell},v_{\ell})_{\Omega}$ are computed exactly, whereas approximation errors, error indicators, and the remaining right-hand sides are computed using a quadrature formula that is exact for polynomials of degree $19$.
We also have incorporated an extra forcing term $\mathfrak{f}\in L^2(\Omega)$ in the state equation, with the aim of simplifying the construction of exact optimal solutions. 
With this modification, the right-hand side of the state equation now reads: $(\mathfrak{f} + \bar{u}, v)_{\Omega}$.
Additionally, in section \ref{sec:ex_2} below, we go beyond the presented theory and perform numerical experiments with a non-convex domain. 

For a given partition $\mathcal{T}_{\ell}$, we seek $\bar{y}_\ell \in \mathbb{V}_{\ell}$,  $\bar{p}_\ell \in \mathbb{V}_{\ell}$, and $\bar{\mathfrak{u}}_\ell \in U_{ad}$ that solve \eqref{eq:discrete_pde_semi}, \eqref{eq:discrete_adj_eq}, and \eqref{eq:pointwise_charac_discrete}.
We solve such a nonlinear system of equations using the fixed-point solution technique devised in \cite[Section 4]{MR2891922}. 
Once a discrete solution is obtained, we compute the error indicator 
\begin{align}\label{def:total_indicator}
\eta_{ocp,T} := \left(\eta_{st,T}^{2\gamma} + \eta_{adj,T}^{2\gamma}\right)^{\frac{1}{2}} \qquad (\gamma \in (0,1])
\end{align}
to drive the adaptive procedure described in Algorithm \ref{Algorithm}.
We define the total number of degrees of freedom $\rm{Ndofs} = 2\:\rm{dim}\mathbb{V}_{\ell}$ and the effectivity index 
\begin{equation*}
\mathcal{I}_{eff}:= \frac{\eta_{ocp}}{(\|\bar{u} - \bar{\mathfrak{u}}_{\ell}\|_{L^1(\Omega)}^{2} + \|\bar{y} - \bar{y}_{\ell}\|_{\Omega}^{2} + \|\bar{p} - \bar{p}_{\ell}\|_{L^{\infty}(\Omega)}^{2})^{\frac{1}{2}}}.
\end{equation*}

\begin{algorithm}[ht]
\caption{\textbf{Adaptive fixed-point algorithm.}}
\label{Algorithm}
\SetKwInput{set}{Set}
\SetKwInput{ase}{Fixed-point iteration}
\SetKwInput{al}{A posteriori error estimation}
\SetKwInput{Input}{Input}
\SetAlgoLined
\Input{Initial mesh $\mathcal{T}_0$, initial control $\mathfrak{u}_{0}$, constraints $a$ and $b$, desired state $y_\Omega$, and right-hand side $\mathfrak{f}$.}
\set{$\ell=0$.}
\ase{}
Compute $[\bar{y}_{\ell},\bar{p}_{\ell},\bar{\mathfrak{u}}_{\ell}]=\textbf{Fixed-point iteration}[{\mathcal{T}_{\ell}},\mathfrak{u}_{0},a,b,y_\Omega, \mathfrak{f}]$, which implements a similar fixed-point iteration to \cite[Section 4]{MR2891922}; in our case, we solve the state equation using the Newton method.
\\
\al 
\\
For each $T\in{\mathcal{T}_\ell}$ compute the local error indicator $\eta_{ocp,T}$ given in \eqref{def:total_indicator}.
\\
Mark an element $T$ for refinement if $\eta_{ocp,T}> \displaystyle\frac{1}{2}\max_{T'\in{\mathcal{T}_\ell}} \eta_{ocp,T'}$.
\\
From step $\mathbf{3}$, construct a new mesh, using a longest edge bisection algorithm. Set $\ell \leftarrow \ell + 1$, and go to step $\mathbf{1}$.
\end{algorithm}
\normalsize

\subsection{Exact solution on convex domain}\label{sec:ex_1}
We take the example from \cite[Section 3.3]{MR3095657} (see also \cite[Section 5.3]{MR3385653}) and set $\Omega:=(0,1)^{2}$, $a=-1$, $b=1$, $\gamma=1$, and take $\mathfrak{f}$ and $y_{\Omega}$ such that 
\begin{align*}
  &\bar{y}(x_{1},x_{2}) = 16x_{1}x_{2}(1-x_1)(1-x_2), \\ 
  & \bar{p}(x_{1},x_{2}) = -\sin(2\pi x_{1})\sin(2\pi x_{2}), \quad \bar{u}(x_{1},x_{2}) = -\mathrm{sign}(\bar{p}(x_{1},x_{2}))
\end{align*}
for $(x_{1},x_{2})\in \Omega$.
We consider two different choices of the nonlinear function, namely $f(\cdot,y)=y^3$ and $f(\cdot,y)=\arctan(y)$.

Figures \ref{fig:ex_1_1}, \ref{fig:ex_1_2}, and \ref{fig:ex_1_3} show the results obtained for this example.
In Fig. \ref{fig:ex_1_1} we display experimental convergence rates for each contribution of the total error when uniform and adaptive refinements are considered, choosing $f(\cdot,y)=\arctan(y)$ as the nonlinear term.
The same information is shown in Fig. \ref{fig:ex_1_2}, but choosing the nonlinear term $f(\cdot,y)=y^3$ instead.
We observe that all the approximation errors obtained for both cases exhibit optimal convergence rates.
In Fig. \ref{fig:ex_1_3}, we present an approximate optimal control $\bar{\mathfrak{u}}_{\ell}$ and its associated adaptively refined meshes obtained after 5 and 10 iterations. 
We observe that, even when the adaptive refinement is not necessarily concentrated near the discrete switching set, this set seems to converge to the continuous one when the total number of degrees of freedom increases. 
We observe the classical bang-bang structure of the control.

\begin{figure}
  \begin{tikzpicture}
  \begin{groupplot}[group style={group size= 3 by 1},width=0.4\textwidth,cycle list/Dark2-6,
                      cycle multiindex* list={
                          mark list*\nextlist
                          Dark2-6\nextlist},
                      every axis plot/.append style={ultra thick},
                      grid=major,
                      xlabel={Ndofs},]
         \nextgroupplot[title={Error contributions},ymode=log,xmode=log,
           legend entries={\hspace{-0.5cm}\tiny{$\|\bar{y}-\bar{y}_{\ell}\|_{\Omega}$},\tiny{$\|\bar{p}-\bar{p}_{\ell}\|_{L^{\infty}(\Omega)}$},\tiny{$\|\bar{u}-\bar{\mathfrak{u}}_{\ell}\|_{L^{1}(\Omega)}$}},
                      legend pos=north east]
                \addplot table [x=dofs,y=error_y] {data_for_plots/Ex_1/Errors_ex1_atan_unif.dat};
                \addplot table [x=dofs,y=error_p] {data_for_plots/Ex_1/Errors_ex1_atan_unif.dat};
                 \addplot table [x=dofs,y=error_u] {data_for_plots/Ex_1/Errors_ex1_atan_unif.dat};
                
             \logLogSlopeTriangleBelow{0.7}{0.25}{0.15}{1}{black}{{\small $1$}}
              \nextgroupplot[title={Error contributions},ymode=log,xmode=log,
           legend entries={\hspace{-0.5cm}\tiny{$\|\bar{y}-\bar{y}_{\ell}\|_{\Omega}$},\tiny{$\|\bar{p}-\bar{p}_{\ell}\|_{L^{\infty}(\Omega)}$},\tiny{$\|\bar{u}-\bar{\mathfrak{u}}_{\ell}\|_{L^{1}(\Omega)}$}},
                      legend pos=north east]
                \addplot table [x=dofs,y=error_y] {data_for_plots/Ex_1/Errors_ex1_atan_adap.dat};
                \addplot table [x=dofs,y=error_p] {data_for_plots/Ex_1/Errors_ex1_atan_adap.dat};
                \addplot table [x=dofs,y=error_u] {data_for_plots/Ex_1/Errors_ex1_atan_adap.dat};
                
                \logLogSlopeTriangle{0.8}{0.2}{0.45}{1}{black}{{\small $1$}};
                \logLogSlopeTriangleBelow{0.7}{0.2}{0.15}{1}{black}{{\small $1$}}
    \end{groupplot}
\end{tikzpicture}
  \caption{Experimental convergence rates for individual contributions of the total error with uniform (left) and adaptive (right) refinements for the problem from section \ref{sec:ex_1} with $f(\cdot,y)=\arctan(y)$.}
\label{fig:ex_1_1}
\end{figure}


\begin{figure}
  \begin{tikzpicture}
  \begin{groupplot}[group style={group size= 2 by 1},width=0.4\textwidth,cycle list/Dark2-6,
                      cycle multiindex* list={
                          mark list*\nextlist
                          Dark2-6\nextlist},
                      every axis plot/.append style={ultra thick},
                      grid=major,
                      xlabel={Ndofs},]
         \nextgroupplot[title={Error contributions},ymode=log,xmode=log,
           legend entries={\hspace{-0.5cm}\tiny{$\|\bar{y}-\bar{y}_{\ell}\|_{\Omega}$},\tiny{$\|\bar{p}-\bar{p}_{\ell}\|_{L^{\infty}(\Omega)}$},\tiny{$\|\bar{u}-\bar{\mathfrak{u}}_{\ell}\|_{L^{1}(\Omega)}$}},
                      legend pos=north east]
                \addplot table [x=dofs,y=error_y] {data_for_plots/Ex_1/Errors_ex1_y3_unif.dat};
                \addplot table [x=dofs,y=error_p] {data_for_plots/Ex_1/Errors_ex1_y3_unif.dat};
                \addplot table [x=dofs,y=error_u] {data_for_plots/Ex_1/Errors_ex1_y3_unif.dat};

                \logLogSlopeTriangleBelow{0.7}{0.25}{0.15}{1}{black}{{\small $1$}}
              \nextgroupplot[title={Error contributions},ymode=log,xmode=log,
           legend entries={\hspace{-0.5cm}\tiny{$\|\bar{y}-\bar{y}_{\ell}\|_{\Omega}$},\tiny{$\|\bar{p}-\bar{p}_{\ell}\|_{L^{\infty}(\Omega)}$},\tiny{$\|\bar{u}-\bar{\mathfrak{u}}_{\ell}\|_{L^{1}(\Omega)}$}},
                      legend pos=north east]
                \addplot table [x=dofs,y=error_y] {data_for_plots/Ex_1/Errors_ex1_y3_adap.dat};
                \addplot table [x=dofs,y=error_p] {data_for_plots/Ex_1/Errors_ex1_y3_adap.dat};
                \addplot table [x=dofs,y=error_u] {data_for_plots/Ex_1/Errors_ex1_y3_adap.dat};
                
                \logLogSlopeTriangle{0.8}{0.2}{0.45}{1}{black}{{\small $1$}};
                \logLogSlopeTriangleBelow{0.7}{0.2}{0.15}{1}{black}{{\small $1$}}
    \end{groupplot}
\end{tikzpicture}
  \caption{Experimental convergence rates for individual contributions of the total error with uniform (left) and adaptive (right) refinements for the problem from section \ref{sec:ex_1} with $f(\cdot,y)=y^3$.}
\label{fig:ex_1_2}
\end{figure}


\begin{figure}[!ht]
\begin{minipage}[c]{0.27\textwidth}\centering
\includegraphics[trim={0 0 0 0},clip,width=4.4cm,height=4.4cm,scale=0.30]{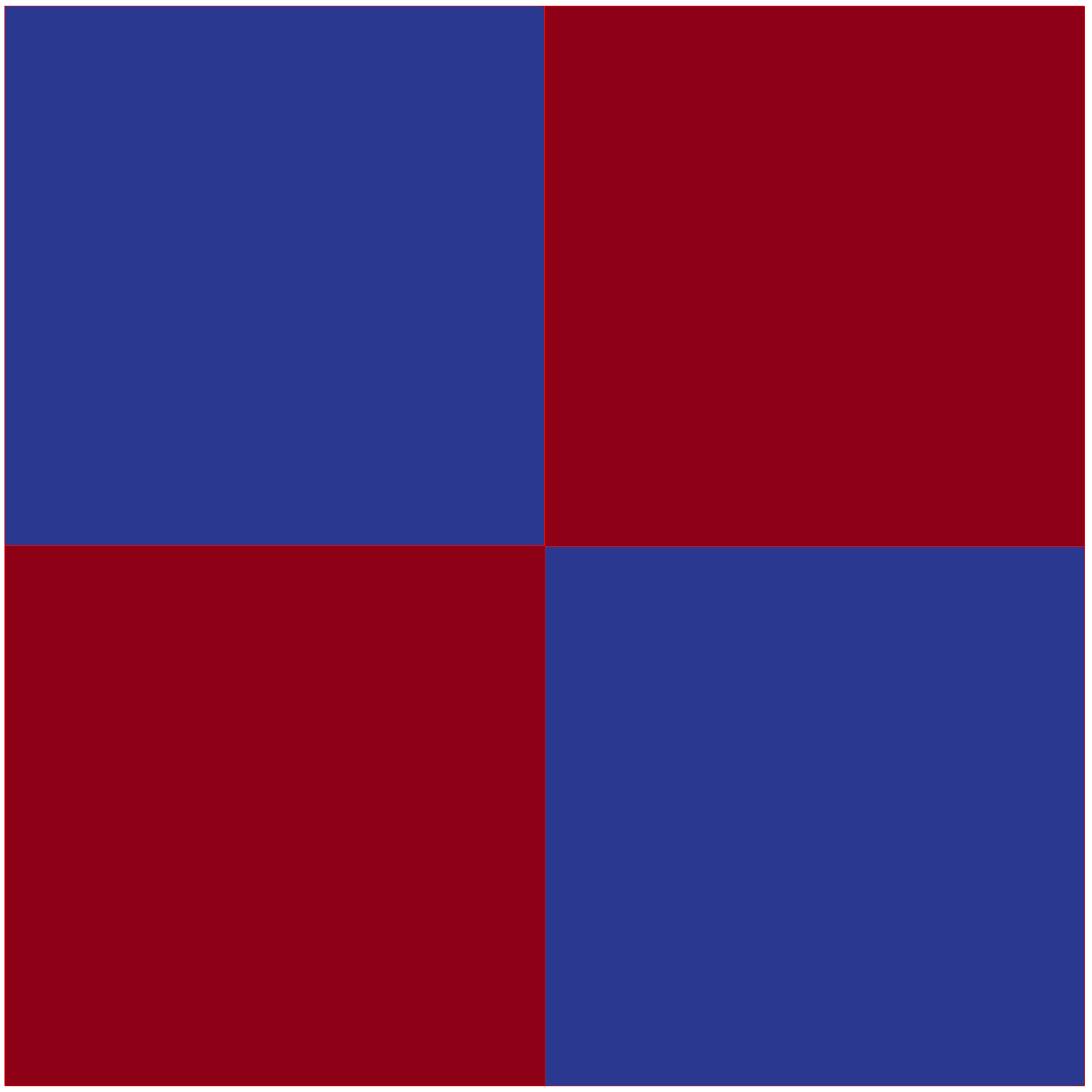}\\
\end{minipage}
\begin{minipage}[c]{0.27\textwidth}\centering
\includegraphics[trim={0 0 0 0},clip,width=4.4cm,height=4.4cm,scale=0.30]{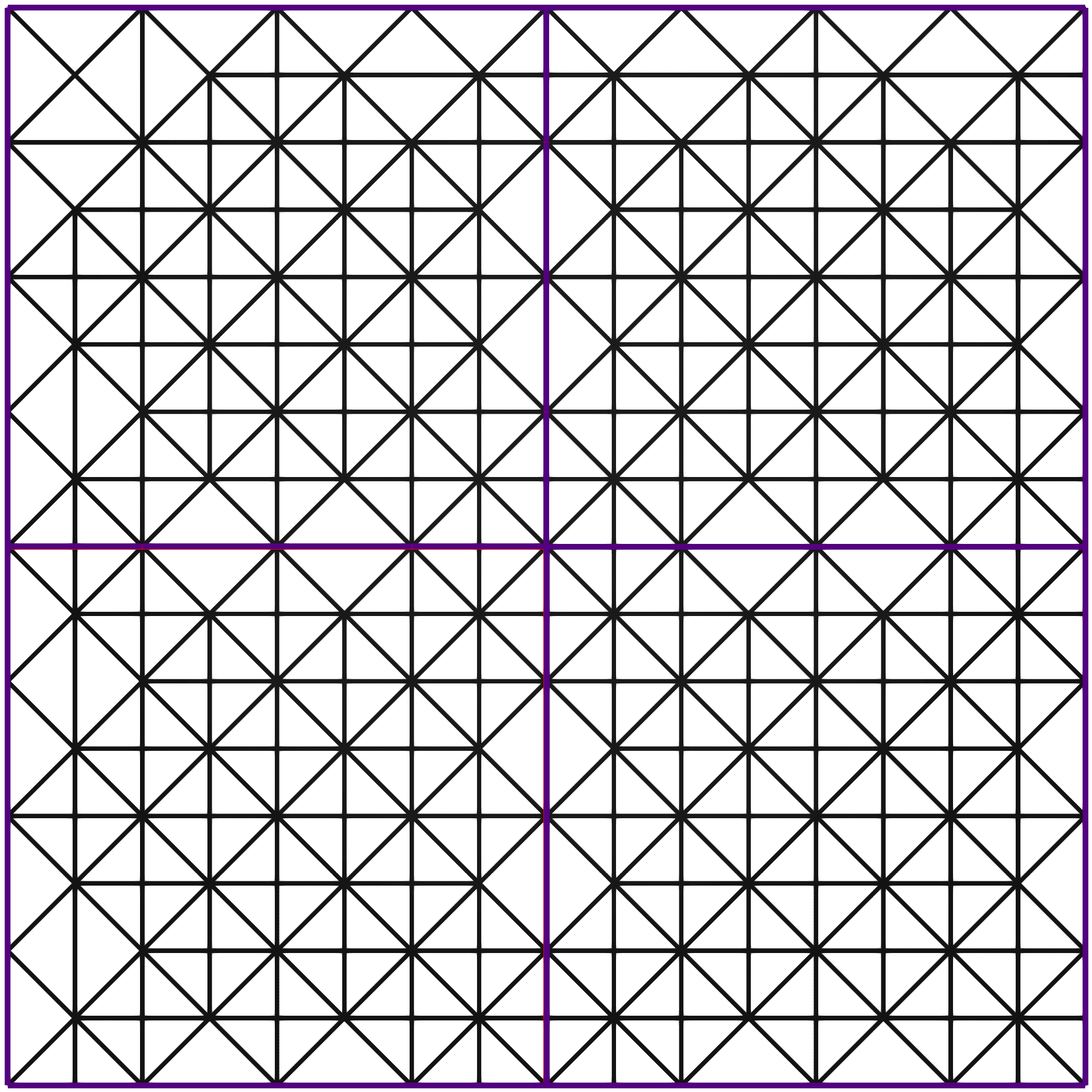}\\
\end{minipage}
\begin{minipage}[c]{0.27\textwidth}\centering
\includegraphics[trim={0 0 0 0},clip,width=4.4cm,height=4.4cm,scale=0.30]{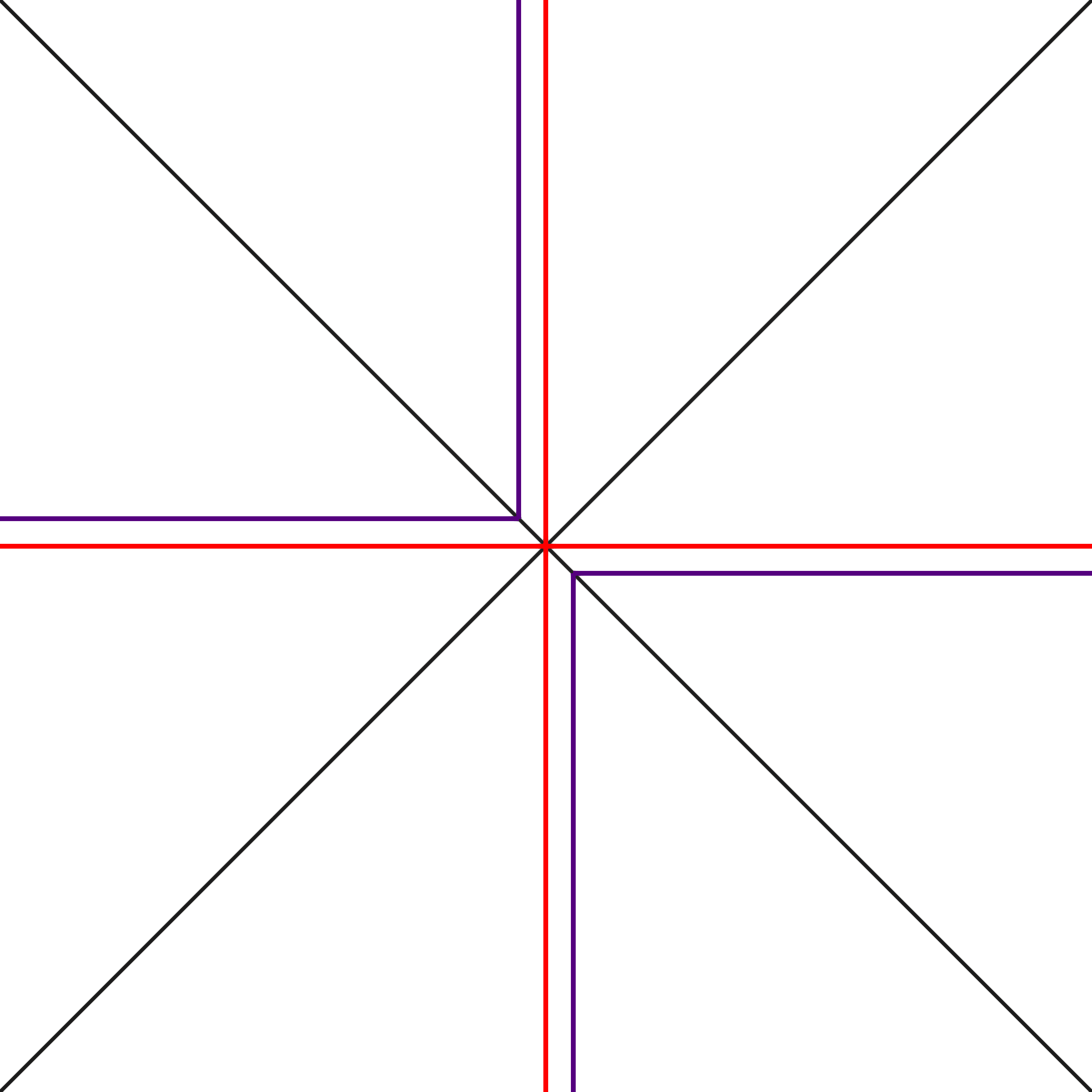}\\
\end{minipage}
\\
\begin{minipage}[c]{0.27\textwidth}\centering
\includegraphics[trim={0 0 0 0},clip,width=4.4cm,height=4.4cm,scale=0.30]{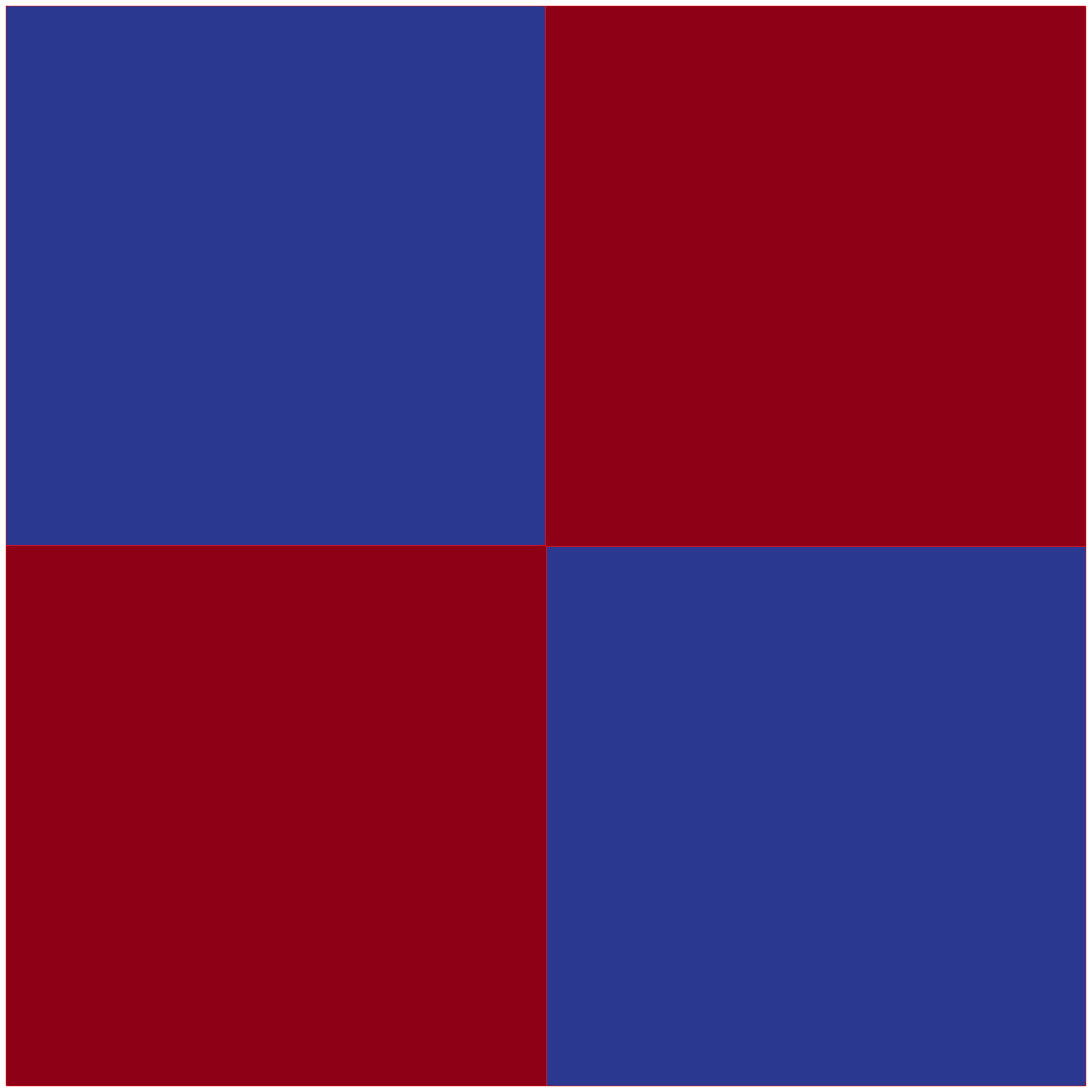}\\
\end{minipage}
\begin{minipage}[c]{0.27\textwidth}\centering
\includegraphics[trim={0 0 0 0},clip,width=4.4cm,height=4.4cm,scale=0.30]{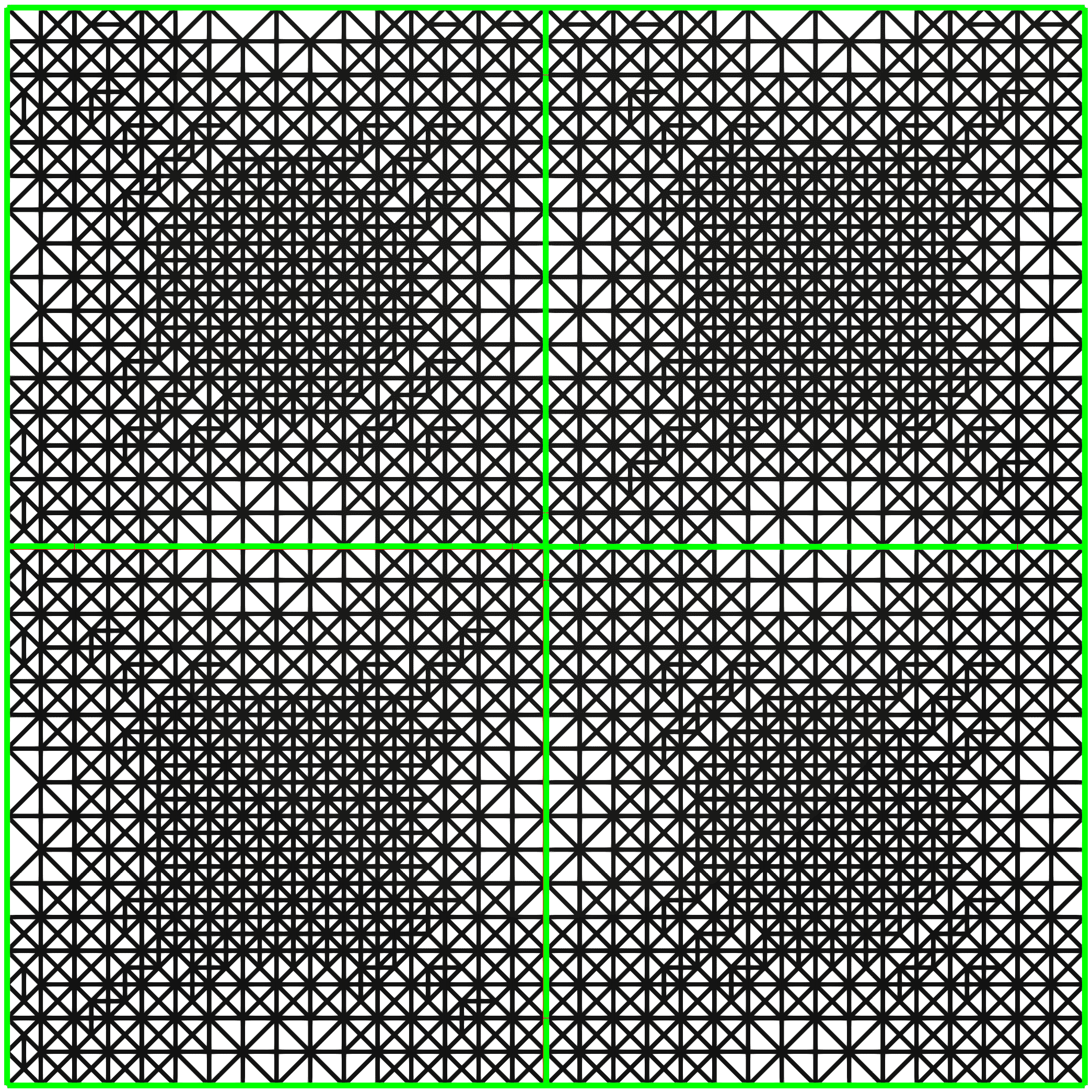}\\
\end{minipage}
\begin{minipage}[c]{0.27\textwidth}\centering
\includegraphics[trim={0 0 0 0},clip,width=4.4cm,height=4.4cm,scale=0.30]{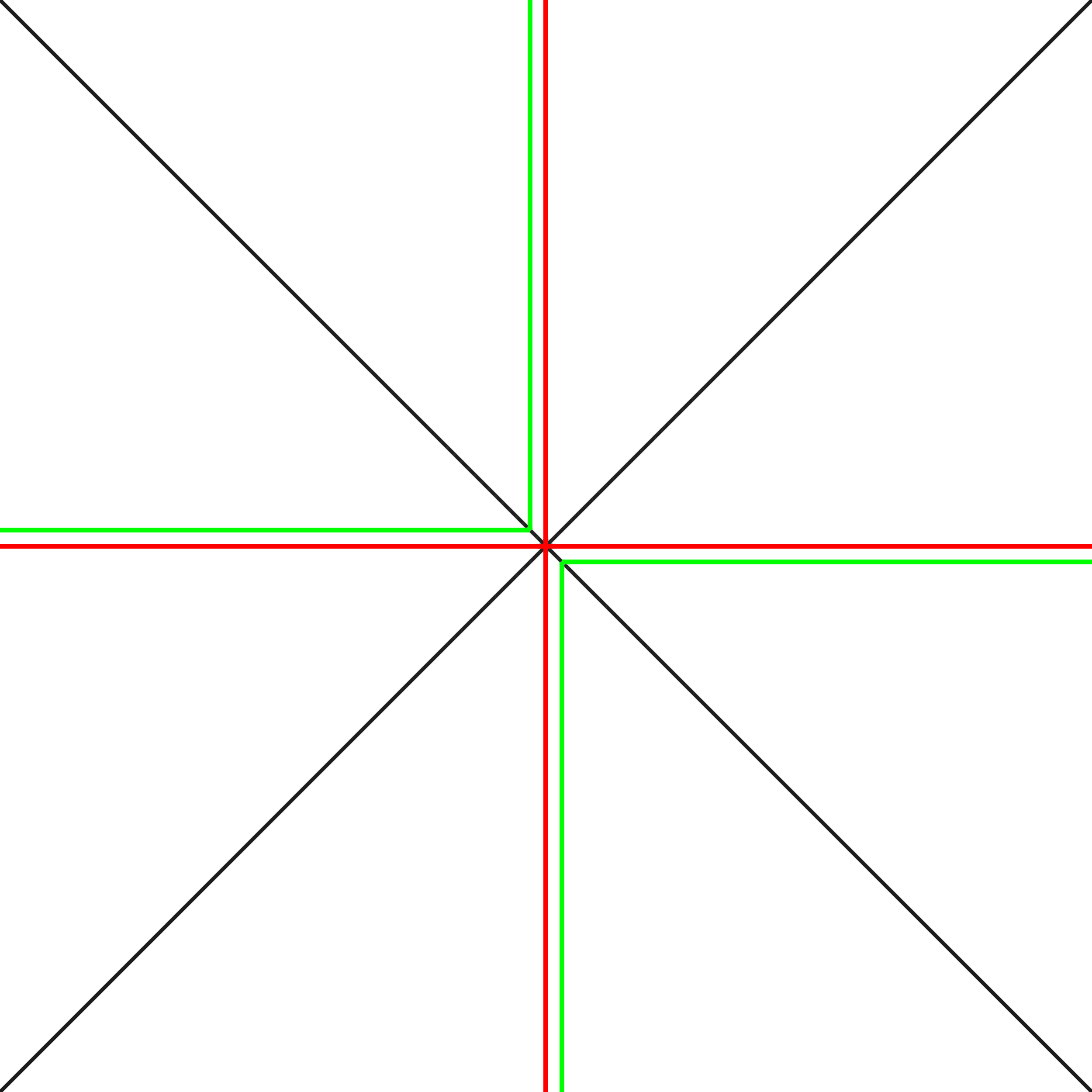}\\
\end{minipage}
\caption{Approximate control $\bar{\mathfrak{u}}_{\ell}$ and comparison of the continuous (red) and discrete switching sets on the adaptively refined meshes obtained after $5$ (upper row) and $10$ (lower row) iterations for the problem from section \ref{sec:ex_1} with $f(\cdot,y)=\arctan(y)$; in the red region $\bar{\mathfrak{u}}_{\ell} = 1$  whereas in the blue region $\bar{\mathfrak{u}}_{\ell} = -1$.}
\label{fig:ex_1_3}
\end{figure}


\subsection{Exact solution on non-convex domain}\label{sec:ex_2}
We consider $\Omega=(-1,1)^2\setminus[0,1)\times(-1,0]$, $a=-1$, $b=1$, and $\gamma=1$.
The functions $\mathfrak{f}$ and $y_{\Omega}$ are such that the exact optimal state and adjoint state are given, in polar coordinates $(\rho,\omega)$ with $\omega\in[0,3\pi/2]$, by
\begin{align*}
\bar{y}(\rho,\omega)&=\sin(\pi(\rho\sin(\omega)+1)/2)\sin(\pi(\rho \cos(\omega)+1)/2)\rho^{2/3}\sin(2\omega/3), 
\\
\bar{p}(\rho,\omega)&= (0.5-\rho)\bar{y}(\rho,\omega).
\end{align*}
The nonlinear function in this example is $f(\cdot,y)=y^3$.
In this examples, we investigate the performance of the error estimator $\eta_{ocp}$ when we violate the convexity assumption on the domain, considered in our analysis.

The numerical results for this example are shown in Figures \ref{fig:ex_2_1}, \ref{fig:ex_2_2}, and \ref{fig:ex_2_3}.
In Fig. \ref{fig:ex_2_1}, we display experimental convergence rates for each contribution of the total error when uniform and adaptive refinements are considered.
We observe that the proposed adaptive procedure outperforms uniform refinement.
In particular, it exhibits optimal convergence rates for each contribution of the total error.
In Fig. \ref{fig:ex_2_2}, we show experimental convergence rates for all the individual contributions of the error estimator $\eta_{ocp}$ and the effectivity index, when adaptive refinement is considered.
We observe that the effectivity index stabilizes around the value $2$ when the total number of degrees of freedom increases.
In Fig. \ref{fig:ex_2_3}, we present an approximate optimal control $\bar{\mathfrak{u}}_{\ell}$ and its associated adaptively refined meshes obtained after 5 and 10 iterations. 
It can be observed that the refinement is being concentrated close to the re-entrant corner (0,0).
Moreover, even when the adaptive refinement is not necessarily concentrated near the discrete switching set, this set seems to converge to the continuous one as the total number of degrees of freedom increases. 
The classical bang-bang structure of the control is also observed.


\begin{figure}
  \begin{tikzpicture}
  \begin{groupplot}[group style={group size= 2 by 1},width=0.4\textwidth,cycle list/Dark2-6,
                      cycle multiindex* list={
                          mark list*\nextlist
                          Dark2-6\nextlist},
                      every axis plot/.append style={ultra thick},
                      grid=major,
                      xlabel={Ndofs},]
         \nextgroupplot[title={Error contributions},ymode=log,xmode=log,
           legend entries={\hspace{-0.5cm}\tiny{$\|\bar{y}-\bar{y}_{\ell}\|_{\Omega}$},\tiny{$\|\bar{p}-\bar{p}_{\ell}\|_{L^{\infty}(\Omega)}$},\tiny{$\|\bar{u}-\bar{\mathfrak{u}}_{\ell}\|_{L^{1}(\Omega)}$}},
                      legend pos=north east]
                \addplot table [x=dofs,y=error_y] {data_for_plots/Ex_2/Errors_ex2_unif.dat};
                \addplot table [x=dofs,y=error_p] {data_for_plots/Ex_2/Errors_ex2_unif.dat};
                 \addplot table [x=dofs,y=error_u] {data_for_plots/Ex_2/Errors_ex2_unif.dat};
                
                \logLogSlopeTriangle{0.9}{0.25}{0.42}{1/3}{black}{{\small $\frac{1}{3}$}};
                \logLogSlopeTriangleBelow{0.8}{0.25}{0.1}{2/3}{black}{{\small $\frac{2}{3}$}}
              \nextgroupplot[title={Error contributions},ymode=log,xmode=log,
           legend entries={\hspace{-0.5cm}\tiny{$\|\bar{y}-\bar{y}_{\ell}\|_{\Omega}$},\tiny{$\|\bar{p}-\bar{p}_{\ell}\|_{L^{\infty}(\Omega)}$},\tiny{$\|\bar{u}-\bar{\mathfrak{u}}_{\ell}\|_{L^{1}(\Omega)}$}},
                      legend pos=north east]
                \addplot table [x=dofs,y=error_y] {data_for_plots/Ex_2/Errors_ex2_adap.dat};
                \addplot table [x=dofs,y=error_p] {data_for_plots/Ex_2/Errors_ex2_adap.dat};
                \addplot table [x=dofs,y=error_u] {data_for_plots/Ex_2/Errors_ex2_adap.dat};
                
                \logLogSlopeTriangle{0.9}{0.25}{0.35}{1}{black}{{\small $1$}};
                \logLogSlopeTriangleBelow{0.7}{0.25}{0.15}{1}{black}{{\small $1$}}
    \end{groupplot}
\end{tikzpicture}
  \caption{Experimental convergence rates for individual contributions of the total error with uniform (left) and adaptive (right) refinements for the problem from section \ref{sec:ex_2}.}
\label{fig:ex_2_1}
\end{figure}


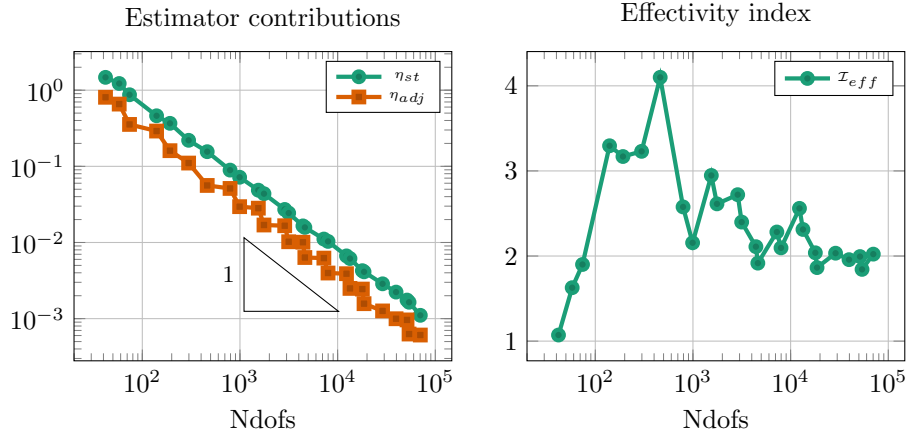
\begin{figure}
  \begin{tikzpicture}
  \begin{groupplot}[group style={group size= 2 by 1},width=0.4\textwidth,cycle list/Dark2-6,
                      cycle multiindex* list={
                          mark list*\nextlist
                          Dark2-6\nextlist},
                      every axis plot/.append style={ultra thick},
                      grid=major,
                      xlabel={Ndofs},]
         \nextgroupplot[title={Estimator contributions},ymode=log,xmode=log,
           legend entries={\tiny{$\eta_{st}$},\tiny{$\eta_{adj}$}},
                      legend pos=north east]
                \addplot table [x=dofs,y=est_y] {data_for_plots/Ex_2/Errors_ex2_adap.dat};
                \addplot table [x=dofs,y=est_p_inf] {data_for_plots/Ex_2/Errors_ex2_adap.dat};
                          
                \logLogSlopeTriangleBelow{0.7}{0.25}{0.16}{1}{black}{{\small $1$}}
                
              \nextgroupplot[title={Effectivity index}, xmode=log,         
           legend entries={\tiny{$\mathcal{I}_{eff}$}},
                      legend pos=north east]
                \addplot table [x=dofs,y=eff_index] {data_for_plots/Ex_2/Errors_ex2_adap.dat};
    \end{groupplot}
\end{tikzpicture}
  \caption{Experimental convergence rates for individual contributions of the estimator $\eta_{ocp}$ (left) and effectivity index (right) with adaptive refinement for the problem from section \ref{sec:ex_2}.}
\label{fig:ex_2_2}
\end{figure}


\begin{figure}[!ht]
\begin{minipage}[c]{0.27\textwidth}\centering
\includegraphics[trim={0 0 0 0},clip,width=4.4cm,height=4.4cm,scale=0.30]{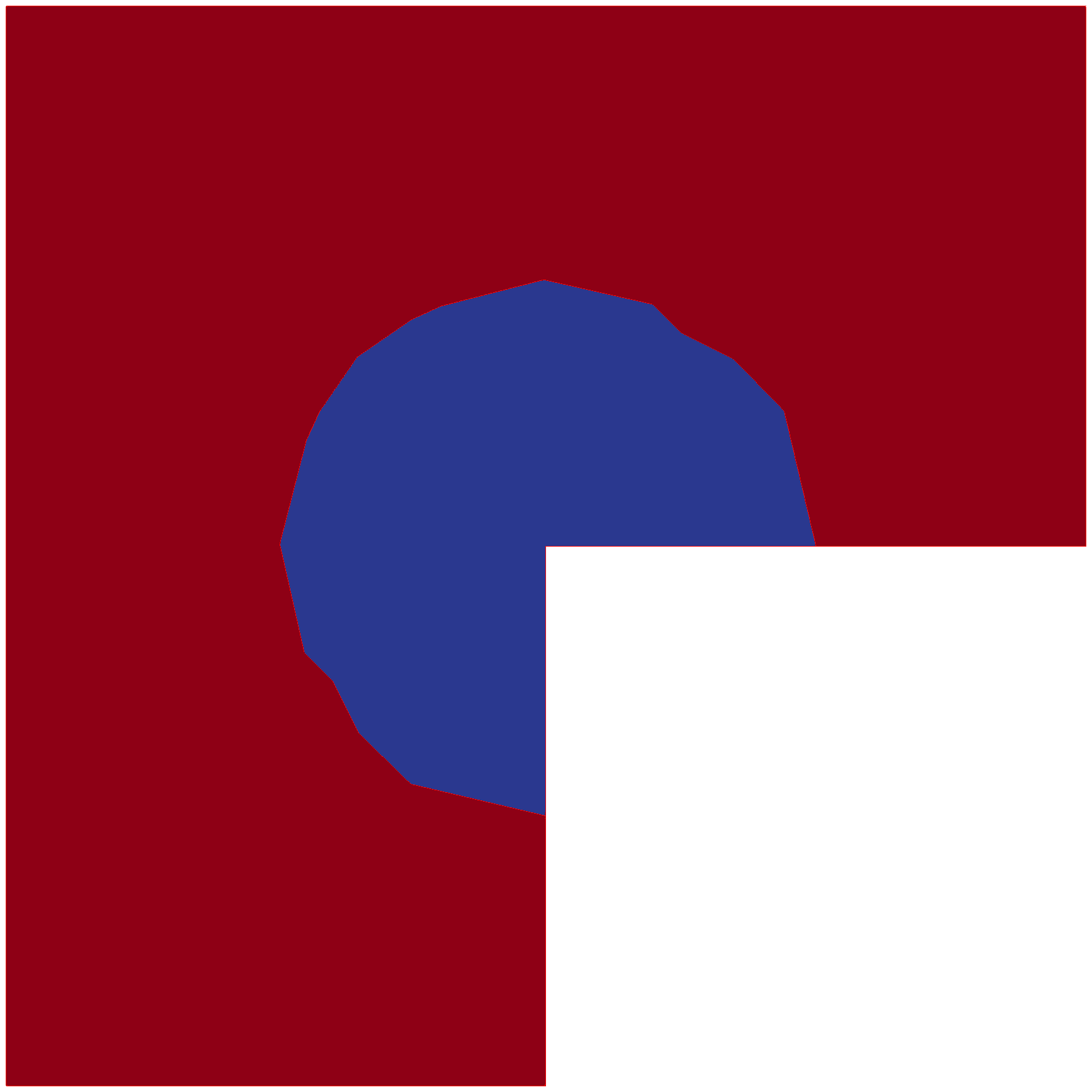}\\
\end{minipage}
\begin{minipage}[c]{0.27\textwidth}\centering
\includegraphics[trim={0 0 0 0},clip,width=4.4cm,height=4.4cm,scale=0.30]{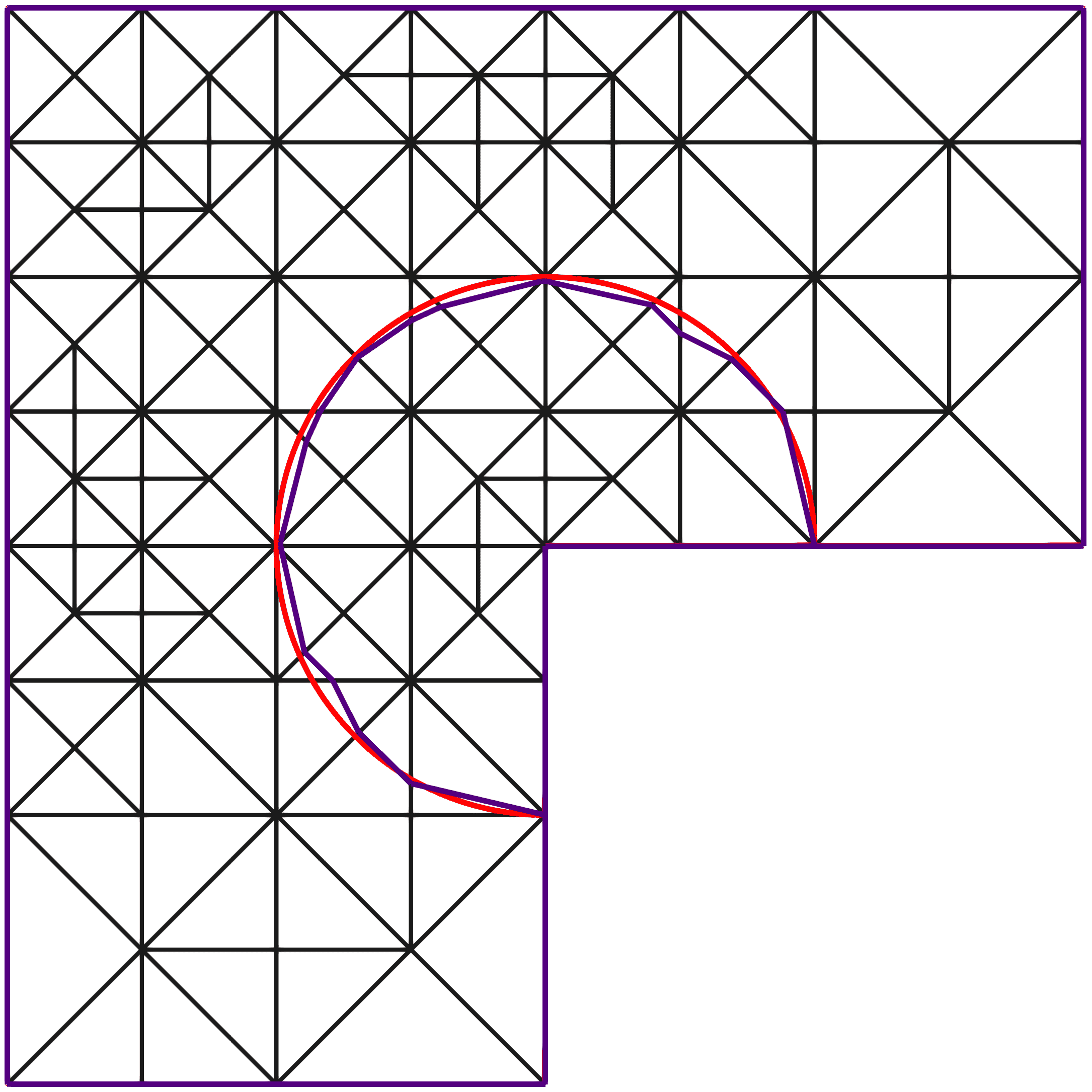}\\
\end{minipage}
\begin{minipage}[c]{0.27\textwidth}\centering
\includegraphics[trim={0 0 0 0},clip,width=4.4cm,height=4.4cm,scale=0.30]{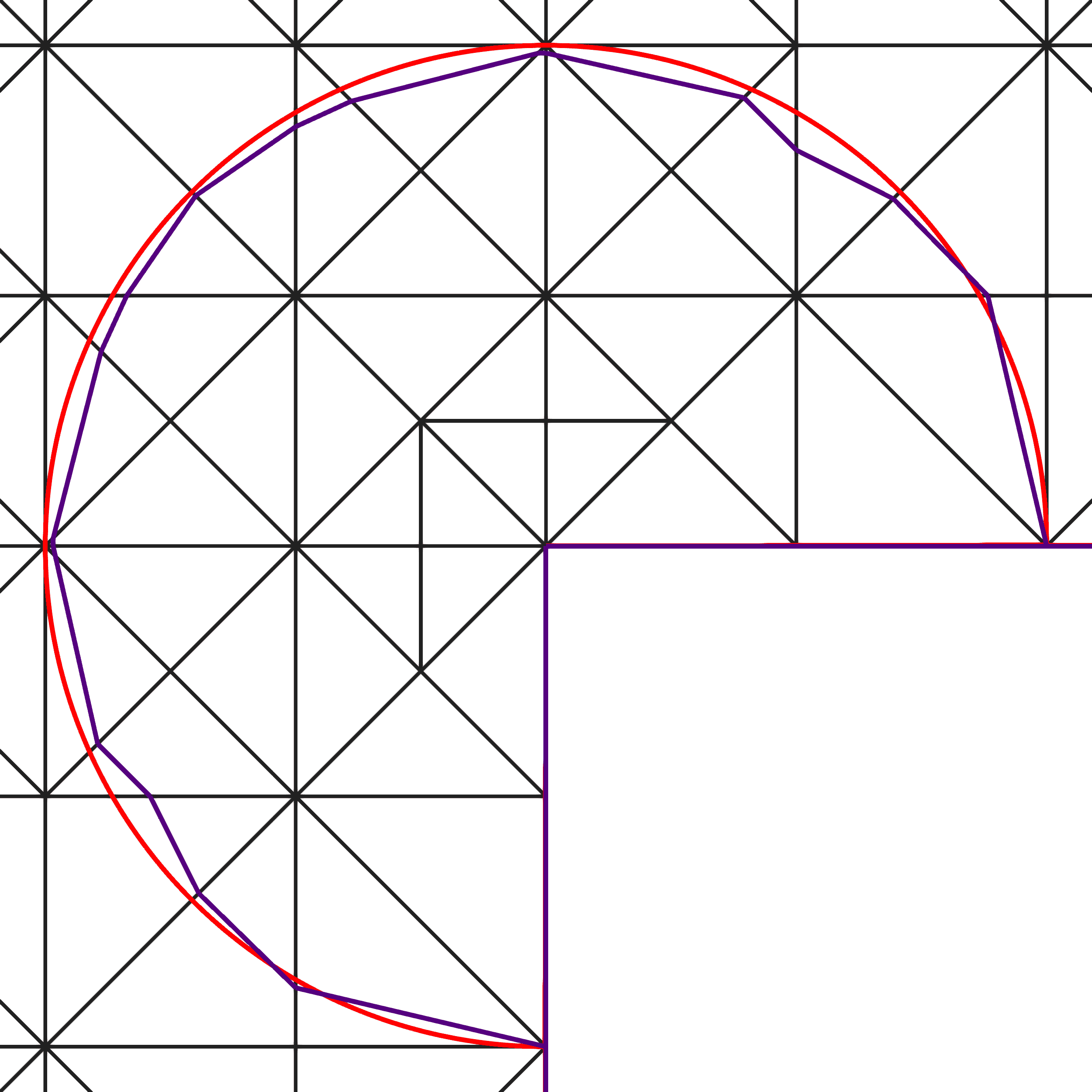}\\
\end{minipage}
\\
\begin{minipage}[c]{0.27\textwidth}\centering
\includegraphics[trim={0 0 0 0},clip,width=4.4cm,height=4.4cm,scale=0.30]{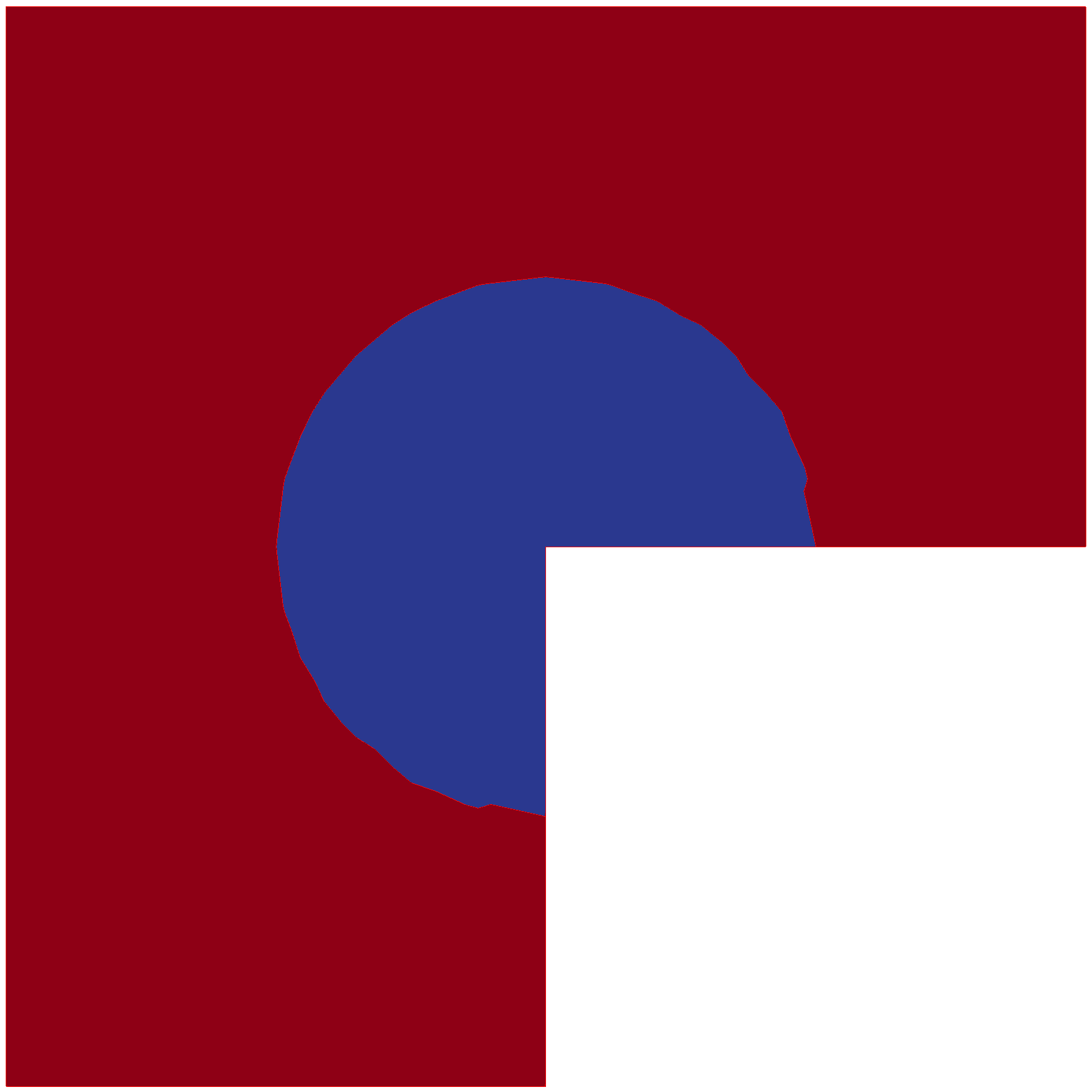}\\
\end{minipage}
\begin{minipage}[c]{0.27\textwidth}\centering
\includegraphics[trim={0 0 0 0},clip,width=4.4cm,height=4.4cm,scale=0.30]{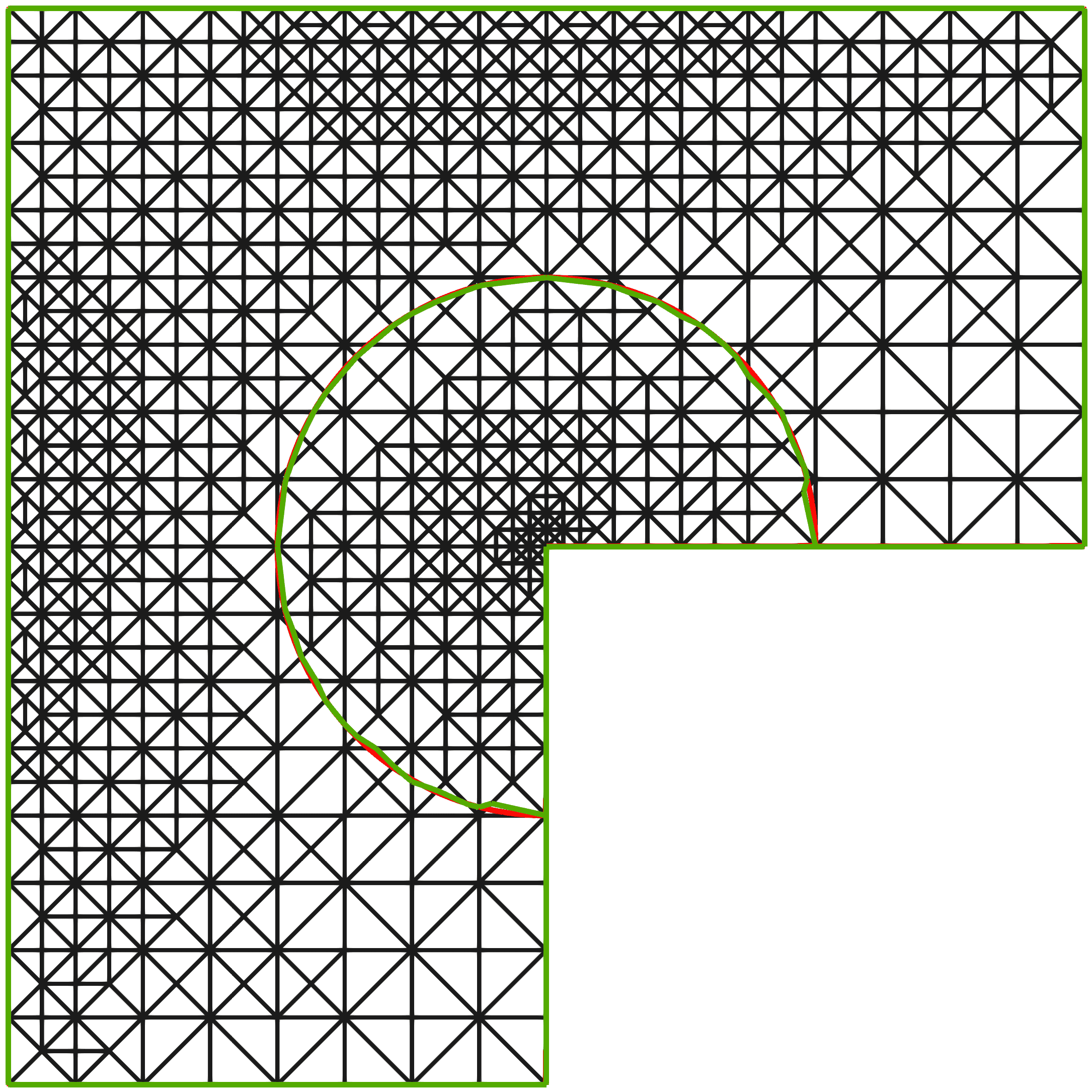}\\
\end{minipage}
\begin{minipage}[c]{0.27\textwidth}\centering
\includegraphics[trim={0 0 0 0},clip,width=4.4cm,height=4.4cm,scale=0.30]{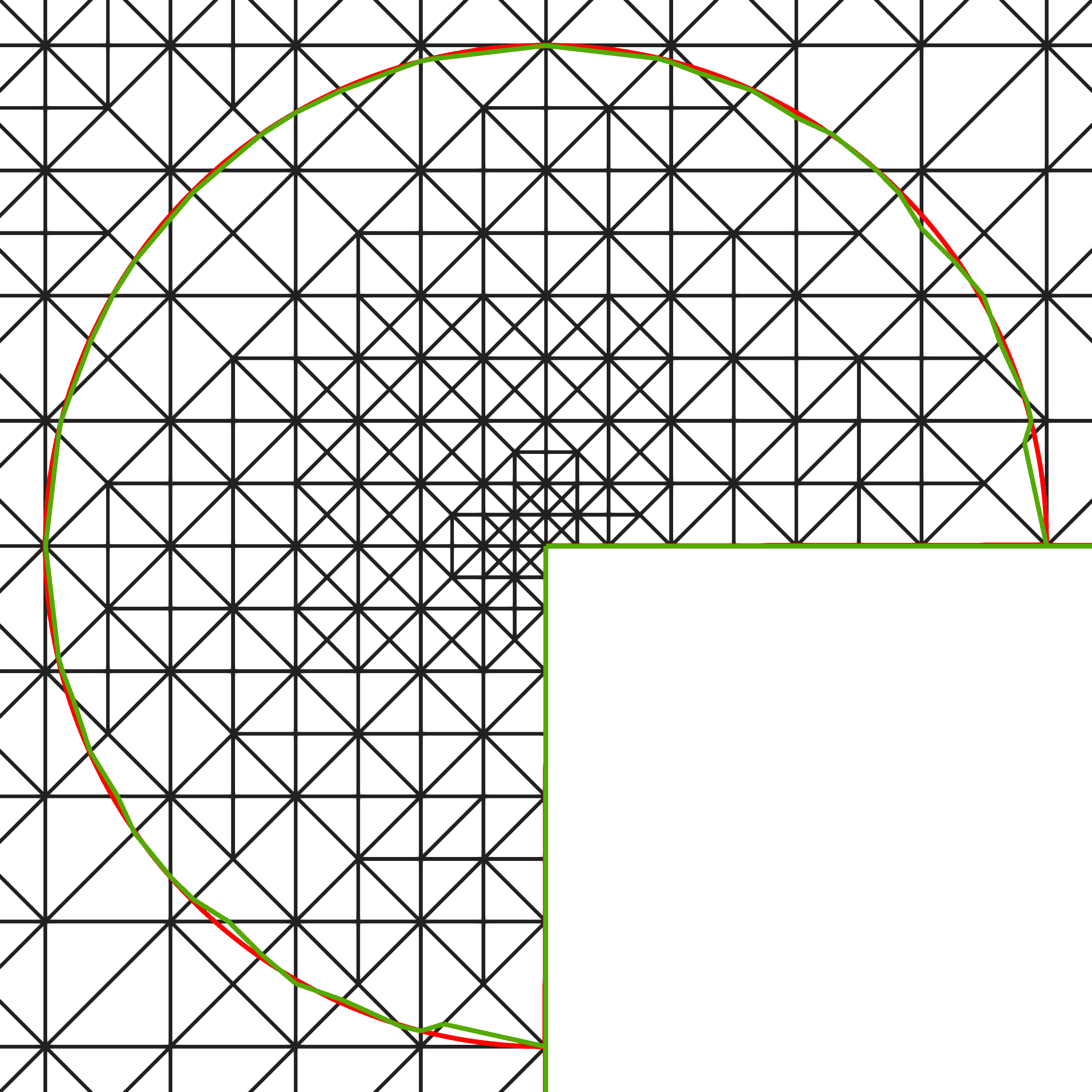}\\
\end{minipage}
\caption{Approximate control $\bar{\mathfrak{u}}_{\ell}$ and comparison of the continuous (red) and discrete switching sets on the adaptively refined meshes obtained after $5$ (upper row) and $10$ (lower row) iterations for the problem from section \ref{sec:ex_2}; in the red region $\bar{\mathfrak{u}}_{\ell} = 1$  whereas in the blue region $\bar{\mathfrak{u}}_{\ell} = -1$.}
\label{fig:ex_2_3}
\end{figure}


\subsection{Exact solution on a disk}\label{sec:ex_3}

Inspired by the numerical experiment provided in \cite[Section 7]{CM2021}, we set $\Omega:=\{(x_1,x_2) ~ : \ x_1^2+x_2^2 \leq 1\}$, that is, the unit two-dimensional disk. 
We also choose $a=-1$, $b=1$, $f(\cdot,y)=y^3$, and $\mathfrak{f}$ and $y_{\Omega}$ such that 
\begin{align*}
  & \bar{y}(x_{1},x_{2}) = 0.5(1 - x_1^2 - x_2^2), \\ 
  & \bar{p}(x_{1},x_{2}) = 2\bar{y}(x_{1},x_{2})\left((1/12-x_1)|1/12-x_1| + (1/12-x_2)|1/12-x_2|\right), \\
  & \bar{u}(x_{1},x_{2}) = -\mathrm{sign}(\bar{p}(x_{1},x_{2}))
\end{align*}
for $(x_{1},x_{2})\in \Omega$.
We note that $\bar{p}$ corresponds to a polynomial of order $2$.
It can be proved that $\bar{u}$ satisfies the growth condition \eqref{eq:assumption_control_growth} with $\gamma = 0.5$. 
In this example, we study the influence of $\gamma$ on the adaptive refinement steered by the error indicator \eqref{def:total_indicator} (see Algorithm \ref{Algorithm}) by choosing $\gamma=0.5$ and $\gamma = 1$.

The numerical results for this example are shown in Figures \ref{fig:ex_3_1}, \ref{fig:ex_3_2},  \ref{fig:ex_3_3}, and \ref{fig:ex_3_4}.
In Fig. \ref{fig:ex_3_1}, we display experimental convergence rates for each contribution of the total error when uniform and adaptive (with $\gamma = 0.5$) refinements are considered. 
We observe that the error associated with the control variable does not exhibit an optimal convergence rate, even when considering adaptive refinement; see Remark \ref{rmk:convergences}.
In contrast, in Fig. \ref{fig:ex_3_2}--which is related to the case $\gamma = 1$ in the error indicator \eqref{def:total_indicator}--we observe optimal convergence rates for each contribution of the total error. 
In Figs. \ref{fig:ex_3_3} ($\gamma = 0.5$) and \ref{fig:ex_3_4} ($\gamma=1$), we present an approximate optimal control $\bar{\mathfrak{u}}_{\ell}$ and its associated adaptively refined meshes obtained after 5 and 10 iterations.
In both figures, we observe the classical bang-bang structure of the control and that the discrete switching set seems to converge to the continuous one as the total number of degrees of freedom increases. 

\begin{remark}[reduced convergences rates when $\gamma=0.5$]\label{rmk:convergences}
    In Fig. \ref{fig:ex_3_1} we observed that $\|\bar{u}-\bar{\mathfrak{u}}_{\ell}\|_{L^{1}(\Omega)}$ does not converge with an optimal rate when using $\gamma = 0.5$ in the adaptive refinement. 
    This may be due to the fact that a large amount of elements are being refined in this case (see Fig. \ref{fig:ex_3_3}), which implies that the refinement is to a great extent uniform. 
    In contrast, when using $\gamma = 1$ in the adaptive refinement, we are able to recover optimal convergence rates for $\|\bar{u}-\bar{\mathfrak{u}}_{\ell}\|_{L^{1}(\Omega)}$.
\end{remark}


\begin{figure}
  \begin{tikzpicture}
  \begin{groupplot}[group style={group size= 2 by 1},width=0.4\textwidth,cycle list/Dark2-6,
                      cycle multiindex* list={
                          mark list*\nextlist
                          Dark2-6\nextlist},
                      every axis plot/.append style={ultra thick},
                      grid=major,
                      xlabel={Ndofs},]
         \nextgroupplot[title={Error contributions},ymode=log,xmode=log,
           legend entries={\hspace{-0.5cm}\tiny{$\|\bar{y}-\bar{y}_{\ell}\|_{\Omega}$},\tiny{$\|\bar{p}-\bar{p}_{\ell}\|_{L^{\infty}(\Omega)}$},\tiny{$\|\bar{u}-\bar{\mathfrak{u}}_{\ell}\|_{L^{1}(\Omega)}$}},
                      legend pos=north east]
                \addplot table [x=dofs,y=error_y] {data_for_plots/Ex_3/Errors_ex3_unif.dat};
                \addplot table [x=dofs,y=error_p] {data_for_plots/Ex_3/Errors_ex3_unif.dat};
                 \addplot table [x=dofs,y=error_u] {data_for_plots/Ex_3/Errors_ex3_unif.dat};

                \logLogSlopeTriangle{0.9}{0.25}{0.40}{2/3}{black}{{\small $\frac{2}{3}$}};
                \logLogSlopeTriangleBelow{0.8}{0.25}{0.1}{1}{black}{{\small $1$}}
              \nextgroupplot[title={Error contributions},ymode=log,xmode=log,
           legend entries={\hspace{-0.5cm}\tiny{$\|\bar{y}-\bar{y}_{\ell}\|_{\Omega}$},\tiny{$\|\bar{p}-\bar{p}_{\ell}\|_{L^{\infty}(\Omega)}$},\tiny{$\|\bar{u}-\bar{\mathfrak{u}}_{\ell}\|_{L^{1}(\Omega)}$}},
                      legend pos=north east]
                \addplot table [x=dofs,y=error_y] {data_for_plots/Ex_3/Errors_ex3_adap05.dat};
                \addplot table [x=dofs,y=error_p] {data_for_plots/Ex_3/Errors_ex3_adap05.dat};
                 \addplot table [x=dofs,y=error_u] {data_for_plots/Ex_3/Errors_ex3_adap05.dat};
                
                \logLogSlopeTriangle{0.9}{0.25}{0.45}{2/3}{black}{{\small $\frac{2}{3}$}};
                \logLogSlopeTriangleBelow{0.7}{0.25}{0.15}{1}{black}{{\small $1$}}
    \end{groupplot}
\end{tikzpicture}
  \caption{Experimental convergence rates for individual contributions of the total error with uniform (left) and adaptive (right) refinements for the problem from section \ref{sec:ex_3} with $\gamma=0.5$.}
\label{fig:ex_3_1}
\end{figure}


\begin{figure}
  \begin{tikzpicture}
  \begin{groupplot}[group style={group size= 2 by 1},width=0.4\textwidth,cycle list/Dark2-6,
                      cycle multiindex* list={
                          mark list*\nextlist
                          Dark2-6\nextlist},
                      every axis plot/.append style={ultra thick},
                      grid=major,
                      xlabel={Ndofs},]
         \nextgroupplot[title={Error contributions},ymode=log,xmode=log,
           legend entries={\hspace{-0.5cm}\tiny{$\|\bar{y}-\bar{y}_{\ell}\|_{\Omega}$},\tiny{$\|\bar{p}-\bar{p}_{\ell}\|_{L^{\infty}(\Omega)}$},\tiny{$\|\bar{u}-\bar{\mathfrak{u}}_{\ell}\|_{L^{1}(\Omega)}$}},
                      legend pos=north east]
                \addplot table [x=dofs,y=error_y] {data_for_plots/Ex_3/Errors_ex3_unif.dat};
                \addplot table [x=dofs,y=error_p] {data_for_plots/Ex_3/Errors_ex3_unif.dat};
                 \addplot table [x=dofs,y=error_u] {data_for_plots/Ex_3/Errors_ex3_unif.dat};

                \logLogSlopeTriangle{0.9}{0.25}{0.40}{2/3}{black}{{\small $\frac{2}{3}$}};
                \logLogSlopeTriangleBelow{0.8}{0.25}{0.1}{1}{black}{{\small $1$}}
              \nextgroupplot[title={Error contributions},ymode=log,xmode=log,
           legend entries={\hspace{-0.5cm}\tiny{$\|\bar{y}-\bar{y}_{\ell}\|_{\Omega}$},\tiny{$\|\bar{p}-\bar{p}_{\ell}\|_{L^{\infty}(\Omega)}$},\tiny{$\|\bar{u}-\bar{\mathfrak{u}}_{\ell}\|_{L^{1}(\Omega)}$}},
                      legend pos=north east]
                \addplot table [x=dofs,y=error_y] {data_for_plots/Ex_3/Errors_ex3_adap10.dat};
                \addplot table [x=dofs,y=error_p] {data_for_plots/Ex_3/Errors_ex3_adap10.dat};
                \addplot table [x=dofs,y=error_u] {data_for_plots/Ex_3/Errors_ex3_adap10.dat};
                
               \logLogSlopeTriangle{0.9}{0.25}{0.41}{1}{black}{{\small $1$}};
                \logLogSlopeTriangleBelow{0.7}{0.25}{0.15}{1}{black}{{\small $1$}}
    \end{groupplot}
\end{tikzpicture}
  \caption{Experimental convergence rates for individual contributions of the total error with uniform (left) and adaptive (right) refinements for the problem from section \ref{sec:ex_3} with $\gamma=1$.}
\label{fig:ex_3_2}
\end{figure}


\begin{figure}[!ht]
\begin{minipage}[c]{0.27\textwidth}\centering
\includegraphics[trim={0 0 0 0},clip,width=4.4cm,height=4.4cm,scale=0.30]{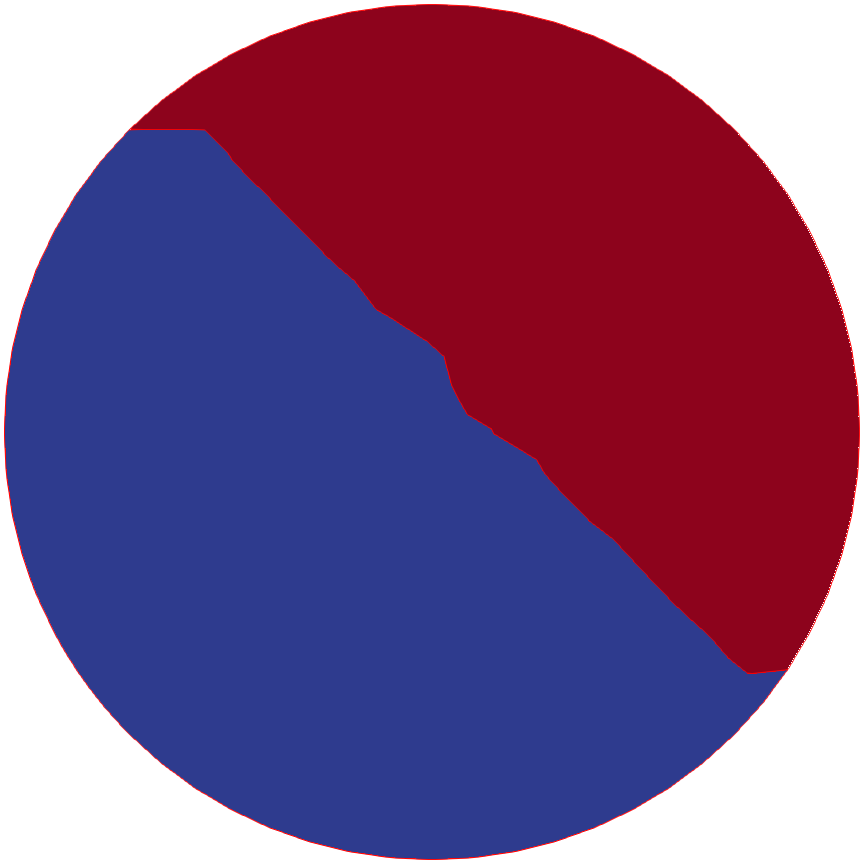}\\
\end{minipage}
\begin{minipage}[c]{0.27\textwidth}\centering
\includegraphics[trim={0 0 0 0},clip,width=4.4cm,height=4.4cm,scale=0.30]{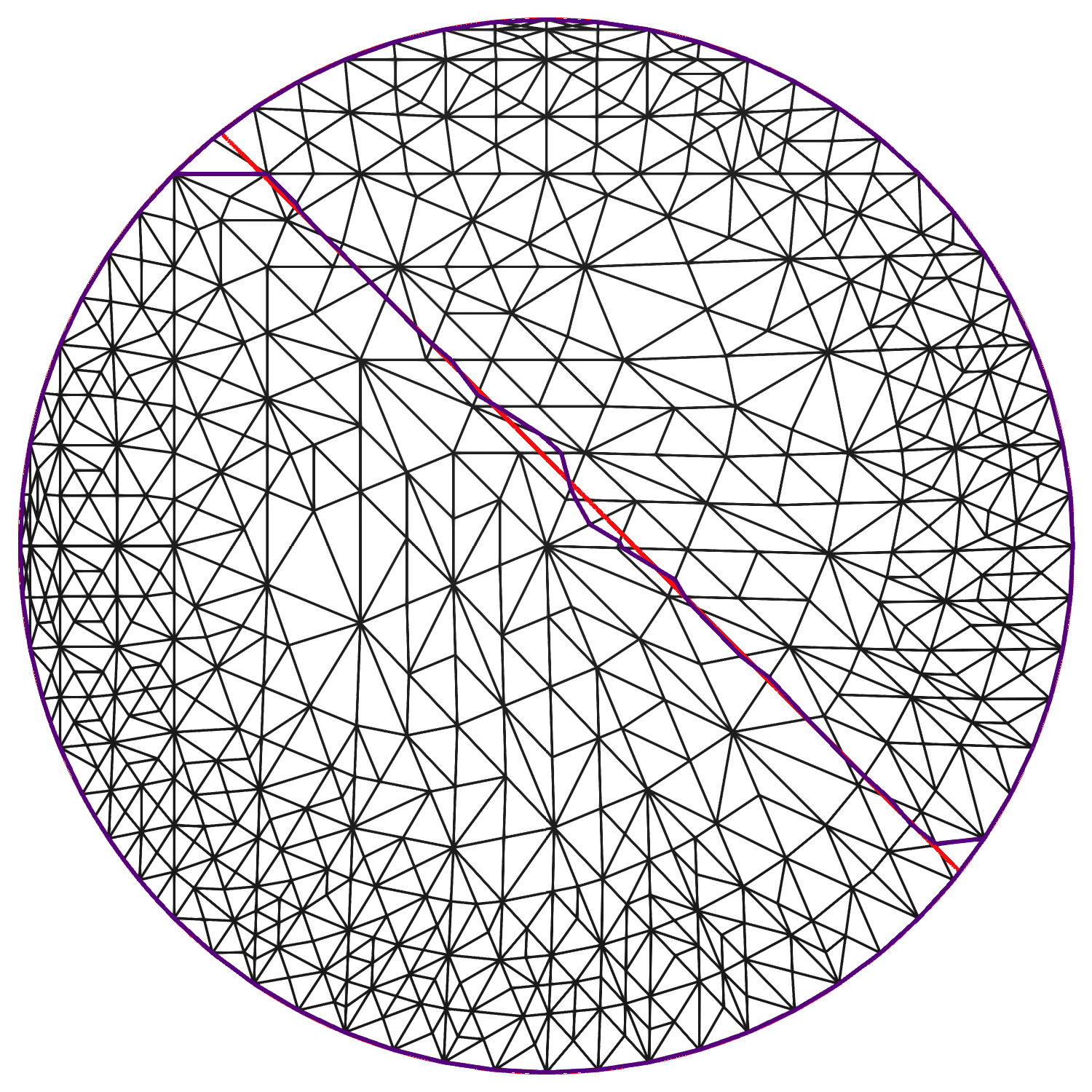}\\
\end{minipage}
\begin{minipage}[c]{0.27\textwidth}\centering
\includegraphics[trim={0 0 0 0},clip,width=4.4cm,height=4.4cm,scale=0.30]{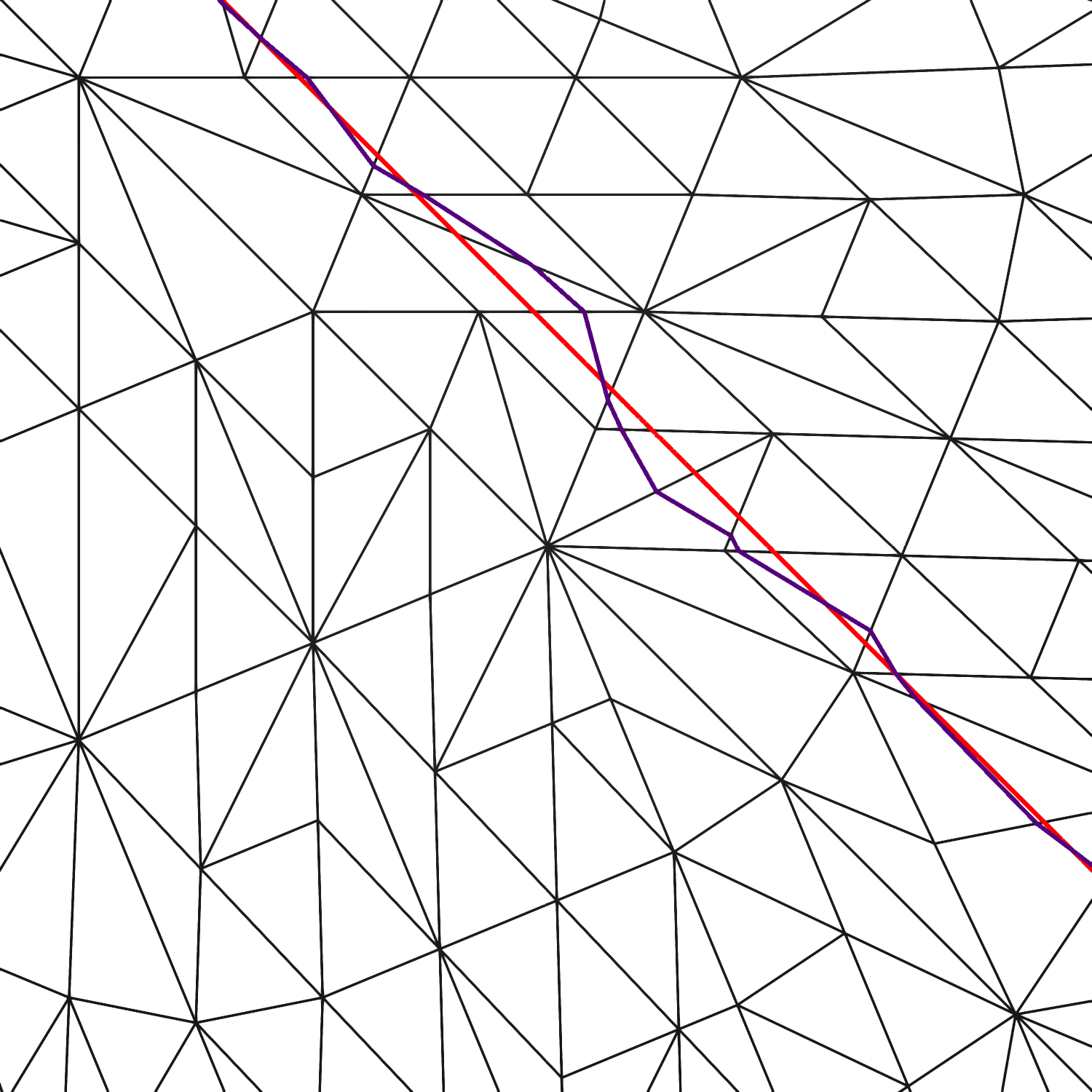}\\
\end{minipage}
\\
\begin{minipage}[c]{0.27\textwidth}\centering
\includegraphics[trim={0 0 0 0},clip,width=4.4cm,height=4.4cm,scale=0.30]{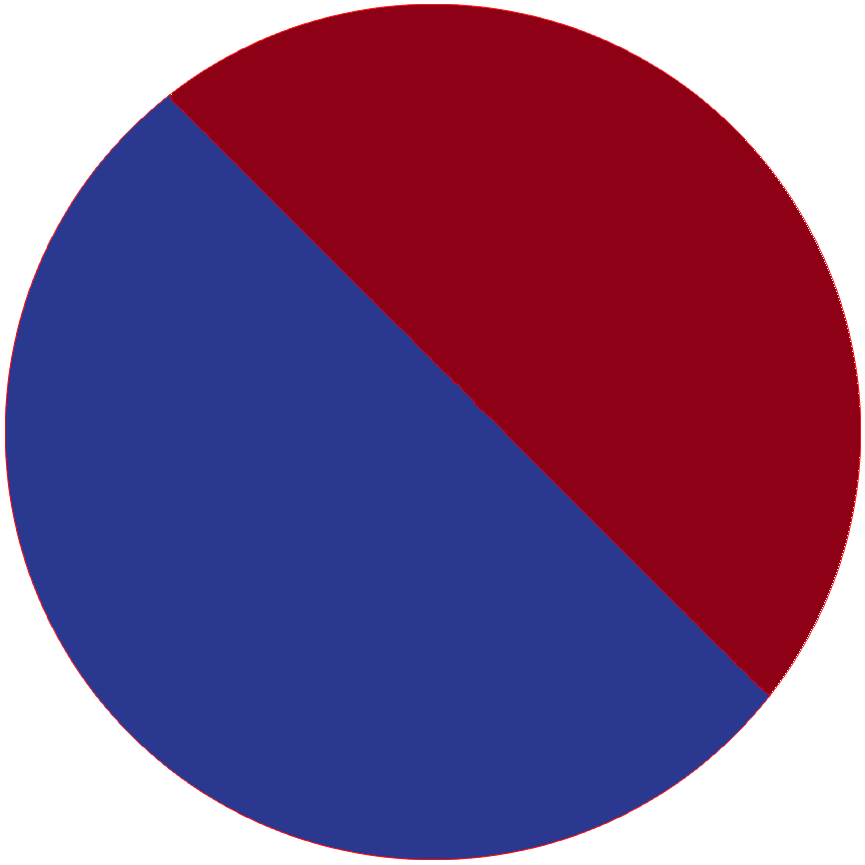}\\
\end{minipage}
\begin{minipage}[c]{0.27\textwidth}\centering
\includegraphics[trim={0 0 0 0},clip,width=4.4cm,height=4.4cm,scale=0.30]{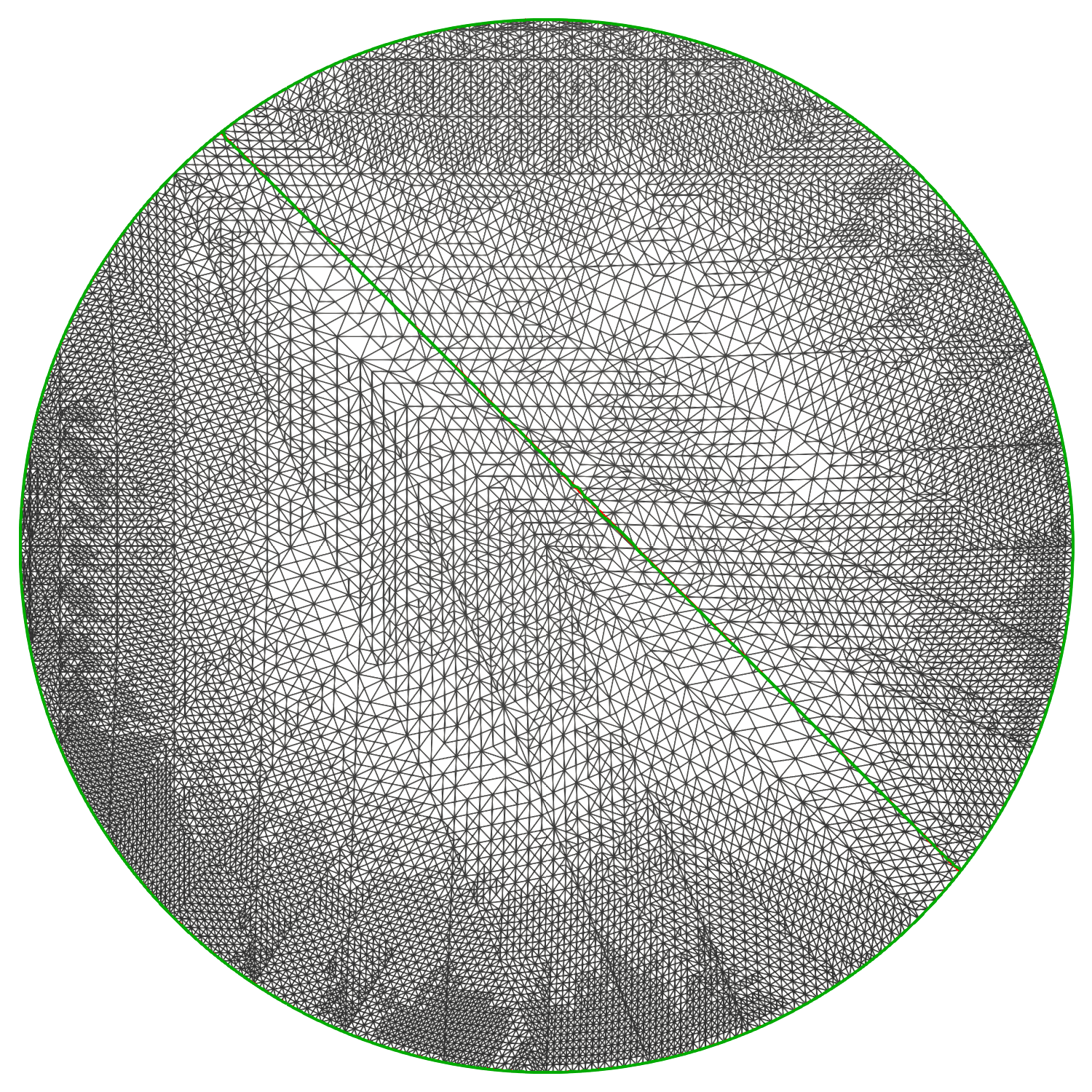}\\
\end{minipage}
\begin{minipage}[c]{0.27\textwidth}\centering
\includegraphics[trim={0 0 0 0},clip,width=4.4cm,height=4.4cm,scale=0.30]{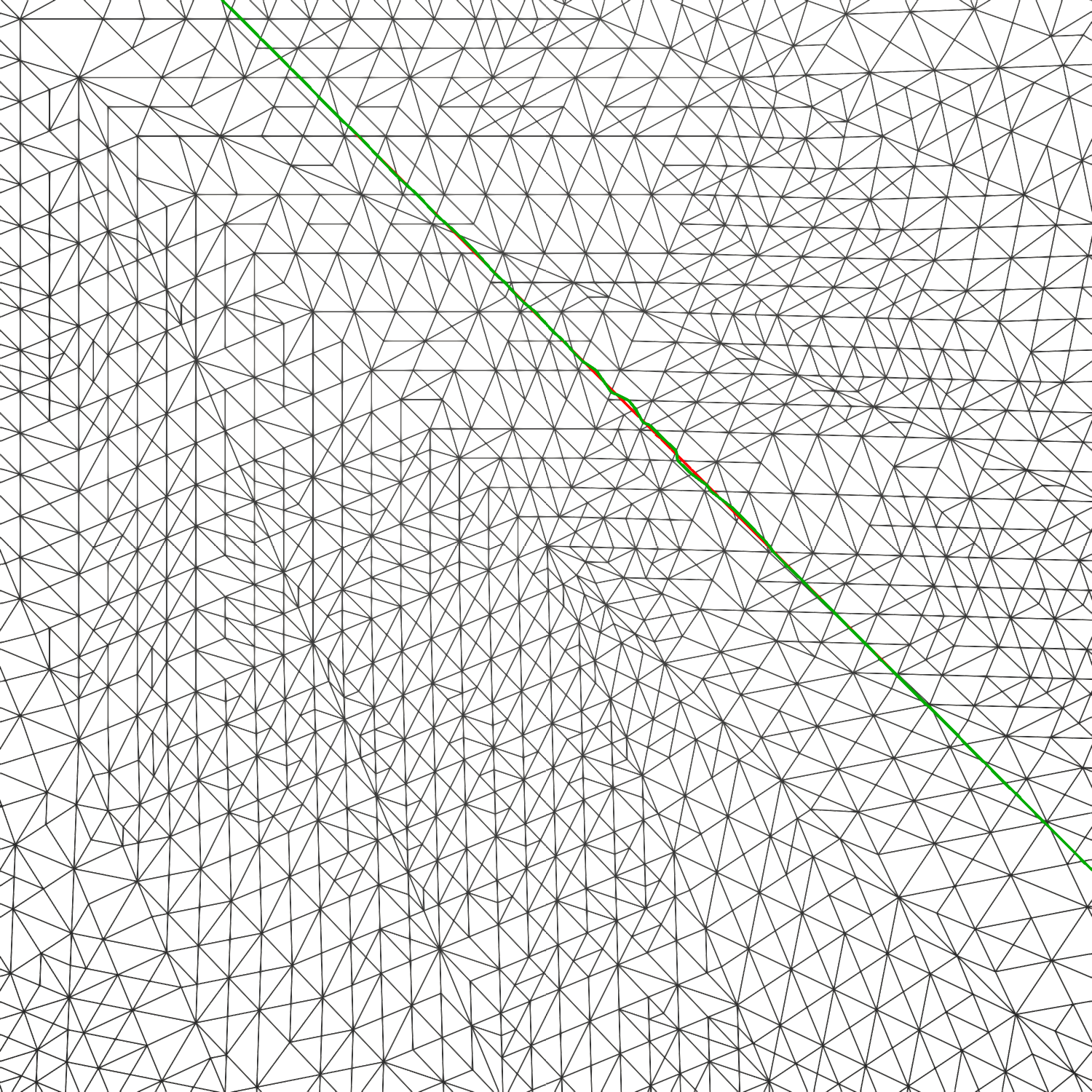}\\
\end{minipage}
\caption{Approximate control $\bar{\mathfrak{u}}_{\ell}$ and comparison of the continuous (red) and discrete switching sets on the adaptively refined meshes obtained after $5$ (upper row) and $10$ (lower row) iterations for the problem from section \ref{sec:ex_3} with $\gamma = 0.5$; in the red region $\bar{\mathfrak{u}}_{\ell} = 1$  whereas in the blue region $\bar{\mathfrak{u}}_{\ell} = -1$.}
\label{fig:ex_3_3}
\end{figure}


\begin{figure}[!ht]
\begin{minipage}[c]{0.27\textwidth}\centering
\includegraphics[trim={0 0 0 0},clip,width=4.4cm,height=4.4cm,scale=0.30]{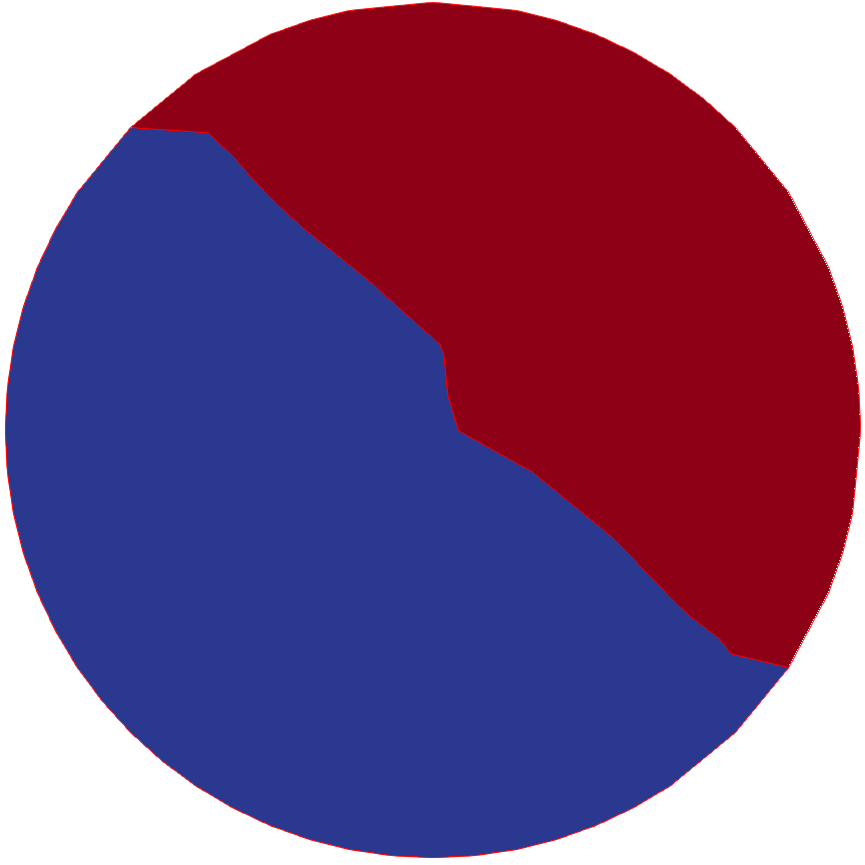}\\
\end{minipage}
\begin{minipage}[c]{0.27\textwidth}\centering
\includegraphics[trim={0 0 0 0},clip,width=4.4cm,height=4.4cm,scale=0.30]{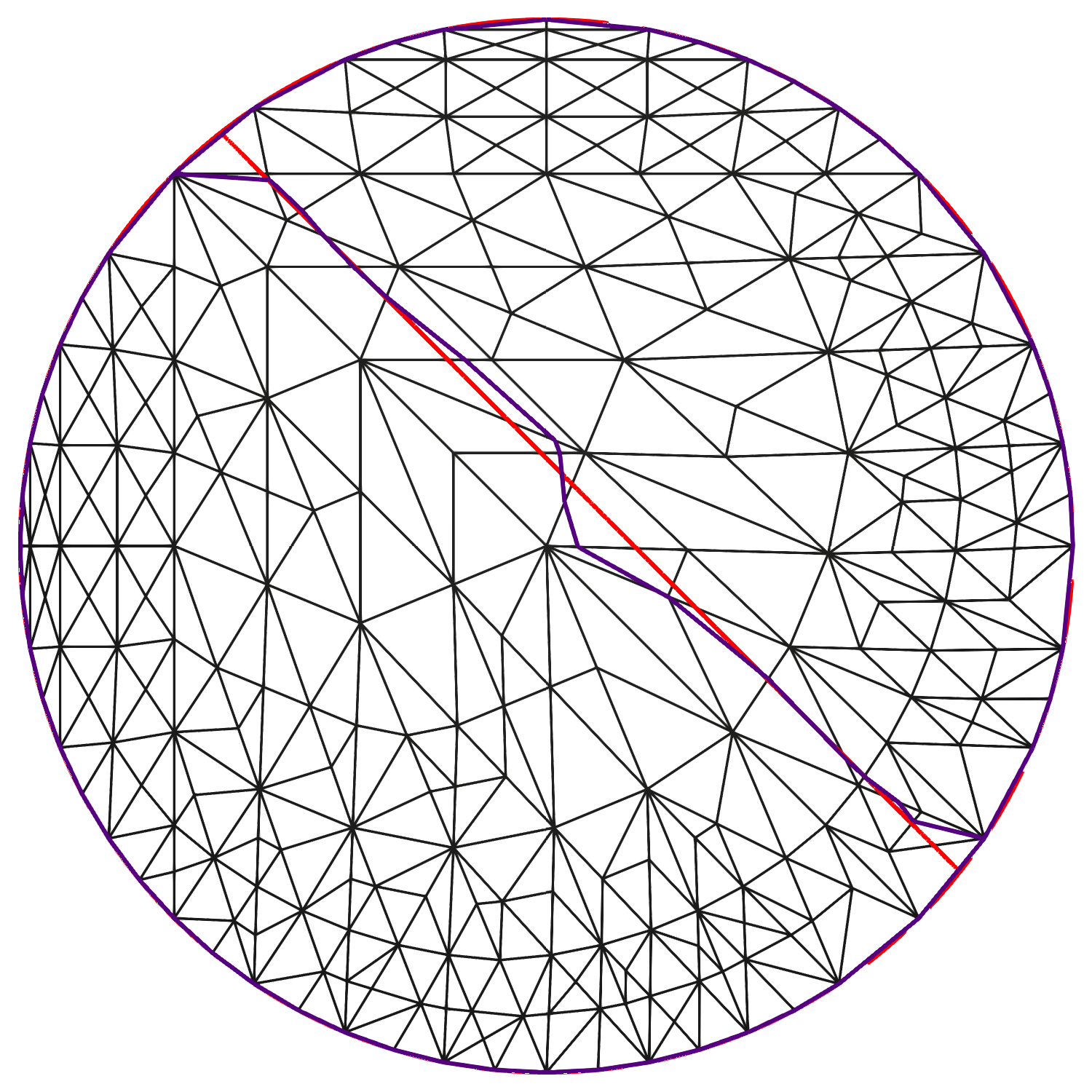}\\
\end{minipage}
\begin{minipage}[c]{0.27\textwidth}\centering
\includegraphics[trim={0 0 0 0},clip,width=4.4cm,height=4.4cm,scale=0.30]{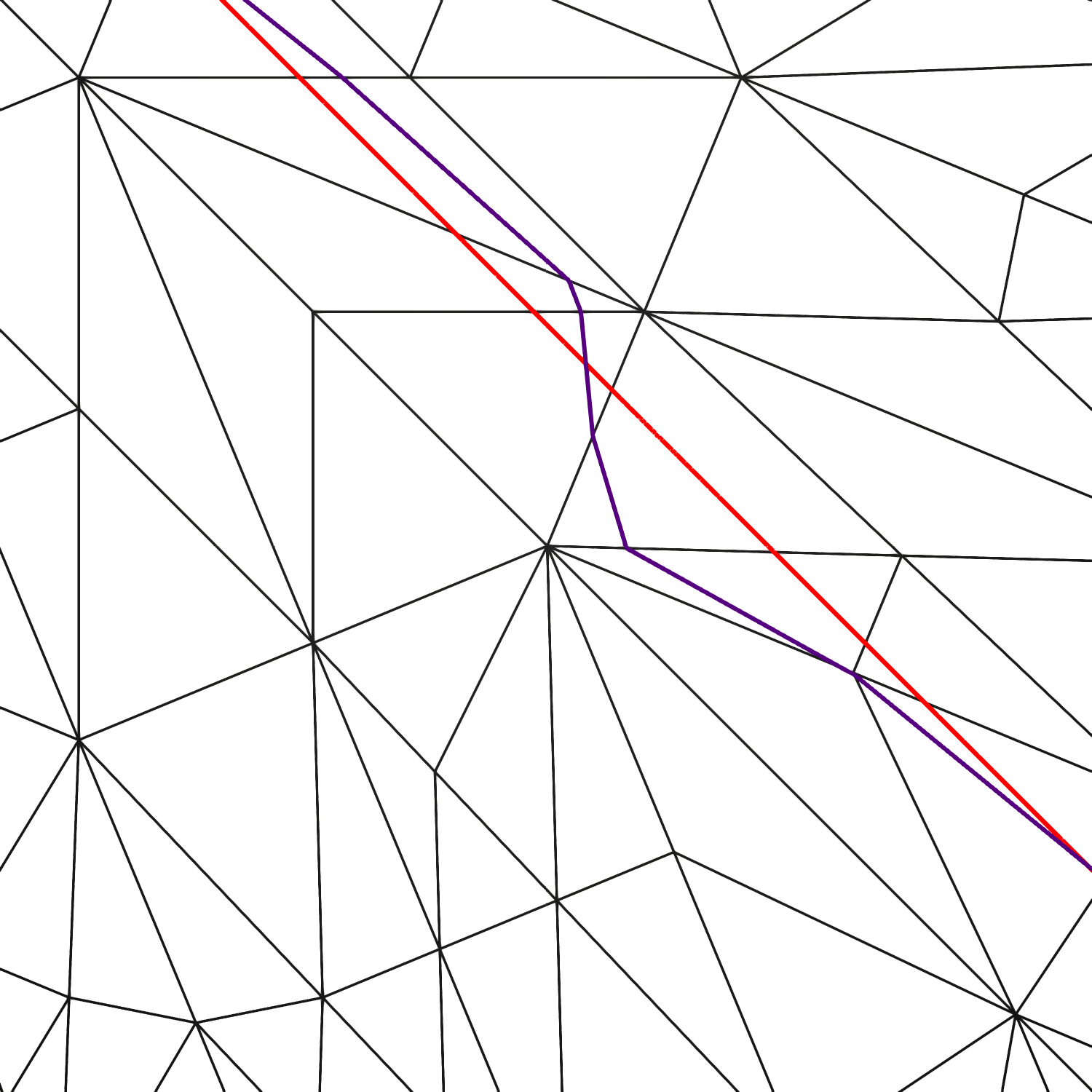}\\
\end{minipage}
\\
\begin{minipage}[c]{0.27\textwidth}\centering
\includegraphics[trim={0 0 0 0},clip,width=4.4cm,height=4.4cm,scale=0.30]{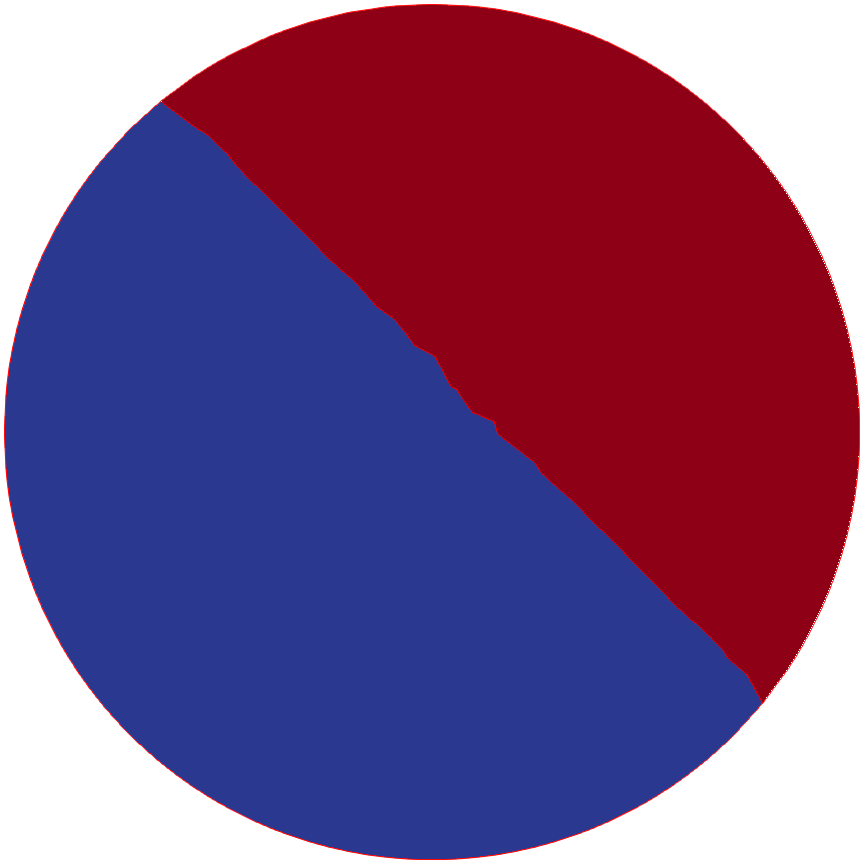}\\
\end{minipage}
\begin{minipage}[c]{0.27\textwidth}\centering
\includegraphics[trim={0 0 0 0},clip,width=4.4cm,height=4.4cm,scale=0.30]{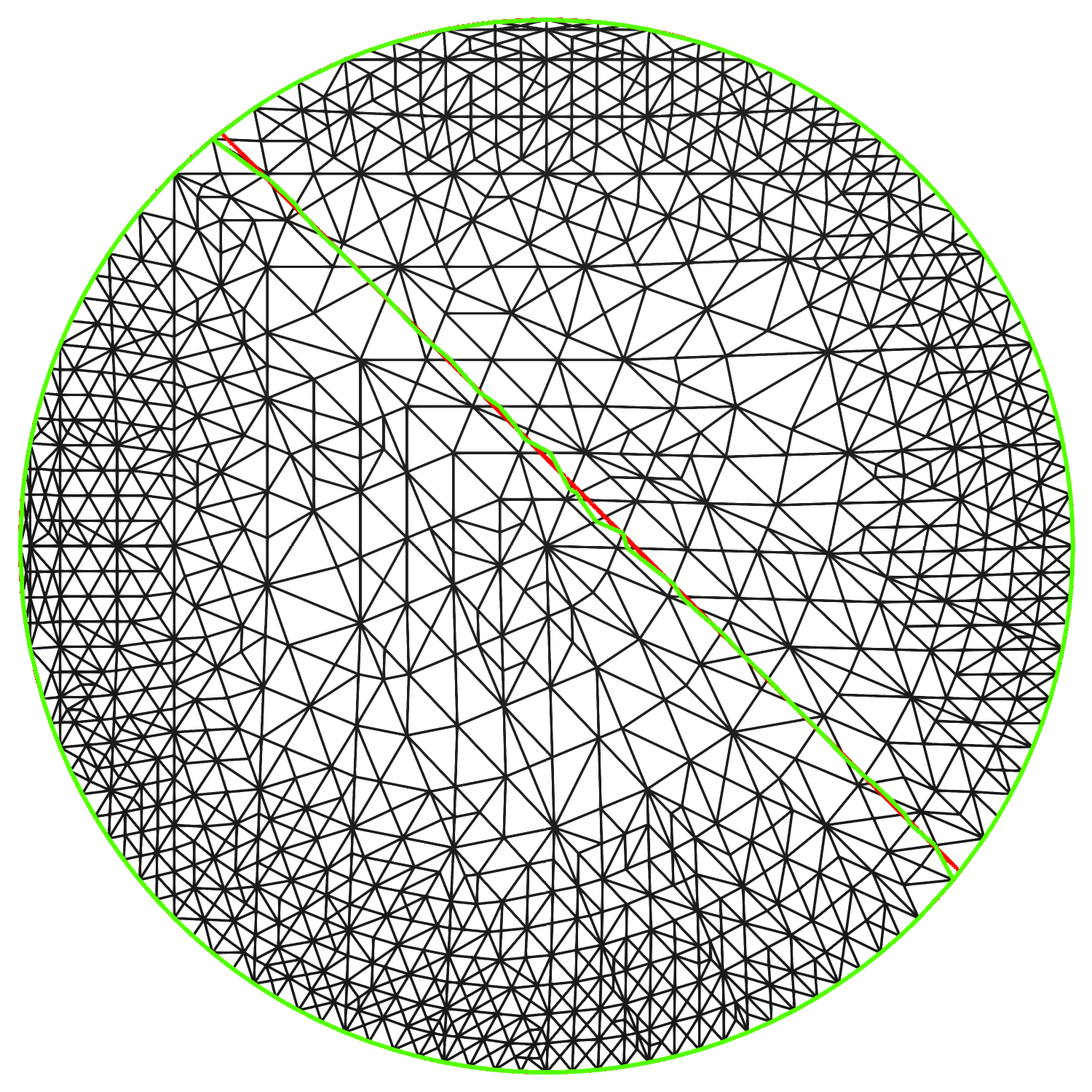}\\
\end{minipage}
\begin{minipage}[c]{0.27\textwidth}\centering
\includegraphics[trim={0 0 0 0},clip,width=4.4cm,height=4.4cm,scale=0.30]{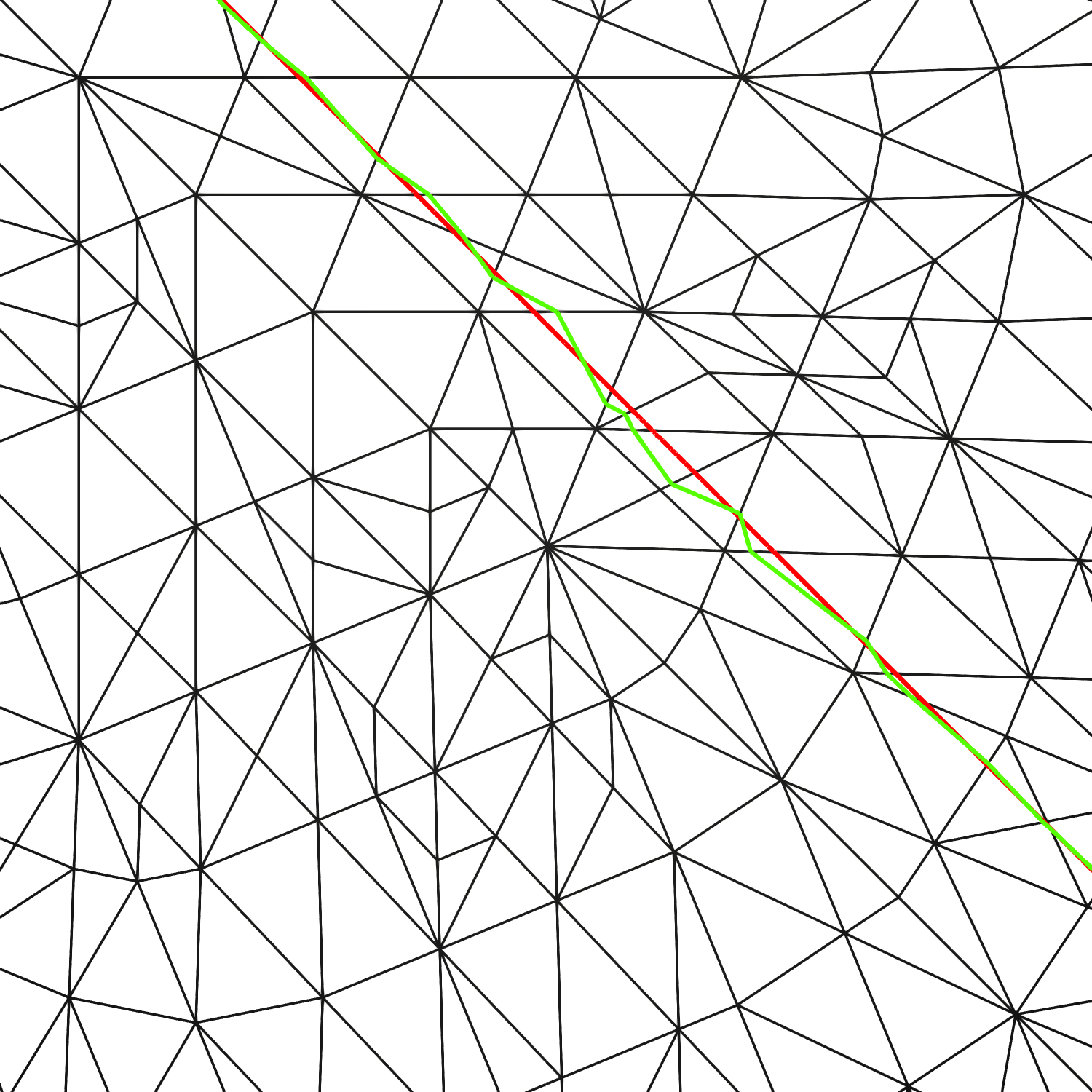}\\
\end{minipage}
\caption{Approximate control $\bar{\mathfrak{u}}_{\ell}$ and comparison of the continuous (red) and discrete switching sets on the adaptively refined meshes obtained after $5$ (upper row) and $10$ (lower row) iterations for the problem from section \ref{sec:ex_3} with $\gamma = 1$; in the red region $\bar{\mathfrak{u}}_{\ell} = 1$  whereas in the blue region $\bar{\mathfrak{u}}_{\ell} = -1$.}
\label{fig:ex_3_4}
\end{figure}


\appendix
\section{Results on the Green function}\label{sec:appendix}

A common strategy for performing error analysis for finite element approximations in the maximum norm is to represent the pointwise error by using a Green’s function.
For this reason, we include some results regarding such a function.

For each $x\in \Omega$, Green’s function $G(x, \xi) : \Omega \times \Omega \to \mathbb{R}$ is defined as the solution (in the sense of distributions) to
\begin{align}\label{def:Green_function}
    -\Delta_{\xi}G(x, \xi) = \delta(x - \xi) \quad \xi \in \Omega, \qquad G(x, \xi) = 0 \quad \xi\in \Omega.
\end{align}
Here $\delta(\cdot)$ is the $d$-dimensional Dirac $\delta$-distribution.
The following pointwise representation is derived from this definition:
\begin{align}\label{eq:Green_rep}
    w(x) = (\nabla G(x, \cdot), \nabla w)_{\Omega} \quad \text{ for any } w\in H_0^{1}(\Omega)\cap W^{1,q}(\Omega) \text{ with } q > d.
\end{align}

In what follows, we summarize some properties of Green’s function that are important for our analysis. 
For a proof, we refer to \cite[Proposition 4.1]{MR3878607} (see also \cite[Theorem 1]{MR3520007}).

\begin{theorem}[properties of $G$]\label{thm:prop_G}
Let $G$ be the Green function defined in \eqref{def:Green_function}.
If $\omega_\rho(x)$ denotes a ball of radius $\rho$ centered at $x\in \Omega$, then 
\begin{itemize}
    \item[(i)] $\nabla G\in L^{\frac{d}{d-1},\infty}(\Omega)$, which implies for every $\mathsf{p}\in [1,\nicefrac{d}{(d-1)})$ that
    \begin{align}\label{eq:DG_est_rho}
        \|\nabla G\|_{L^{\mathsf{p}}(\omega_\rho(x))} \lesssim \rho^{1-d+\frac{d}{\mathsf{p}}},
    \end{align}
    where the hidden constant depends on $\mathsf{p}$ and $d$ and blows up as $\mathsf{p}\uparrow \nicefrac{d}{(d-1)}$, and

    \item[(ii)] $G\in W^{2,1}(\Omega\setminus \omega_\rho(x))$ and satisfies
    \begin{align}\label{eq:D2G_est_rho}
        \|D^{2}G\|_{L^{1}(\Omega\setminus \omega_\rho(x))} \lesssim |\log \rho^{-1}|.
    \end{align}
\end{itemize}
\end{theorem}


\bibliographystyle{siam}
\bibliography{biblio}

\end{document}